\documentclass[12pt]{amsart}

\usepackage{amssymb}

\usepackage{enumerate}
\usepackage{amsfonts,amsmath,amsthm,pb-diagram,mathtools}

\usepackage{ifpdf}
\ifpdf
\usepackage[pdftex]{graphicx}
\usepackage[pdftex,dvipsnames,usenames]{color}
\else
\usepackage{graphicx}
\usepackage[dvips,dvipsnames,usenames]{color}
\fi

\makeatletter
\@namedef{subjclassname@2010}{%
\textup{2010} Mathematics Subject Classification}
\makeatother

\newtheorem{thm}{Theorem}[section]
\newtheorem{cor}[thm]{Corollary}
\newtheorem{lem}[thm]{Lemma}
\newtheorem{prop}[thm]{Proposition}

\theoremstyle{definition}
\newtheorem{defin}[thm]{Definition}
\newtheorem{rem}[thm]{Remark}
\newtheorem{exa}[thm]{Example}

\numberwithin{equation}{section}

\begin{document}

\baselineskip=17pt



\title[Minimal universal metric spaces]{Minimal universal metric spaces}

\author[V. Bilet]{V. Bilet}
\address{Division of Applied Problems in Contemporary Analysis,
Institute of Mathematics of NASU,
Tereshenkivska str.~3,
Kyiv 01601,
Ukraine}
\email{biletvictoriya@mail.ru}

\author[O. Dovgoshey]{O. Dovgoshey}
\address{Division of Applied Problems in Contemporary Analysis\\
Institute of Mathematics of NASU\\
Tereshenkivska str.~3\\
Kyiv 01601\\
Ukraine}
\email{aleksdov@mail.ru}
\thanks{The research of the second author was supported by a grant received from TUBITAK within 2221-Fellowship Programme for Visiting Scientists and Scientists on Sabbatical Leave and also as a part of EUMLS project with grant agreement PIRSES --- GA --- 2011 --- 295164.}

\author[M. K\"{u}\c{c}\"{u}kaslan]{M. K\"{u}\c{c}\"{u}kaslan}
\address{Mersin University, Faculty of Art and Sciences\\
Department of Mathematics\\
Mersin 33342,\\
Turkey}
\email{mkucukaslan@mersin.edu.tr}

\author[E. Petrov]{E. Petrov}
\address{Division of Applied Problems in Contemporary Analysis\\
Institute of Mathematics of NASU\\
Tereshenkivska str.~3\\
Kyiv 01601\\
Ukraine}
\email{eugeniy.petrov@gmail.com}



\date{}

\begin{abstract}
Let $\mathfrak{M}$ be a class of metric spaces. A metric space $Y$ is minimal $\mathfrak{M}$-universal if every $X\in\mathfrak{M}$ can be isometrically embedded in $Y$ but there are no proper subsets of $Y$ satisfying this property. We find conditions under which, for given metric space $X$, there is a class $\mathfrak{M}$ of metric spaces such that $X$ is minimal $\mathfrak{M}$-universal. We generalize the notion of minimal $\mathfrak{M}$-universal metric space to notion of minimal $\mathfrak{M}$-universal class of metric spaces and prove the uniqueness, up to an isomorphism, for these classes. The necessary and sufficient conditions under which the disjoint union of the metric spaces belonging to a class $\mathfrak{M}$ is minimal $\mathfrak{M}$-universal are found. Examples of minimal universal metric spaces  are constructed for the classes of the three--point metric spaces and $n$-dimensional normed spaces. Moreover minimal universal metric spaces are found for some subclasses of the class of metric spaces $X$ which possesses the following property. Among every three distinct points of $X$ there is one point lying between the other two points.
\end{abstract}

\maketitle

\tableofcontents

\section{Introduction}

Let $X$ and $Y$ be metric spaces. Recall that a function $f: X\to Y$ is an \emph{isometric embedding} if the equality $d_{X}(x,y)=d_{Y}(f(x), f(y))$ holds for all $x,y\in X.$ In what follows the notation $f: X\hookrightarrow Y$ means that $f$ is an isometric embedding of $X$ in $Y.$ We say that $X$ is isometrically embedded in $Y$ and write $X\hookrightarrow Y$ if there exists $f: X\hookrightarrow Y.$ The situation when $X\hookrightarrow Y$ does not hold will be denoted as $X\not\hookrightarrow Y.$

We find it useful to set that the empty metric space is isometrically embedded in every metric space $Y,$ $\varnothing\hookrightarrow Y.$ Moreover, $Y\hookrightarrow\varnothing$ holds if and only if $Y=\varnothing.$

\begin{defin}\label{univ}
Let $\mathfrak{M}$ be a class of metric spaces. A metric space $Y$ is said to be universal for $\mathfrak{M}$ or $\mathfrak{M}$-universal if $X\hookrightarrow Y$ holds for every $X\in\mathfrak{M}.$
\end{defin}

In what follows the expression $\mathfrak{M}\hookrightarrow Y$ signifies that $Y$ is a $\mathfrak{M}$-universal metric space.

Recall that a class $\mathfrak{A}$ is a set if and only if there is a class $\mathfrak{B}$ such that $\mathfrak{A}\in\mathfrak{B}.$ We shall use the capital Gothic letters $\mathfrak{A}, \mathfrak{B}, ...$ to denote classes of metric spaces, the capital Roman letters $A, B, ...$ to denote metric spaces and the small letters $a, b, ...$ for points of these spaces.

It is relevant to remark that throughout this paper we shall make no distinction in notation between a set $X$ and a metric space with a support $X.$ For example  $X\subseteq Y$ means that $X$ is a subset of a set $Y$ or that $X$ is a subspace of a metric space $Y$ with the metric induced from $Y.$

\begin{defin}\label{muniv}
Let $\mathfrak{M}$ be a class of metric spaces and let $Y$ be a $\mathfrak{M}$-universal metric space. The space $Y$ is minimal $\mathfrak{M}$-universal if the implication
\begin{equation}\label{muneqv}
(\mathfrak{M}\hookrightarrow Y_{0})\Rightarrow (Y_{0}=Y)
\end{equation}
holds for every subspace $Y_{0}$ of $Y.$
\end{defin}

This basic for us definition was previously used by W.~Holstynski in \cite{Ho2} and \cite{Ho3}. It should be noted here that some other ''natural'' definitions of minimal universal metric spaces can be introduced. For example instead of \eqref{muneqv} we can use the implication
\begin{equation}\label{0.1*}
(\mathfrak{M}\hookrightarrow Y_{0})\Rightarrow (Y\hookrightarrow Y_{0})
\end{equation}
or
\begin{equation}\label{0.1**}
(\mathfrak{M}\hookrightarrow Y_{0})\Rightarrow (Y_{0}\simeq Y)
\end{equation}
where $Y_{0}\simeq Y$ means that $Y$ and $Y_0$ are isometric. It is easy to show that \eqref{muneqv}, \eqref{0.1*} and \eqref{0.1**} lead to the three different concepts of minimal universal metric spaces. See, in particular, Proposition \ref{pr2.17} for an example of a family $\mathfrak{M}$ which have unique, up to isometry, minimal $\mathfrak{M}$--universal metric space if we use \eqref{0.1*} or \eqref{0.1**}, and which do not admit any minimal $\mathfrak{M}$--universal metric space in the sense of~\eqref{muneqv}. In the present paper we mainly consider the minimal universal metric spaces in the sense of Definition \ref{muniv}.

The structure of the paper can be described as following.

\smallskip

\noindent$\bullet$ The second section is a short survey of some results related to universal metric spaces.

\smallskip

\noindent$\bullet$ Sufficient conditions under which a metric space is minimal universal for a class of metric spaces are obtained in Section~\ref{sect3}. Moreover, there we find some conditions of non existence of minimal $\mathfrak{M}$-universal metric spaces for given $\mathfrak{M}$.

\smallskip

\noindent$\bullet$ The subspaces of minimal universal metric spaces are discussed briefly in Section~\ref{sect4}.

\smallskip

\noindent$\bullet$ In the fifth section we generalize the notion of minimal $\mathfrak{M}$-universal metric space to the notion of minimal $\mathfrak{M}$-universal class of metric spaces. It is proved that a minimal $\mathfrak{M}$-universal class, if it exists, is unique up to an isomorphism.

\smallskip

\noindent$\bullet$ In the sixth section we discuss when one can construct a minimal $\mathfrak{M}$-universal metric space using disjoint union of metric spaces belonging to $\mathfrak{M}$.

\smallskip

\noindent$\bullet$ Section~\ref{sect7} deals with metric spaces which are minimal universal for some subclasses of the class of metric spaces $X$ that possesses the following property. Among every three distinct points of $X$ there is one point lying between the other two points.

\smallskip

\noindent$\bullet$ Two simple examples of universal metric spaces which are minimal for the class of three--point metric spaces are given in Section~\ref{sect8}.

\section{A short survey of universal metric spaces}
\label{sect2}

Let us denote by $\mathfrak{S}$ the class of separable metric spaces. In 1910 M.~Fre\-chet \cite{Fr} proved that the space $l^{\infty}$ of bounded sequences of real numbers with the sup-norm is $\mathfrak{S}$-universal. This result admits a direct generalization to the class $\mathfrak{S}_{\tau}$ of metric spaces of weight at most $\tau$ with an arbitrary cardinal number $\tau.$ Indeed, in 1935, K. Kuratowski \cite{Kur} proved that every metric space $X$ is isometrically embedded in the space $L^{\infty}(X)$ of bounded real-valued functions on $X$ with sup-norm. The Kuratowski embedding $$X\ni x\mapsto f_{x}\in L^{\infty}(X)$$ was defined as $f_{x}(y)=d_{X}(x, y)-d_{X}(x_0, y),$ where $x_0$ is a marked point in $X.$ It is evident that for every dense subset $X_0$ of $X$ the equality
$$
\sup_{y\in X}|f_{x_1}(y)-f_{x_2}(y)| = \sup_{y\in X_{0}}|f_{x_1}(y)-f_{x_2}(y)|
$$
holds for all $x_1, x_2 \in X.$ Hence if $A$ is a set with $|A|=\tau,$ then $L^{\infty}(A)$ is $\mathfrak{S}_{\tau}$-universal.

In 1924, P. S. Urysohn \cite{Ur}, \cite{Ur1} was the first who gave an example of separable $\mathfrak{S}$-universal metric space (note that $l^{\infty}$ is not se\-parable). In 1986 M.~Katetov \cite{Ka} proposed an extension of Urysohn's construction to a $\mathfrak{S}_\tau$--universal metric space of the weight $\tau=\tau^{<\tau}>\omega$. Recently V. Uspenskiy~\cite{Usp1990}, \cite{Usp2002}, \cite{Usp2004}, followed by A.~Vershik \cite{Versh1998}, \cite{Versh02FMT}, \cite{Versh02DR} and later M.~Gromov, positioned the Urysohn space with the correspondence to several mathematical disciplines: Functional Analysis, Probability Theory, Dynamics, Combinatorics, Model Theory and General Topology. In 2009 D. Le\v{s}nik gives a ''constructive model'' of Urysohn space~\cite{Le}. A computable version of this space was given by Hiroyasu Kamo in 2005~\cite{HK}.

The graphic metric space of the Rado graph \cite{Rado} (the vertex-set consists of all prime numbers $p\equiv 1 (\mathrm{mod}\ 4)$ with $pq$ being an edge if $p$ is a quadratic residue modulo $q$) is a universal metric space for the class of at most countable metric spaces with distances $0$, $1$ and $2$ only.

The Banach-Mazur theorem \cite{Ba} asserts that the space $C[0,1]$ of continuous functions $f:[0,1]\to\mathbb R$ with the sup-norm is $\mathfrak{S}$-universal. This famous theorem has numerous interesting modifications. As an example we only mention that every separable Banach space is isometrically embedded in the subspace of $C[0,1]$ consisting of nowhere differentiable functions that was proved by L. Rodriguez-Piazza \cite{R-P} in 1995.

The following unexpected result was obtained by W. Holsztynski in 1978 \cite{Ho}. There exists a compatible with the usual topology metric on $\mathbb R$ such that $\mathbb R$ with this metric is universal for the class $\mathfrak{F}$ of finite metric spaces.
Some interesting examples of  minimal universal metric spaces for the classes $\mathfrak{F}_k$ of metric spaces $X$ with $|X|\le k,$ $k=2,3,4,$ can be found in \cite{Ho2}, \cite{Ho3}.

The class $\mathfrak{SU}$ of separable ultrametric spaces is another example of an important class of metric spaces for which the universal spaces are studied in some details. It was proved by A. F. Timan and I. A. Vestfrid in 1983 \cite{TV} that the space $l_2$ of real sequences $(x_n)_{n\in\mathbb N}$ with the norm $\Bigl(\sum_{n\in\mathbb N}x_{n}^{2}\Bigr)^{1/2}$ is $\mathfrak{SU}$-universal. The first example of ultrametric space which is universal for $\mathfrak{SU}$ was obtained by I. A. Vestfrid in 1994 \cite{Ve}. As was proved by Lemins in \cite{LL}, if an ultrametric space $Y$ is $\mathfrak{F}_2$-universal, then the weight of $Y$ is not less than the continuum $\mathfrak{c}.$
Consequently any $\mathfrak{SU}$-universal ultrametric space cannot be separable. In this connection it should be pointed out that, under some set-theoretic assumptions, for every cardinal $\tau> \mathfrak{c}$ there is an ultrametric space $LW_{\tau}\in\mathfrak{S}_{\tau}$ such that every ultrametric space from $\mathfrak{S}_{\tau}$ can be isometrically embedded into $LW_{\tau}$. The last statement was proved by J.~Vaughan in 1999~\cite{Va}.

There exist also results about spaces which are universal for some classes of compact metric spaces and separable Banach spaces (see, e.g., \cite{DL} and \cite{GK}). In particular, the necessary and sufficient conditions under which for given family of compact metric spaces there exists a compact universal metric space were found by S.~D.~Iliadis in 1995 \cite{Il1}. The same author considered also the existence of universal spaces in various subclasses of $\mathfrak{S}$~\cite{Il2}.

\section{Minimal universal metric spaces. Existence, nonexistence and uniqueness}
\label{sect3}

Let $X$ and $Y$ be metric spaces. Recall that an isometric embedding $f: X\hookrightarrow Y$ is an isometry if $f$ is a surjection. The spaces $X$ and $Y$ are isometric if there is an isometry $f: X\hookrightarrow Y.$  We shall write $X\simeq Y$ if $X$ and $Y$ are isometric. Otherwise, we use the notation $X\not\simeq Y.$

\begin{defin}\label{shemb}
A metric space $Y$ is shifted if there is $X\subseteq Y$ such that $X\simeq Y$ and $Y\setminus X\ne\varnothing.$ Otherwise, the space $Y$ is said to be not shifted.
\end{defin}
It follows directly from the definition, that a non--empty metric space $Y$ is not shifted if and only if every self embedding $f: Y\hookrightarrow Y$ is an isometry.

\begin{exa}\label{ex1*}
All finite metric spaces are not shifted. In particular, the empty metric space is not shifted.
\end{exa}

\begin{prop}\label{vspompr}
Let $Y$ be a metric space. The following conditions are equivalent.
\begin{enumerate}
\item[\rm(i)]\textit{Y is not shifted.}

\item[\rm(ii)]\textit{Y is minimal universal for the class consisting of the unique metric space $Y.$}

\item[\rm(iii)]\textit{There exists a class $\mathfrak{M}$ of metric spaces such that $Y$ is minimal $\mathfrak{M}$-universal.}
\end{enumerate}
\end{prop}
\begin{proof}
It follows directly from that definitions, that $\textrm{(i)}\Leftrightarrow\textrm{(ii)}$ holds. The implication $\textrm{(ii)}\Rightarrow \textrm{(iii)}$ is evident. Let us prove $\textrm{(iii)}\Rightarrow\textrm{(i)}.$ Suppose there is a class $\mathfrak{M}$ of metric spaces such that $Y$ is minimal $\mathfrak{M}$-universal. We must show that
\begin{equation}\label{eq*}
f(Y)=Y
\end{equation}
holds for every $f: Y\hookrightarrow Y.$ Let us consider an arbitrary $f: Y\hookrightarrow Y.$ Since $Y$ is $\mathfrak{M}$-universal, the space $f(Y)$ is also $\mathfrak{M}$-universal. Now using \eqref{muneqv} with $Y_{0}=f(Y),$ we obtain \eqref{eq*}.
\end{proof}

Proposition~\ref{vspompr} implies that every minimal universal metric space is not shifted. Simple examples show that there is a family $\mathfrak{M}$ of metric spaces and a not shifted metric space $Y$ such that $\mathfrak{M}\hookrightarrow Y$ and $Y$ is not minimal $\mathfrak{M}$-universal. The situation is quite different if, for not shifted $Y$,  we presuppose $\mathfrak{M}\hookrightarrow Y$ and $Y\in\mathfrak{M}$.

\begin{thm}\label{mainth1}
Let $\mathfrak{Y}$ be a class of metric spaces, $Y\in\mathfrak{Y}$ and let $\mathfrak{Y}\hookrightarrow Y.$ Then the following conditions are equivalent.
\begin{enumerate}
\item[\rm(i)]\textit{Y is not shifted.}

\item[\rm(ii)]\textit{Y is minimal $\mathfrak{Y}$-universal.}

\item[\rm(iii)]\textit{There exists a minimal $\mathfrak{Y}$-universal $X\in\mathfrak{Y}.$}

\item[\rm(iv)]\textit{There exists a minimal $\mathfrak{Y}$-universal metric space.}
\end{enumerate}
\end{thm}
\begin{proof}
Suppose that, on the contrary, $\textrm{(i)}$ holds but $Y$ is not minimal $\mathfrak{Y}$-universal. Then, by Definition~\ref{muniv}, there is $y_0\in Y$ such that $\mathfrak{Y}\hookrightarrow Y\setminus\{y_0\}$. Since $Y\in\mathfrak{Y},$ there exists $f: Y\hookrightarrow Y\setminus\{y_0\}.$ Let $in: Y\setminus\{y_0\}\to Y$ be the standard injection, $in(y)=y$ for every $y\in Y\setminus\{y_0\}.$ Then the isometric embedding is
$$
\begin{array}{cccc}
Y & \xrightarrow{\ \ \mbox{\emph{f}}\ \ } &
Y\setminus\{y_0\}\xrightarrow {\ \ \mbox{\emph{in}}\ \ } & Y
\end{array}
$$
not an isometry, contrary to $\textrm{(i)}.$ The implication $\textrm{(i)}\Rightarrow \textrm{(ii)}$ follows. The implications $\textrm{(ii)}\Rightarrow \textrm{(iii)}$ and $\textrm{(iii)}\Rightarrow \textrm{(iv)}$ are trivial. Furthermore, Proposition~\ref{vspompr} implies $\textrm{(ii)}\Rightarrow \textrm{(i)}$. To complete the proof, it suffices to show that $\textrm{(iv)}\Rightarrow \textrm{(ii)}$ holds. Let $X$ be a minimal $\mathfrak{Y}$-universal metric space. We claim that $Y$ is also minimal $\mathfrak{Y}$-universal. Indeed, since $Y\in\mathfrak{Y}$ and $\mathfrak{Y}\hookrightarrow X,$ there is $X_{0}\subseteq X$ such that $X_{0}\simeq Y.$ Now from $\mathfrak{Y}\hookrightarrow Y$ and $X_{0}\simeq Y$ it follows that $\mathfrak{Y}\hookrightarrow X_{0}.$ Since $X$ is minimal $\mathfrak{Y}$-universal, the statements $\mathfrak{Y}\hookrightarrow X_{0}$ and $X_{0}\subseteq X$ imply $X_{0}=X.$ Hence, $X_{0}$ is minimal $\mathfrak{Y}$-universal. Since $X_{0}\simeq Y$, the metric space $Y$ is also minimal $\mathfrak{Y}$-universal.
\end{proof}

Analyzing the proof of Theorem~\ref{mainth1}, we obtain the following.

\begin{prop} \label{prisom}
Let $\mathfrak{M}$ be a class of metric spaces and let $X$ and $Y$ be metric spaces such that $\mathfrak{M}\hookrightarrow X$ and $\mathfrak{M}\hookrightarrow Y.$ If $Y$ is minimal $\mathfrak{M}$-universal and $X\in\mathfrak{M},$ then $X\simeq Y$.
\end{prop}

\begin{cor}\label{cor2}
\textbf{\emph{(}Isometry of minimal universal metric spaces\emph{)}.}
Let  $\mathfrak{Y}$ be a class of metric spaces. If there exists a $\mathfrak{Y}$-universal space $Y\in\mathfrak{Y},$ then every two minimal $\mathfrak{Y}$-universal metric spaces, if they exist, are isometric.
\end{cor}

If a class $\mathfrak{Y}$ of metric spaces does not contain any $\mathfrak{Y}$-universal metric space, then we may generally have two minimal $\mathfrak{Y}$-universal metric spaces $W$ and $Z$ such that $W\not\simeq Z$.

\begin{exa}\label{ex2}
Let $\mathfrak{Y}=\{X, Y\}$ with
$$
X=\{x_1, x_2\},\, d_{X}(x_1, x_2)=1\,\, \mbox{and} \,\,Y=\{y_1, y_2\},\, d_{Y}(y_1, y_2)=2
$$
and let $1<a<b<2$. The metric spaces
$$
Z=\{z_1, z_2, z_3\}, \quad W=\{w_1, w_2, w_3\}
$$
with
$$
d_{Z}(z_1, z_2)=d_{W}(w_1, w_2)=1, \, d_{Z}(z_2, z_3)=d_{W}(w_2, w_3)=2,
$$
$$
d_{Z}(z_1, z_3)=a, \, d_{W}(w_1, w_3)=b
$$
are minimal $\mathfrak{Y}$-universal. It is clear that $Z\not\simeq W.$
\end{exa}

In the fifth section of the paper we shall show that the nonuniqueness of minimal $\mathfrak{M}$-universal metric spaces can be overcome by an extending of the concept of a minimal $\mathfrak{M}$-universal metric space to the concept of minimal $\mathfrak{M}$-universal class of metric spaces (see Definition~\ref{guniv*} and Theorem~\ref{L3.main}).

\begin{prop} \emph{\textbf{(}}\textbf{Nonexistence of minimal universal spaces}\emph{\textbf{)}}.\label{cor1}
Let  $\mathfrak{M}$ be a class of non--empty metric spaces. The class $\mathfrak{M}$ admits no minimal universal metric spaces if at least one from the following conditions holds.
\begin{enumerate}
\item[\rm(i)] There are $X, Y\in\mathfrak{M}$ such that
$$
\mathfrak{M}\hookrightarrow X, \, \mathfrak{M}\hookrightarrow Y \, \mbox{and}\,\, X\not\simeq Y.
$$

\item[\rm(ii)] For every $X\in\mathfrak{M}$ there exist metric spaces $Y, Y_{1}$ and $Y_{2}$ which satisfy the conditions
\begin{equation}\label{equat1}
Y\in\mathfrak{M}, \quad Y_{1}\subseteq Y,\quad Y_{2}\subseteq Y, \quad Y_{1}\cap Y_{2}=\varnothing
\end{equation}
and
\begin{equation}\label{equat2}
Y_{1}\simeq X\simeq Y_{2}.
\end{equation}
\end{enumerate}
\end{prop}

\begin{proof}
It is an immediate consequence of Proposition~\ref{prisom} that (i) implies the nonexistence of minimal $\mathfrak{M}$-universal metric spaces.  Let us consider (ii).

Suppose that $W$ is a minimal $\mathfrak{M}$-universal metric space and $\textrm{(ii)}$ holds. Let $w_{0}\in W$ and $$W_{0}=W\setminus\{w_{0}\}.$$
Since $W$ is minimal $\mathfrak{M}$-universal, there is $X\in\mathfrak{M}$ such that $X\not\hookrightarrow W_{0}$. Using $\textrm{(ii)}$ we can find metric spaces $Y, Y_{1}$ and $Y_{2}$ which satisfy \eqref{equat1} and \eqref{equat2}. From $Y\hookrightarrow W$ and \eqref{equat1} it follows that there exist $W_{1}\subseteq W$ and $W_{2}\subseteq W$ such that
\begin{equation}\label{equat3}
W_{1}\simeq Y_{1}, \quad W_{2}\simeq Y_{2} \quad\mbox{and}\quad W_{1}\cap W_{2}=\varnothing.
\end{equation}
The last equality implies $W_{1}\cap\{w_{0}\}=\varnothing$ or $W_{2}\cap\{w_{0}\}=\varnothing.$ We may assume, without loss of generality, that $W_{1}\cap\{w_{0}\}=\varnothing.$ Hence, $W_{1}\subseteq W_{0}$ holds. It follows from \eqref{equat2} and \eqref{equat3} that $X\simeq W_{1}.$ Consequently, we have $X\hookrightarrow W_{0},$ which is a contradiction.
\end{proof}

\begin{rem}
Proposition~\ref{cor1} remains valid if we suppose that there is non--empty $Y\in\mathfrak{M}$ instead of $X\ne\varnothing$ for all $X\in\mathfrak{M}.$
If $\mathfrak{M}=\{\varnothing\},$ then $\mathfrak{M}$ satisfies condition $\textrm{(ii)}$ of Proposition~\ref{cor1} and $X = \varnothing$ is minimal $\mathfrak{M}$-universal.
\end{rem}

\begin{exa}\label{exampl}
Let $\mathfrak{F}$ be the class of finite non--empty metric spaces. Condition (ii) of Proposition~\ref{cor1} is valid for $\mathfrak{M}=\mathfrak{F}.$ Consequently, by Proposition~\ref{cor1}, the class $\mathfrak{F}$ does not admit any minimal $\mathfrak{F}$-universal metric spaces. In particular, the Holsztynski metric space (see \cite{Ho}) is $\mathfrak{F}$-universal, but not minimal $\mathfrak{F}$-universal.
\end{exa}

In Example~\ref{exampl} instead of $\mathfrak{F}$ we can take the class of finite metric subspaces of the usual real line $\mathbb R.$ Some other examples of families $\mathfrak{M}$, which do not admit minimal $\mathfrak{M}$-universal metric spaces will be given after the corresponding lemmas.

Recall that a metric space $X$ is ultrametric if the strong triangle ine\-quality $$d_{X}(x, y)\le\max\{d_{X}(x, z), d_{X}(z, y)\}$$ holds for all $x, y, z\in X$.

\begin{lem}\label{Ultra}
Let $X$ and $Y$ be disjoint ultrametric spaces, let $x_{0}\in X$ and $y_{0}\in Y,$ and let $r_0$ be a positive real number. Then there is an ultrametric $d_{Z}$ on $Z=X \cup Y$ such that
\begin{equation}\label{equat4}
d_{Z}(x, y)=\begin{cases}
         d_{X}(x, y) & \mbox{if} \,$ $ x, y\in X\\
         d_{Y}(x, y) & \mbox{if} \,$ $ x, y\in Y\\
         r_{0} & \mbox{if} \,$ $ x=x_0 \,\, \mbox{and}\,\, y=y_0.\\
\end{cases}
\end{equation}
\end{lem}
\begin{proof}
It is a particular case of Theorem~2 from \cite{DovP}.
\end{proof}

\begin{lem}
Let $r_0\in (0, \infty)$ and let $X_1$ and $X_2$ be disjoint non--empty metric spaces with
$$
r_0 \ge\max\{\emph{diam} X_{1}, \emph{diam} X_{2}\}>0,
$$
where
$$
\emph{diam} X_{i}=\sup\{d_{X_i}(x, y): x, y\in X_{i}\},\,i=1, 2.
$$
Then the function
\begin{equation}\label{equat5}
d_{Z}(x, y)=\begin{cases}
         d_{X_1}(x, y) & \mbox{if} \,$ $ x, y\in X_1\\
         d_{X_2}(x, y) & \mbox{if} \,$ $ x, y\in X_2\\
         r_{0} & \mbox{if} \,$ $ x\in X_1, \,\, y\in X_2 \text{ or }x\in X_2, \  y\in X_1,\\
\end{cases}
\end{equation}
is a metric on $Z= X_1\cup X_2.$
\end{lem}

A proof is simple, so that we omit it here.

\begin{exa}\label{ex2.13}
Let $\tau$ be an infinite cardinal number and $\mathfrak{S}_{\tau}\mathfrak{U}$ be the class of ultrametric spaces of weight at most $\tau.$ If
$$
X, Y\in \mathfrak{S}_{\tau}\mathfrak{U},\ X\cap Y = \varnothing, \ x_0\in X,\ y_0\in Y,\  r_0 \in (0,\infty), \ Z=X\cup Y
$$
and $d_{Z}$ satisfies \eqref{equat4}, then the weight of $Z$ is at most $\tau.$ Hence, condition $\textrm{(ii)}$ of Proposition~\ref{cor1} holds with $\mathfrak{M}=\mathfrak{S}_{\tau}\mathfrak{U}$. Consequently, $\mathfrak{S}_{\tau}\mathfrak{U}$ admits no minimal $\mathfrak{S}_{\tau}\mathfrak{U}$-universal metric spaces. In particular, the ultrametric space $LW_{\tau}$, which was considered by A.~Lemin and V.~Lemin in \cite{LL}, is $\mathfrak{S}_{\tau}\mathfrak{U}$-universal, but not minimal $\mathfrak{S}_{\tau}\mathfrak{U}$-universal.
\end{exa}

The next example shows that the class $\mathfrak{S}$ of separable metric spaces does not admit any minimal universal spaces.

\begin{exa}\label{ex1}
Let $\overline{C}[0, 1]$ be the metric subspace of $C[0, 1]$ consisting of nowhere differentiable functions.
Since $C[0, 1]$ is complete and $\overline{C}[0, 1]$ is incomplete, we have $C[0, 1]\not\simeq\overline{C}[0, 1].$ Using the Banach-Mazur theorem, the theorem of Rodriegues-Piazza and Proposition~\ref{cor1} with $\mathfrak{M}=\mathfrak{S}$, we see that there are no minimal $\mathfrak{S}$-universal metric spaces.
\end{exa}

\begin{exa}
Let $\mathfrak{D}_3$ be the class of all at most countable metric spaces with the distance sets in $\{0, 1, 2\}$. Using condition $\textrm{(ii)}$ of Propositon~\ref{cor1} and defining $d_{Z}$ by \eqref{equat5} with $r_{0}=2$ for disjoint $X, Y\in \mathfrak{D}_3,$ we see that $\mathfrak{D}_3$ admits no minimal universal metric spaces. Consequently, the graphic metric space of the Rado graph \cite{Rado} is $\mathfrak{D}_3$-universal, but not minimal $\mathfrak{D}_3$-universal.
\end{exa}

Recall that a metric $d_X$ on $X$ is discrete if the equality $d_X(x,y)=1$ holds for all distinct $x, y \in X$.

\begin{exa}\label{exnew}
Let $\tau$ be an infinite cardinal number. Write $\mathfrak{M}\mathfrak{D}^{\tau}$ for the class of metric spaces $X$ with $|X|\le\tau$ and discrete $d_{X}$. It is clear that:
\smallskip

$\bullet$ $(X\simeq Y)\Leftrightarrow (|X|=|Y|)$ holds for all $X, Y\in\mathfrak{M}\mathfrak{D}^{\tau}$;
\smallskip

$\bullet$ A metric space $X\in\mathfrak{M}\mathfrak{D}^{\tau}$ is $\mathfrak{M}\mathfrak{D}^{\tau}$-universal if and only if $|X| = \tau$;
\smallskip

$\bullet$ There are no minimal $\mathfrak{M}\mathfrak{D}^{\tau}$-universal metric spaces.
\end{exa}


\begin{prop}\label{pr2.17}
If $Y\in\mathfrak{M}\mathfrak{D}^{\tau},$ then the implication
\begin{equation}\label{equat6}
(\mathfrak{M}\mathfrak{D}^{\tau}\hookrightarrow Y_0)\Rightarrow (Y_0\simeq Y)
\end{equation}
holds for every subspace $Y_0$ of $Y$. Moreover if $Y$ is an arbitrary metric space and \eqref{equat6} holds for every $Y_0\subseteq Y$, then $Y \in \mathfrak{M}\mathfrak{D}^{\tau}$ and $|Y|=\tau$.
\end{prop}
A simple proof is missed here.




\begin{rem}\label{r2.20}
As in Example \ref{ex2.13} we can show that the class $\mathfrak{S}_\tau$ with $\tau>\omega$ does not admit any minimal universal metric spaces. Thus the Katetov space is not minimal $\mathfrak{S}_\tau$--universal.
\end{rem}

We now turn to some conditions under which the metric spaces are not shifted. The following lemma seems to be known up to the terminology.

\begin{lem}\label{compsh}
Every compact metric space is not shifted.
\end{lem}
\begin{proof}
Let $X$ be a compact metric space and let $f: X\hookrightarrow X$ be an isometric self embedding. Suppose that a point $p$ belongs to $X\setminus
f(X).$ Define a sequence $(p_i)_{i \in\mathbb N}$ in $X$ as
\begin{equation}\label{point}
p_1 = p, \, p_2=f(p_1), \, p_{3}=f(p_2), ...
\end{equation}
and so on.
Since $f$ is continuous, $f(X)$ is a compact subset of $X$. Hence the number
$$
\varepsilon_{0}:=\inf_{y\in f(X)}d_{X}(p_1, y)
$$
is strictly positive, $\varepsilon_{0}>0.$ It follows from \eqref{point}, that $p_{i}\in f(X)$ for every $i\ge 2.$ Consequently, the inequality $d_{X}(p_1, p_i)\ge\varepsilon_{0}$ holds for every $i\ge 2$. The function $f$ preserves the distances. Hence, for $i> j,$ we have
$$
d_{X}(p_j, p_i)=d_{X}(f(p_{j-1}), f(p_{i-1}))$$ $$=d_{X}(p_{j-1}, p_{i-1})=... =d_{X}(p_1,
p_{i-(j-1)})\ge\varepsilon_{0}.
$$
Thus, $d_{X}(p_{i}, p_{j})\ge\varepsilon_{0}>0$ if $i\ne j.$ In particular, the sequence $(p_i)_{i\in\mathbb N}$ has not any convergent subsequence, contrary to the compactness of $X$. This contradiction implies $f(X)=X.$ By Definition~\ref{shemb}, $X$ is not shifted.
\end{proof}

Lemma~\ref{compsh} and Theorem~\ref{mainth1} give us the following result (cf. Proposition \ref{pr2.17}).
\begin{thm}\label{mainth2}
Let $\mathfrak{Y}$ be a class of compact metric spaces. Then the
following conditions are equivalent.
\begin{enumerate}
\item[\rm(i)]\textit{There is a minimal $\mathfrak{Y}$-universal metric space belonging to $\mathfrak{Y}$.}

\item[\rm(ii)]\textit{There is a $\mathfrak{Y}$-universal metric space belonging to $\mathfrak{Y}$.}
\end{enumerate}
\end{thm}

In Proposition \ref{Prop6.16} of the paper we construct a family $\mathfrak{M}$ of metric spaces such that:

\smallskip

\noindent $\bullet$ All $X\in\mathfrak{M}$ are compact;

\smallskip

\noindent $\bullet$ There is a compact  $\mathfrak{M}$-universal metric space;

\smallskip

\noindent $\bullet$ There is at least one minimal $\mathfrak{M}$-universal metric spaces;

\smallskip

\noindent $\bullet$ If $X$ and $Y$ are minimal $\mathfrak{M}$-universal, then $X$ and $Y$ are not compact and $X\simeq Y$.

\smallskip

Recall that a metric space $X$ is boundedly compact if all closed bounded subspaces of $X$ are compact. We can prove an analogue of Lemma~\ref{compsh} for some boundedly compact metric spaces.

\begin{defin}\label{d2.20}
A metric space $X$ is homogeneous if, for every pair $x, y\in X,$ there is $f: X\hookrightarrow X$ such that $f(x)=y.$
\end{defin}

\begin{rem}\label{r2.20*}
Definition~\ref{d2.20} is slightly non--standard. Usually we say that a metric space $X$ is homogenous if its isometry group $\operatorname{Iso}(X)$ acts transitively on points. The following lemma shows, in particular, that the standard definition and Definition~\ref{d2.20} are equivalent for boundedly compact $X$.
\end{rem}

\begin{lem}\label{homsh}
All boundedly compact homogeneous metric spaces are not shifted.
\end{lem}

\begin{proof}
Let $X$ be a boundedly compact homogeneous metric space. We must show that $f(X)=X$ for every $f: X\hookrightarrow X.$ Let $p_{0}\in X$. Write $p_{1}=f(p_0).$ Since $X$ is homogeneous there is  $g: X\hookrightarrow X$ such that $g(p_1)=p_0.$ Since $g$ is an injective function, the equality $f(X)=X$ holds if
\begin{equation*}\label{eqvv}
g(f(X))=X.
\end{equation*}
It is easy to prove that the last equality holds if
\begin{equation}\label{eqbal}
g(f(B_{r}(p_0)))=B_{r}(p_0)
\end{equation}
for every ball $B_{r}(p_0)=\{x: d_{X}(p_0, x)\le r\}.$ The function $g\circ f$ is an isometric self embedding of $X.$ Since $p_0$ is a fixed point of $g\circ f,$ the restriction $g\circ f |_{B_{r}(p_0)}$ is an isometric self embedding of the compact metric space $B_{r}(p_0).$ Now \eqref{eqbal} follows from Lemma~\ref{compsh}.
\end{proof}

Just as in Lemma \ref{homsh} we can prove the following proposition.

\begin{prop}\label{2.24*}
Let $X$ be a boundedly compact metric space. If for every $f: X\hookrightarrow X$ there is a fixed point $x_0 \in X$, $f(x_0)=x_0$, then $X$ is not shifted.
\end{prop}

\begin{rem}\label{rem2.24}
If $X$ is a finite-dimensional normed linear space over the field $\mathbb R$ or the field $\mathbb C$, then $X$ is not shifted because such spaces are homogeneous and boundedly compact. This fact is well known. In particular in the high school geometry, the isometries of the plane $\mathbb R^2$ usually defined as mappings $f\colon \mathbb R^2 \to \mathbb R^2$ which preserve distances. As was noted by Serge Lang in~\cite[p. 31]{SL}: "It is not immediately clear that an isometry has an inverse". Next in the book~\cite{SL} it is indicated that the existence of the inverses follows from the representation of any isometry by a composition of reflections through some straight lines in $\mathbb R^2$. Thus Lemma \ref{homsh} can be considered as a generalization of this elementary fact.
\end{rem}



The hyperbolic plane $\mathbb H$ gives an important example of boundedly compact homogenous metric space. The key point here that in the upper half-plane model $\mathbb H$, the isometries of $\mathbb H$  can be identified with elements of the subgroup  $\operatorname{M\ddot{o}b}(\mathbb H)$ of the general M\"{o}bius group and $\operatorname{M\ddot{o}b}(\mathbb H)$ acts transitively on $\mathbb H$ (see~\cite{JWA} for details).

Using Lemma~\ref{homsh}, Theorem~\ref{mainth1}, Corollary~\ref{cor2} and Proposition~\ref{prisom} we obtain the following.
\begin{thm}\label{mainth3*}
Let $\mathfrak{M}$ be a class of metric spaces. If $X\in\mathfrak{M}$ is boundedly compact, homogeneous and $\mathfrak{M}$-universal, then $X$ is minimal $\mathfrak{M}$-universal and $Y\simeq X$ holds for all minimal $\mathfrak{M}$-universal $Y$ and all $\mathfrak{M}$-universal $Y\in\mathfrak{M}.$
\end{thm}

\begin{prop}\label{pr3}
Let $\mathfrak{M}$ be a class of metric spaces and let $X$ be a homogeneous metric space. If $X$ is minimal $\mathfrak{M}$-universal, then there is $Y\in\mathfrak{M}$ such that $X\simeq Y$.
\end{prop}
\begin{proof}
Suppose $X$ is $\mathfrak{M}$-universal and $X\not\simeq Y$ holds for every $Y\in\mathfrak{M}.$ We must show that $X$ is not minimal $\mathfrak M$-universal. By the supposition, for every $Y\in\mathfrak{M},$ there are $f: Y\hookrightarrow X$ and $x_{f}\in X$ such that $$f(Y)\subseteq X\setminus\{x_f\}.$$
Let $a$ be a point of $X.$ Since $X$ is homogeneous, there is $g: X \hookrightarrow X$ such that $g(x_f)=a.$ The function
$$
\begin{array}{cccc}
Y & \xrightarrow{\ \ \mbox{\emph{f}}\ \ } &
X\xrightarrow {\ \ \mbox{\emph{g}}\ \ } & X
\end{array}
$$
is an isometric embedding of $Y$ in $X$ for which
$$
g(f(X))\subseteq X\setminus\{a\}.
$$
Consequently, we have $\mathfrak{M}\hookrightarrow X\setminus\{a\}$. Thus $X$ is not minimal $\mathfrak{M}$-universal, that is a contradiction.
\end{proof}

\begin{cor}
Let $\mathfrak{M}$ be a class of metric spaces and let $X$ be a homogeneous metric space.
If $X$ is minimal $\mathfrak{M}$-universal, then $X\simeq Y$ holds for every minimal $\mathfrak{M}$-universal metric space $Y.$
\end{cor}

The so--called strongly rigid metric spaces give us another subclass of not shifted metric spaces.

\begin{defin}(\cite{LJ})\label{sr}
A metric space $X$ is said to be strongly rigid if, for all $x, y, u, v\in X,$
$$d_{X}(x, y)=d_{X}(u, v) \quad\mbox{and}\quad x\ne y$$
imply that $\{x, y\}=\{u, v\}.$
\end{defin}

\begin{lem}\label{l:sr}
Every strongly rigid metric space is not shifted.
\end{lem}

\begin{proof}
For every $f: X\hookrightarrow X$ the equality
$$
d_{X}(x, y)=d_{X}(f(x), f(y))
$$
holds for all $x, y\in X$. Let $X$ be a strongly rigid metric space. From Definition~\ref{sr} it follows that
$$
\{x\}\subseteq\{x, y\}=\{f(x), f(y)\}\subseteq f(X)
$$
hold for all distinct $x, y\in X.$ It follows that $x\in f(X)$ holds for every $x\in X$ and every $f: X\hookrightarrow X$. Hence, every self embedding $f: X\hookrightarrow X$ is a surjection and, therefore, is an isometry.
\end{proof}

\begin{defin}\label{d:esr}
A metrizable space $X$ is said to be eventually strongly rigid if there is
a compatible with the topology of $X$ metric $d_X$ such that $(X, d_{X})$ is strongly rigid. (See \cite{LJ}, \cite{HWM}).
\end{defin}

\begin{lem}(\cite{HWM})\label{l:esr}
Let $X$ be a non--empty metrizable space with covering dimension zero, $\emph{dim}(X)=0.$ If $|X|\le \mathfrak{c},$ where $\mathfrak{c}$ is the cardinality of continuum, then $X$ is eventually strongly rigid.
\end{lem}

In analogy with Definition~\ref{d:esr}, we shall say that a metrizable space $X$ is eventually not shifted if there is a compatible with the topology of $X$ metric $d_{X}$ such that $(X, d_{X})$ is not shifted.

Lemma~\ref{l:sr} and Lemma~\ref{l:esr} imply the following proposition.

\begin{prop}\label{pr2.35}
Let $X$ be a non--empty metrizable space with $\emph{dim}(X)=0.$ If $|X|\le \mathfrak{c},$ then $X$ is eventually not shifted.
\end{prop}

\section{Not shifted metric subspaces of minimal universal metric spaces}
\label{sect4}

For a metric space $X$ denote by $\operatorname{Iso}(X)$ the group of all isometries of $X$. In what follows we do not presuppose that $\operatorname{Iso}(X)$ is equipped with any topology.

\begin{thm}\label{t2.1*}
Let $X$ be a not shifted and complete metric space. Then the following statements hold.
\begin{itemize}
  \item [(i)] If for every $g\in \operatorname{Iso}(X)$ and every $x\in X$ there is a positive integer number $m$ such that the equality
      \begin{equation}\label{e3.1*}
      g^m(x)=x
      \end{equation}
      holds, then every dense subset of $X$ is not shifted.
  \item [(ii)] If $X$ contains no isolated points and all dense subsets of $X$ are not shifted, then for every $g\in \operatorname{Iso}(X)$ and every $x\in X$, there is a positive integer number $m$ such that~(\ref{e3.1*}) holds.
\end{itemize}
\end{thm}
\begin{proof}
(i). Suppose that the condition of statement (i) holds. Let $Y$ be a dense subset of $X$ and let $f$ be an isometric self embedding of $Y$. We must show that $f\in \operatorname{Iso}(Y)$.

If $Z$, $S$ and $W$ are metric spaces,  $S$ is complete and $Z$ is a dense subspace of $W$, then for every uniformly continuous $\psi : Z\to S$ there is a unique continuous $g : W\to S$ such that $g|_{Z}=\psi$ (see, for example,~\cite[p.176]{Se}). The metric space $Y$ is a dense subspace of $X$ and the metric space $X$ is complete and, moreover, the isometric embedding
$$
Y\xrightarrow{f}Y\xrightarrow{in}X
$$
is, evidently, uniformly continuous. Hence there is a continuous mapping $g$ such that the diagram
\begin{equation}\label{3.2*}
\begin{diagram}
\node{Y}
      \arrow{e,t}{f} \arrow{s,l}{in}
   \node{Y} \arrow{e,t}{in}
      \node{X}
\\
\node{X}\arrow{ene,b}{g}
\end{diagram}
\end{equation}
is commutative. Let us show that $g$ is an isometry. Indeed if $x,y \in X$, then there are sequences $(x_n)_{n\in \mathbb N}$ and $(y_n)_{n\in \mathbb N}$ such that $x_n, y_n\in Y$ for every $n\in \mathbb N$ and
\begin{equation}\label{eq3.3}
x = \lim\limits_{n\to \infty}x_n \,\text{ and } \, y = \lim\limits_{n\to \infty}y_n.
\end{equation}
The function $f$ is an isometric self embedding and the function $g$ is continuous. Consequently, using \eqref{eq3.3} and the commutativity of \eqref{3.2*}, we obtain
\begin{equation*}
d_{X}(g(x), g(y))=\lim_{n\to\infty}d_{X}(g(x_n), g(y_n))=\lim_{n\to\infty}d_{X}(f(x_n), f(y_n))
\end{equation*}
\begin{equation*}
=\lim_{n\to\infty}d_{Y}(f(x_n), f(y_n))=\lim_{n\to\infty}d_{Y}(x_n, y_n)=\lim_{n\to\infty}d_{X}(x_n, y_n)=d_{X}(x, y).
\end{equation*}
Thus $g$ is an isometric self embedding. It follows that $g \in \operatorname{Iso}(X)$ because $X$ is not shifted. Suppose now that $f\notin \operatorname{Iso}(Y)$. Then the set $Y \setminus f(Y)$ is not empty. Let us consider an arbitrary point $x\in Y \setminus f(Y)$. Let $m$ be a positive integer number such that
\begin{equation}\label{3.3*}
    g^m(x)=x.
\end{equation}
It is clear that
$$
f^m(x) \in f^m(Y) \subseteq f^{m-1}(Y)\subseteq ... \subseteq f(Y).
$$
Since $g|_Y=f$, ~(\ref{3.3*}) implies $x=f^m(x)$. Hence $x\in f(Y)$, contrary to $x\in Y \setminus f(Y)$.

(ii). Suppose now that $X$ contains no isolated points and all dense subsets of $X$ are not shifted. If $g\in \operatorname{Iso} (X)$, $x_0\in X$ and $g^m(x_0)\neq x_0$ for every $m\in \mathbb N$, then $g^{n_1}(x_0)\neq g^{n_2}(x_0)$ for all distinct $n_1, n_2 \in \mathbb Z$. Write
$$
A:=\{x\in X : x=g^{-n}(x_0), n\in \mathbb N\}.
$$
The set $X\setminus A$ is a dense subset of $X$. Indeed, we have
$$
X\setminus A = \bigcap\limits_{n=1}^{\infty}(X\setminus \{g^{-n}(x_0)\})
$$
where $\{g^{-n}(x_0)\}$ is the one point set consisting of the unique point $g^{-n}(x_0)$. Since $X$ has no isolated points, the sets $X \setminus \{g^{-n}(x_0)\}$ are open dense subsets of $X$. Using the Baire category theorem we obtain that $X \setminus  A$ is also dense in $X$. Let $p$ be an arbitrary point of $X \setminus  A$. It is easy to show that $g(p) \in X \setminus  A$. Indeed, if $g(p) \notin X \setminus  A$, then there is $n \in \mathbb N$ such that
$$
g(p) = g^{-n}(x_0).
$$
Consequently, we obtain $p = g^{-n-1}(x_0)$. Hence $p$ belongs to $A$, contrary to $p\in X \setminus  A$. Thus the restriction $g|_{X \setminus  A}$ is an isometric self embedding of $X \setminus  A$. Let us prove the statement
\begin{equation}\label{3.4*}
x_0\notin g(X \setminus  A).
\end{equation}
Suppose $t\in (X \setminus  A)$ and $g(t)=x_0$ hold. The last equality implies $t=g^{-1}x_0$. Hence, by the definition of $X \setminus  A$, we have $t\in A$. Is is a contradiction of $t\in X \setminus  A$, so that~(\ref{3.4*}) holds. Since~(\ref{3.4*}) implies
$$
g|_{X \setminus  A}\notin \operatorname{Iso}(X\setminus A),
$$
the metric space $X \setminus  A$ is not shifted. Statement (ii) follows.
\end{proof}
Recall that a group $G$ is torsion if for every $g\in G$ there is $n \in \mathbb N$ such that $g^n = e$ where $e$ is the identity element of $G$.

\begin{cor}\label{c3.2}
Let $X$ be not shifted and complete. If the isometry group $\operatorname{Iso}(X)$ is torsion, then every dense subset of $X$ is not shifted.
\end{cor}
Let $X$ be a not shifted and complete metric space without isolated points. Suppose that there are $x \in X$ and $g \in \operatorname{Iso}(X)$ such that $g^m(x)=x$ holds if and only if $m=0$. Then Theorem \ref{t2.1*} implies the existence of $A \subseteq X$ which is shifted and dense (in $X$). Let us consider an example of such spaces.

\begin{exa}\label{ex3.3}
Let $\mathbb S = \{z \in \mathbb C: |z|=1\}$ be the unit circle in the complex plane $\mathbb C$. It is clear that $\mathbb S$ is a compact subset of $\mathbb C$ without isolated points. By Lemma \ref{compsh} the metric space $\mathbb S$ is not shifted. An irrational rotation is a map $T_\theta: \mathbb S\to \mathbb S$ with
$$
T_\theta(z) = ze^{2\pi i \theta}
$$
where $\theta$ is an irrational number. Every irrational rotation $T_\theta$ is an isometry of $\mathbb S$ and, in addition,
$$
T^m_\theta (z) = ze^{2\pi m i } = z
$$
holds if and only if $m=0$. Consequently, by Theorem \ref{t2.1*} there is $A \subseteq \mathbb S$ which is shifted and dense in $\mathbb S$. In fact the set
$$
A = \{T_\theta^n(z): n \in \mathbb N\}
$$
is dense and shifted for every $z \in \mathbb S$.
\end{exa}

The previous example is fundamental in the theory of dynamical systems. Using this example we can simply show that the balls and spheres in the finite dimensional Euclidean spaces have some dense shifted subsets. In the converse direction we have the following.

\begin{prop}\label{pr3.4}
There is a not shifted metric space $X$ with $|X|=\mathfrak{c}$ such that every subspace $Y$ of $X$ is also not shifted.
\end{prop}
\begin{proof}
It follows from Lemma~\ref{l:sr} and Lemma~\ref{l:esr}.
\end{proof}

\section{Minimal universality of classes of metric spaces}
\label{sect5}

 Let $\mathfrak{M}$ and $\mathfrak{B}$ be classes of metric spaces. We write $\mathfrak{M}\hookrightarrow\mathfrak{B}$ if for every $X\in\mathfrak{M}$ there is $Y\in\mathfrak{B}$ such that $X\hookrightarrow Y$.

The notation $\mathfrak{M}\preceq\mathfrak{B}$ means that, for every $X\in\mathfrak{M}$, there is $Y\in\mathfrak{B}$ such that $X\subseteq Y.$ It is easy to see that $\mathfrak{M}\subseteq\mathfrak{B}$ implies $\mathfrak{M}\preceq\mathfrak{B}$ but not conversely.

\begin{defin}
Let $X$ and $Y$ be metric spaces. If $X\hookrightarrow Y$ or $Y\hookrightarrow X$, then $X$ and $Y$ are said to be comparable. Otherwise, $X$ and $Y$ are incomparable.
\end{defin}

The following two definitions are similar to Definition~\ref{univ} and Definition~\ref{muniv} respectively.

\begin{defin}\label{guniv}
Let $\mathfrak{M}$ be a class of metric spaces. A class $\mathfrak{A}$ of metric spaces is said to be universal for $\mathfrak{M}$ or $\mathfrak{M}$-universal if $\mathfrak{M}\hookrightarrow\mathfrak{A}$.
\end{defin}

\begin{defin}\label{guniv*}
Let $\mathfrak{M}$ and $\mathfrak{A}$ be classes of metric spaces. The class $\mathfrak{A}$ is minimal $\mathfrak{M}$-universal if the following conditions hold:
\begin{enumerate}
\item[\rm(i)] $\mathfrak{A}$ is $\mathfrak{M}$-universal;

\item[\rm(ii)] The implication
\begin{equation}\label{impl}
(\mathfrak{M}\hookrightarrow\mathfrak{B})\wedge (\mathfrak{B}\preceq\mathfrak{A})\Rightarrow (\mathfrak{A}\preceq\mathfrak{B})
\end{equation} holds for every class $\mathfrak{B}$ of metric spaces;

\item[\rm(iii)] Distinct elements of $\mathfrak{A}$ are incomparable metric spaces.
\end{enumerate}
\end{defin}

\begin{prop}\label{p3.3v}
Let $\mathfrak{N}, \mathcal{P}$ and $\mathfrak{A}$ be classes of metric spaces. If $\mathfrak{A}$ is minimal $\mathfrak{N}$-universal, and
$$\mathfrak{N}\hookrightarrow\mathcal{P}\quad\mbox{and}\quad\mathcal{P}\hookrightarrow\mathfrak{N},$$
then $\mathfrak{A}$ is also minimal $\mathcal{P}$-universal.
\end{prop}
\begin{proof}
Suppose that $\mathfrak{A}$ is minimal $\mathfrak{N}$-universal and $\mathfrak{N}\hookrightarrow\mathcal{P}\quad\mbox{and}\quad\mathcal{P}\hookrightarrow\mathfrak{N}.$
We must show that conditions (i)--(iii) of Definition~\ref{guniv*} hold with $\mathfrak{M}=\mathcal{P}.$
\begin{enumerate}
\item[\rm(i).] Since $\mathfrak{A}$ is minimal $\mathfrak{N}$-universal, we have $\mathfrak{N}\hookrightarrow\mathfrak{A},$ that, together with $\mathcal{P}\hookrightarrow\mathfrak{N},$ implies $\mathcal{P}\hookrightarrow\mathfrak{A}.$
\item[\rm(ii).]  Let $\mathfrak{B}$  an arbitrary class of metric spaces. Since condition (ii) holds with $\mathfrak{M}=\mathfrak{N},$  we have
\begin{equation}\label{e3.2}
(\mathfrak{N}\hookrightarrow\mathfrak{B})\wedge (\mathfrak{B}\preceq\mathfrak{A})\Rightarrow (\mathfrak{A}\preceq\mathfrak{B}).
\end{equation}
From $\mathfrak{N}\hookrightarrow\mathcal{P}$ it follows that
\begin{equation*}
(\mathcal{P}\hookrightarrow\mathfrak{B})\wedge (\mathfrak{B}\preceq\mathfrak{A})\Rightarrow
(\mathfrak{N}\preceq\mathfrak{B})\wedge (\mathfrak{B}\preceq\mathfrak{A}).                                                           \end{equation*}
The last implication together with \eqref{e3.2} give us \eqref{impl}  with $\mathfrak{M}=\mathcal{P}.$
\item[\rm(iii).] Condition (iii) does not depend from the choice of class $\mathfrak{M}$ in Definition~\ref{guniv*}. Hence (iii) automatically holds.
\end{enumerate}
\end{proof}
\begin{rem}\label{eqvrem}
Condition (ii) of Definition~\ref{guniv*} admits the following equivalent form:
\newline $(\textrm{ii}^{*})$ The implication
\begin{equation}\label{impl2}
(\mathfrak{M}\hookrightarrow\mathfrak{B})\wedge (\mathfrak{B}\preceq \mathfrak{A})\Rightarrow (\mathfrak{A}\subseteq \mathfrak{B})
\end{equation} holds for every class $\mathfrak{B}$ of metric spaces.
\end{rem}

For the proof we need the following lemma.

\begin{lem}\label{Lvsp1}
Let $\mathfrak{M}$ and $\mathfrak{N}$ be classes such that $\mathfrak{M}\preceq \mathfrak{N}$ and $\mathfrak{N}\preceq \mathfrak{M}$. Suppose the implication
\begin{equation}\label{s1}
(A\subseteq C)\Rightarrow (A= C)
\end{equation}
holds foe all $A, C\in\mathfrak{M}.$ Then we have the inclusion $\mathfrak{M}\subseteq\mathfrak{N}.$
\end{lem}

\begin{proof}
Let $A\in\mathfrak{M}.$ Since $\mathfrak{M}\preceq\mathfrak{N},$ there is $B\in\mathfrak{N}$ such that $A\subseteq B$. From $\mathfrak{N}\preceq\mathfrak{M}$ it follows that there is $C\in\mathfrak{M}$ satisfying $B\subseteq C.$ Thus, we have
\begin{equation}\label{s2}
A\subseteq B\subseteq C
\end{equation}
and $A, C\in\mathfrak{M}$. Using \eqref{s1} we see that $A=C$. This equality and \eqref{s2} imply $A=B$. Consequently, for every $A\in\mathfrak{M}$ there is $B\in\mathfrak{N}$ such that $A=B$, i.e. $\mathfrak{M}\subseteq\mathfrak{N}$.
\end{proof}

\noindent \emph{Proof of Remark}~\ref{eqvrem}. Note that $(\textrm{ii}^{*})\Rightarrow \textrm{(ii)}$ is valid, because we have
$$
(\mathfrak{A}\subseteq\mathfrak{B})\Rightarrow(\mathfrak{A}\preceq \mathfrak{B})
$$
for arbitrary $\mathfrak{A}$ and $\mathfrak{B}$. Moreover, condition $\textrm{(iii)}$ of Definition~\ref{guniv*} gives us \eqref{s1} for $A, C\in \mathfrak{A}.$ Using Lemma~\ref{Lvsp1}, we obtain that $((\textrm{iii})\wedge (\textrm{ii}))\Rightarrow (\textrm{ii}^{*}).$ $\qquad\qquad\qquad\qquad\qquad\qquad\qquad\qquad\qquad\qquad\qquad\qquad\qquad\qquad\qquad\square$

\begin{cor}\label{Lmin}
Let $\mathfrak{M}$ and $\mathfrak{A}$ be classes of metric spaces and let $\mathfrak{A}$ be minimal $\mathfrak{M}$-universal. Then the implication
\begin{equation}\label{s3}
(\mathfrak{M}\hookrightarrow\mathfrak{B})\wedge (\mathfrak{B}\subseteq \mathfrak{A})\Rightarrow (\mathfrak{B}=\mathfrak{A})
\end{equation} holds for every class $\mathfrak{B}$ of metric spaces.
\end{cor}
\begin{proof}
Let $\mathfrak{B}$ be a class of metric spaces. Then, as was shown above, implication \eqref{impl2} holds. Since we evidently have
\begin{equation*}
(\mathfrak{M}\hookrightarrow\mathfrak{B})\wedge (\mathfrak{B}\subseteq\mathfrak{A})\Rightarrow (\mathfrak{M}\hookrightarrow\mathfrak{B})\wedge(\mathfrak{B}\preceq \mathfrak{A}),
\end{equation*} using \eqref{impl2}, we obtain
\begin{equation*}
(\mathfrak{M}\hookrightarrow\mathfrak{B})\wedge (\mathfrak{B}\subseteq\mathfrak{A})\Rightarrow (\mathfrak{A}\subseteq\mathfrak{B}).
\end{equation*}
The last implication and the trivial implication
\begin{equation*}
(\mathfrak{M}\hookrightarrow\mathfrak{B})\wedge (\mathfrak{B}\subseteq\mathfrak{A})\Rightarrow (\mathfrak{B}\subseteq\mathfrak{A})
\end{equation*} give \eqref{s3}.
\end{proof}

Let us describe the ``degenerate'' minimal universal classes of metric spaces.

\begin{prop}
Let $\mathfrak{M}$ be a non--empty class of metric spaces and let $\mathfrak{A}$ be a minimal $\mathfrak{M}$-universal class of metric spaces. The following statements hold.
\begin{enumerate}
\item[\rm(i)]\textit{We have $\varnothing\in\mathfrak{A}$ if and only if $\mathfrak{A}=\mathfrak{M}=\{\varnothing\}$.}

\item[\rm(ii)]\textit{The following conditions are equivalent:}
\begin{enumerate}
\item[($\textrm{ii}_{1}$)]\textit{There is $A\in\mathfrak{A}$ such that $|A|=1;$}

\item[($\textrm{ii}_{2}$)]\textit{The equality $\mathfrak{A}=\{A\}$ holds with some $A$ having $|A|=1;$}

\item[($\textrm{ii}_{3}$)]\textit{The inequality $|X|\le 1$ holds for every $X\in\mathfrak{M}$ and, in addition, there is $Y\in\mathfrak{M}$ such that $|Y|=1.$}
\end{enumerate}
\end{enumerate}
\end{prop}
\begin{proof}
Let $\varnothing\in\mathfrak{A}$ and $A\in\mathfrak{A}.$ Since $\varnothing\hookrightarrow A,$ condition $\textrm{(iii)}$ of Definition~\ref{guniv*} implies the equality $\varnothing=A.$ The equality $\mathfrak{A}=\{\varnothing\}$ follows. If $X$ is an arbitrary metric space belonging to $\mathfrak{M},$ then from $\mathfrak{M}\hookrightarrow\mathfrak{A}$ and $\mathfrak{A}=\{\varnothing\}$ it follows that $X\hookrightarrow\varnothing,$ thus $X=\varnothing.$ The last equality holds for every $X\in\mathfrak{M},$ i.e. $\mathfrak{M}=\{\varnothing\}.$ The statement $\textrm{(i)}$ is proved. Statement $\textrm{(ii)}$ can be proved similarly, so that we omit it here.
\end{proof}

The next proposition is an analog of Proposition~\ref{vspompr}.

\begin{prop}\label{vspompr*}
Let $\mathfrak{A}$ be a non--empty class of metric spaces. The following conditions are equivalent.
\begin{enumerate}
\item[\rm(i)]\textit{Every $X\in\mathfrak{A}$ is not shifted and every two distinct $Y,$ $Z\in\mathfrak{A}$ are incomparable.}

\item[\rm(ii)]\textit{$\mathfrak{A}$ is minimal universal for itself.}

\item[\rm(iii)]\textit{There is a class $\mathfrak{M}$ of metric spaces such that $\mathfrak{A}$ is minimal $\mathfrak{M}$-universal.}
\end{enumerate}
\end{prop}
\begin{proof}
Let $\textrm{(i)}$ hold. We prove that $\mathfrak{A}$ is minimal $\mathfrak{A}$-universal. It suffices to show that $(\textrm{ii}^{*})$ holds with $\mathfrak{M}=\mathfrak{A}.$ Suppose that for a given class $\mathfrak{B}$ of metric spaces we have $$\mathfrak{A}\hookrightarrow\mathfrak{B}\quad\mbox{and}\quad \mathfrak{B}\preceq\mathfrak{A}.$$ Let $X\in\mathfrak{A}.$ Since $\mathfrak{A}\hookrightarrow\mathfrak{B},$ there is $Y\in\mathfrak{B}$ such that $X\hookrightarrow Y$. The condition $\mathfrak{B}\preceq\mathfrak{A}$ implies that there is $W\in\mathfrak{A}$ satisfying the inclusion $Y\subseteq W.$ The last inclusion together with $X\hookrightarrow Y$ imply $X\hookrightarrow W.$ Hence $X$ and $W$ are comparable. By condition $\textrm{(i)},$ every two distinct $X, W\in\mathfrak{A}$ are incomparable. Hence, $X=W$ holds. Let us consider an arbitrary $f: X\hookrightarrow Y$. Since $X$ is not shifted, the isometric embedding
$$
\begin{array}{cccc}
X\xrightarrow {\ \ \mbox{\emph{f}}\ \ } &
Y\xrightarrow {\ \ \mbox{\emph{in}}\ \ } &
W \xrightarrow {\ \ \mbox{\emph{id}}\ \ } X,
\end{array}
$$
where $in$ is the natural inclusion of $Y$ in $W$ and $id$ is the identical map, is an isometry.
Hence $in(Y)=W$, i.e. the equality $Y=X$ holds. Consequently, $\mathfrak{A}\subseteq\mathfrak{B}$ holds.
Thus we have condition $(\textrm{ii}^{*})$ with $\mathfrak{M}=\mathfrak{A}$.

The implication $\textrm{(ii)}\Rightarrow \textrm{(iii)}$ is evident. Let us prove $\textrm{(iii)}\Rightarrow \textrm{(i)}.$ Suppose that there is a class $\mathfrak{M}$ of metric spaces such that $\mathfrak{A}$ is minimal $\mathfrak{M}$-universal. It suffices to show that every $X\in\mathfrak{A}$ is not shifted. Suppose, contrary, that $X\in\mathfrak{A}$ is shifted. Let $X_{0}\subseteq X,$ $X\ne X_{0}$ and $X\hookrightarrow X_{0}.$ Write
\begin{equation}\label{eqvvsp}
\mathfrak{A}^{0}:=(\mathfrak{A}\setminus\{X\})\cup\{X_{0}\}.
\end{equation}
Since $X\hookrightarrow X_{0}$ and $\mathfrak{M}\hookrightarrow\mathfrak{A},$ we have $\mathfrak{M}\hookrightarrow\mathfrak{A}^{0}.$ It is clear that $\mathfrak{A}^{0}\preceq\mathfrak{A}.$ Using \eqref{impl} with $\mathfrak{B}=\mathfrak{A}^{0},$ we obtain $\mathfrak{A}\preceq\mathfrak{A}^{0}.$ Hence there is $Y\in\mathfrak{A}^{0}$ such that $X\subseteq Y.$ It follows directly from \eqref{eqvvsp}, that $Y\ne X_{0}$. Thus we have $Y\in\mathfrak{A}^{0}\setminus\{X_{0}\},$ i.e. $Y\in\mathfrak{A}\setminus\{X\}.$ Hence, $X$ and $Y$ are some distinct elements of $\mathfrak{A}.$ The implication $$(X\subseteq Y)\Rightarrow (X\hookrightarrow Y)$$ and condition (iii) of Definition~\ref{guniv*} show that $X=Y.$ Hence $X\in\mathfrak{A}^{0}$, contrary to \eqref{eqvvsp}. The implication $\textrm{(iii)}\Rightarrow \textrm{(i)}$ follows.
\end{proof}

As was shown in Example~\ref{ex2}, there is a family $\mathfrak{Y}$ of metric spaces and there are minimal $\mathfrak{Y}$-universal metric spaces $X$ and $Y$ such that $X\not\simeq Y$. The situation is much more satisfactory when we consider the minimal $\mathfrak{Y}$-universal classes of metric spaces: such classes are always isomorphic. For the exact formulation of this result we need the following definition.

\begin{defin}
Let $\mathfrak{A}$ and $\mathfrak{B}$ be classes of metric spaces. A map $F:\mathfrak{A}\rightarrow\mathfrak{B}$ is an isomorphism if it is bijective and $F(X)\simeq X$ holds for every $X\in\mathfrak{A}.$
\end{defin}

We shall say that classes $\mathfrak{A}$ and $\mathfrak{B}$ of metric spaces are isomorphic and write $\mathfrak{A}\simeq \mathfrak{B}$ if there is an isomorphism $F:\mathfrak{A}\rightarrow\mathfrak{B}.$

\begin{rem}\label{R3.1}
If $\mathfrak{M}, \mathfrak{A}$ and $\mathfrak{B}$ are classes of metric spaces and $\mathfrak{A}\simeq \mathfrak{B},$ and $\mathfrak{A}$ is (minimal) $\mathfrak{M}$-universal, then $\mathfrak{B}$ is also (minimal) $\mathfrak{M}$-universal.
\end{rem}

\begin{rem}\label{R3.2}
Let $X$ and $Y$ be metric spaces and let $\mathfrak{A}=\{X\}$ and $\mathfrak{B}=\{Y\}.$ Then the classes $\mathfrak{A}$ and $\mathfrak{B}$ are isomorphic if and only if $X$ and $Y$ are isometric.
\end{rem}

\begin{thm}\label{T3.main}
Let $\mathfrak{M}$ be a non--empty class of metric spaces. The following statements hold.
\begin{enumerate}
\item[\rm(i)]\textit{If $\mathfrak{A}$ and $\mathfrak{B}$ are minimal $\mathfrak{M}$-universal classes of metric spaces, then $\mathfrak{A}\simeq \mathfrak{B}$.}

\item[\rm(ii)]\textit{If there exists a minimal $\mathfrak{M}$-universal class $\mathfrak{A}$ of metric spaces, then there exists $\mathfrak{B}\subseteq\mathfrak{M}$ such that $\mathfrak{A}\simeq \mathfrak{B}.$}
\end{enumerate}
\end{thm}

\begin{proof}
(ii). Let $\mathfrak{A}$ be a minimal $\mathfrak{M}$-universal class of metric spaces. Then there is a class $\mathfrak{M}_{\mathfrak{A}}$ of metric spaces such that $$\mathfrak{M}_{\mathfrak{A}}\simeq\mathfrak{M}\quad\mbox{and}\quad\mathfrak{M}_{\mathfrak{A}}\preceq\mathfrak{A}.$$
Since $\mathfrak{M}_{\mathfrak{A}}\preceq\mathfrak{M}_{\mathfrak{A}}$  evidently holds, condition $(\textrm{ii}^{*})$ implies $\mathfrak{A}\subseteq\mathfrak{M}_{\mathfrak{A}}.$ To prove it we use \eqref{impl2} with $\mathfrak{M}_{\mathfrak{A}}$ instead of $\mathfrak{M}$ and $\mathfrak{B}.$ Now statement (ii) follows from $\mathfrak{M}_{\mathfrak{A}}\simeq\mathfrak{M}$ and $\mathfrak{A}\subseteq\mathfrak{M}_{\mathfrak{A}}.$

(i). Let $\mathfrak{A}$ and $\mathfrak{B}$ be minimal $\mathfrak{M}$-universal classes of metric spaces. We prove that $\mathfrak{A}\simeq\mathfrak{B}$. Using statement (ii), we may suppose that $$\mathfrak{A}\subseteq\mathfrak{M}\quad\mbox{and}\quad\mathfrak{B}\subseteq\mathfrak{M}.$$
For every $X\in\mathfrak{A}$ there are $Y$ and $Y_X$ such that $$Y\in\mathfrak{B}\quad\mbox{and}\quad Y_{X}\subseteq Y\quad\mbox{and}\quad Y_{X}\simeq X.$$
Define a class $\tilde{\mathfrak{A}}$ as $\{Y_{X}: X\in\mathfrak{A}\}$. It is clear that $\tilde{\mathfrak{A}}\simeq\mathfrak{A}$ and $\tilde{\mathfrak{A}}\preceq\mathfrak{B}.$
Since $\mathfrak{A}$ is minimal $\mathfrak{M}$-universal and $\tilde{\mathfrak{A}}\simeq\mathfrak{A},$ the class $\tilde{\mathfrak{A}}$ is also minimal $\mathfrak{M}$-universal. In particular, we have $\mathfrak{M}\hookrightarrow\tilde{\mathfrak{A}}$. Using condition $(\textrm{ii}^{*})$, we obtain $\mathfrak{B}\subseteq\tilde{\mathfrak{A}}.$ By Corollary~\ref{Lmin}, the equality $\tilde{\mathfrak{A}}=\mathfrak{B}$ follows. This equality and $\tilde{\mathfrak{A}}\simeq\mathfrak{A}$ imply $\mathfrak{B}\simeq\mathfrak{A}.$
\end{proof}

\begin{cor}\label{L3.vsp2}
Let $\mathfrak{M}$ be a class of metric spaces and let $\mathfrak{A}$ and $\mathfrak{B}$ be two minimal $\mathfrak{M}$-universal classes of metric spaces. Then $\mathfrak{A}$ and $\mathfrak{B}$ are isomorphic.
\end{cor}

The following proposition is a particular case of Theorem~\ref{T3.main}.

\begin{prop}\label{p3.15}
Let $\mathfrak{M}$ be an arbitrary class of metric spaces. If there exists a metric space $X$ such that the class $\{X\}$ is minimal $\mathfrak{M}$-universal, then there exists a minimal $\mathfrak{M}$-universal metric space $Y\in\mathfrak{M}$ and $Y\simeq X$.
\end{prop}

\begin{proof}
Let $X$ be a metric space and let the class $\mathfrak{A}=\{X\}$ be minimal $\mathfrak{M}$-universal. Theorem~\ref{T3.main} implies that there exists $\mathfrak{B}\subseteq\mathfrak{M}$ such that $\mathfrak{A}\simeq \mathfrak{B}.$ Since $\mathfrak{A}$ is an one-element class, $\mathfrak{A}\simeq \mathfrak{B}$ and  $\mathfrak{B}\subseteq\mathfrak{M},$ there is $Y\in\mathfrak{M}$ such that $\mathfrak{B}=\{Y\}.$
As was noted in Remark~\ref{R3.2}, from $\{Y\}\simeq\{X\}$ follows that $X\simeq Y.$ Since $\mathfrak{M}\hookrightarrow\mathfrak{A},$ we have $\mathfrak{M}\hookrightarrow X.$ Now from $\mathfrak{M}\hookrightarrow X$ and $X\simeq Y$ it follows that $\mathfrak{M}\hookrightarrow Y.$ Using Proposition~\ref{vspompr*}, we see that $X$ is not shifted. Hence $Y$ is also not shifted. By Theorem~\ref{mainth1}, $Y$ is minimal $\mathfrak{M}$-universal.
\end{proof}

A natural question is how the minimal $\mathfrak{M}$-universal subclasses of $\mathfrak{M}$ can be described.
\begin{prop}\label{p3.15a}
Let $\mathfrak{M}$ be a class of metric spaces and let $\mathfrak{A}\subseteq\mathfrak{M}$. If $\mathfrak{M}\hookrightarrow\mathfrak{A}$ holds, then the following conditions are equivalent:
\begin{enumerate}
\item[\rm(i)]\textit{$\mathfrak{A}$ is minimal $\mathfrak{M}$-universal;}
\item[\rm(i)]\textit{All metric spaces $X\in\mathfrak{A}$ are not shifted and every two distinct $Y, Z\in\mathfrak{A}$ are incomparable.}
\end{enumerate}
\end{prop}
\begin{proof}
Suppose $\mathfrak{M}\hookrightarrow\mathfrak{A}$ holds. It is clear that $$(\mathfrak{A}\subseteq\mathfrak{M})\Rightarrow (\mathfrak{A}\hookrightarrow\mathfrak{M}).$$ Consequently, by
Proposition~\ref{p3.3v}, $\mathfrak{A}$ is minimal $\mathfrak{M}$-universal if and only if $\mathfrak{A}$ is minimal
$\mathfrak{A}$-universal. Now the logical equivalence $(\textrm{i})\Leftrightarrow(\textrm{ii})$ follows from Proposition~\ref{vspompr*}.
\end{proof}

We shall give a description of minimal $\mathfrak{M}$-universal subclasses $\mathfrak{A}$ of $\mathfrak{M}$ in the case when $\mathfrak{M}$ is a set of metric spaces.

\begin{defin}
A binary relation $\leq$ on a set $P$ is said to be a quasi-order if $\leq$ is reflexive and transitive. An antisymmetric quasi-order $\leq$ on $P$ is a partial order on $P.$ Recall that a binary relation $\leq$ on $P$ is antisymmetric if  $$(x\leq y \, \wedge  \, y\leq x)\Rightarrow (x=y)$$ holds for all $x, y\in P.$
\end{defin}

\begin{lem}\label{L3.main}
If $\leq$ is a quasi-order on a set $P,$ then a relation $\Theta_{\leq}$ defined by
\begin{equation}\label{l3eq1}
(a \,\, \Theta_{\leq}\,\, b)\Leftrightarrow (a\leq b \, \wedge  \, b\leq a)
\end{equation}
is an equivalence relation on $P.$ Moreover, a relation $\sqsubseteq$ defined on the quotient set $P/ \Theta_{\leq}$ as
\begin{equation}\label{l3eq2}
([a]_{\Theta_{\leq}}\sqsubseteq [b]_{\Theta_{\leq}})\Leftrightarrow (\exists\, x\in [a]_{\Theta_{\leq}} \,\wedge\, \exists\, y\in [b]_{\Theta_{\leq}}: \, x\leq y)
\end{equation}
is a correctly defined partial order on $P/\Theta_{\leq}$.
\end{lem}

This lemma is a standard fact from the theory of ordered sets (see, for example, \cite[p. 141]{DP}), so that we omit the proof here.

If $X, Y$ and $Z$ are metric spaces, then we evidently have $X\hookrightarrow X$ and
$$(X\hookrightarrow Y)\, \wedge\, (Y\hookrightarrow Z)\Rightarrow (X\hookrightarrow Z).$$
Hence, if $\mathfrak{M}$ is an arbitrary non--empty set of metric spaces, then the restriction of the binary relation $\hookrightarrow$ on the set
$\mathfrak{M}\times\mathfrak{M}$ is an quasi-order on $\mathfrak{M}.$ Let us define the relations $\Theta_{\hookrightarrow}$ and $\sqsubseteq$ as in \eqref{l3eq1} and \eqref{l3eq2} respectively. The following proposition is a reformulation of Lemma~\ref{L3.main} for the case when $\leq$ is $\hookrightarrow.$

\begin{prop}
Let $\mathfrak{M}$ be an arbitrary non--empty set of metric spaces. Then $\mathfrak{M}/\Theta_{\hookrightarrow}$ is a poset with the partial order $\sqsubseteq$.
\end{prop}

\begin{rem}If $X$ and $Y$ are isometric metric spaces, then the statements $X\hookrightarrow Y$ and $Y\hookrightarrow X$ are evident. In general, the converse is not true. Indeed, if $X=[0, \infty)$ and $Y=(0, \infty),$ then we have $X\hookrightarrow Y$ and $Y\hookrightarrow X,$ and $X\not\simeq Y.$
\end{rem}

\begin{lem}\label{L:isom}
Let $X$ and $Y$ be metric spaces. Suppose $X$ is not shifted. If there exist $f: X\hookrightarrow Y$ and $g: Y\hookrightarrow X,$ then the mappings $f$ and $g$ are isometries.
\end{lem}

\begin{proof}
Since $X$ is not shifted, the mapping
$$\begin{array}{cccc}
X & \xhookrightarrow{f} &
Y\xhookrightarrow{g} & X
\end{array}$$
is an isometry. Hence $f$ and $g$ are surjections. The surjective isometric embeddings are isometries.
\end{proof}

Lemma~\ref{L:isom} implies the following corollary.

\begin{cor}\label{C3}
Let $\mathfrak{M}$ be an arbitrary non--empty set of metric spaces. If $$Y\in [X]_{\Theta_{\hookrightarrow}}$$ and $X$ is not shifted, then $Y$ is not shifted and $X\simeq Y.$
\end{cor}

Let $P$ be a partially ordered set with partial order $ \leq$. A point $p\in P$ is maximal if the implication $$(p\leq x)\Rightarrow (x=p)$$ holds for every $x\in P.$ We shall denote by $Max_{\leq}$ the set of maximal points of the partially ordered set $P.$

Let $\mathfrak{F}=\{A_{i}: i\in I\}$ be a family of sets. A system of distinct representatives for $\mathfrak{F}$ is a set $\{a_{i}: i\in I\}$ with the property that $a_{i}\in A_{i}$ and $a_{i}\ne a_{j}$ for all distinct $i, j\in I.$ The Axiom of Choice states that a system distinct representatives exists for each family $\{A_{i}: i\in I\}$ with mutually disjoint non--empty $A_{i}.$

\begin{thm}\label{T3.Max}
Let $\mathfrak{M}$ be an arbitrary non--empty set of metric spaces and let $Max_{\sqsubseteq}$ be the set of maximal elements of the partially ordered set $\mathfrak{M}/\Theta_{\hookrightarrow}.$ The following statements are equivalent.
\begin{enumerate}
\item[\rm(i)]\textit{There is a minimal $\mathfrak{M}$-universal class of metric spaces.}

\item[\rm(ii)]\textit{A system of the distinct representatives for the set $Max_{\sqsubseteq}$ is a minimal $\mathfrak{M}$-universal class.}

\item[\rm(iii)]\textit{For every $\gamma\in \mathfrak{M}/\Theta_{\hookrightarrow}$ there are $\beta\in Max_{\sqsubseteq}$ and $X\in\beta$ such that $\gamma\sqsubseteq\beta$ and $X$ is not shifted.}
\end{enumerate}
\end{thm}
\begin{rem}\label{R3}
Statement $\textrm{(ii)}$ of Theorem~\ref{T3.Max} is logically equivalent to the conjunction of the following two conditions.
\begin{enumerate}
\item[$\rm(ii_{1})$] For every $\gamma\in \mathfrak{M}/\Theta_{\hookrightarrow}$ there is $\beta\in Max_{\sqsubseteq}$ such that $\gamma\sqsubseteq\beta.$

\item[$\rm(ii_{2})$] If $\gamma\in Max_{\sqsubseteq}$ and $X\in\gamma,$ then $X$ is not shifted.
\end{enumerate}
\end{rem}
\noindent To see it we can use Corollary~\ref{C3} and note that the implication $$(\beta \sqsubseteq\gamma)\Rightarrow (\gamma=\beta)$$ holds for all $\gamma\in\mathfrak{M}/\Theta_{\hookrightarrow}$ and $\beta\in Max_{\sqsubseteq}.$

\begin{proof}[Proof of Theorem~\ref{T3.Max}]
$\textrm{(i)}\Rightarrow \textrm{(iii)}.$ Let $\mathfrak{A}$ be a minimal $\mathfrak{M}$-universal class of metric spaces. We must show that, for every $\gamma\in\mathfrak{M}/\Theta_{\hookrightarrow}$, there are $\beta\in Max_{\sqsubseteq}$ and $X\in\beta$ such that $\gamma\sqsubseteq\beta$ and $X$ is not shifted. Let $\gamma\in\mathfrak{M}/\Theta_{\hookrightarrow}$  and let $Z$ be an arbitrary metric space belonging to $\gamma.$ Since $\gamma\subseteq\mathfrak{M}$ and $\mathfrak{A}$ is minimal $\mathfrak{M}$-universal, then there is $Y\in\mathfrak{A}$ such that $Z\hookrightarrow Y$. By Theorem~\ref{T3.main}, there is $X\in\mathfrak{M}$ such that $X\simeq Y.$ Consequently, $Z\hookrightarrow X.$  Define $\beta:=[X]_{\Theta_{\hookrightarrow}}.$ Using \eqref{l3eq2} with $a=Z$, $b=X$ and $\Theta_{\leq}=\Theta_{\hookrightarrow}$, we obtain from $Z\hookrightarrow X$ that $\gamma\sqsubseteq\beta$. To prove that $X$ is not shifted recall that $X\simeq Y$ and $Y\in\mathfrak{A}$ where $\mathfrak{A}$ is minimal $\mathfrak{M}$-universal. By Proposition~\ref{vspompr*}, the space $Y$ is not shifted. Hence $X$ is not shifted as a space which is isometric to $Y$. It still remains to prove that $\beta\in Max_{\sqsubseteq}$. The statement $\beta\in Max_{\sqsubseteq}$ holds if for every $\alpha\in\mathfrak{M}/\Theta_{\hookrightarrow}$ the inequality $\beta\sqsubseteq\alpha$ implies $\beta=\alpha$. Let $\beta=[X]_{\Theta_{\hookrightarrow}}$ and $\alpha=[W]_{\Theta_{\hookrightarrow}}$ where $X$ is the same as above and $W$ is an arbitrary element of $\alpha.$ Inequality $\beta\sqsubseteq\alpha$ implies that $X\hookrightarrow W.$ Since $\alpha\subseteq\mathfrak{M},$ we have $W\in\mathfrak{M}.$ Consequently, there is a space $Q\in\mathfrak{A}$ such that $W\hookrightarrow Q.$ Let $X\hookrightarrow Y$ with the same $Y\in\mathfrak{A}$ as above. By Proposition~\ref{vspompr*}, we have either $Q=Y$ or $Q$ and $Y$ are incomparable. Since $X\hookrightarrow W,$ $W\hookrightarrow Q$ and $Y\simeq X$ hold, the metric spaces $Q$ and $Y$ are comparable. Hence, the equality $Q=Y$ holds. In particular, we obtain $Q\hookrightarrow X.$ From $X\hookrightarrow W,$ $W\hookrightarrow Q$ and $Q\hookrightarrow X$ it follows that $X\hookrightarrow W$ and $W\hookrightarrow X.$ Hence we have $\beta=[X]_{\Theta_{\hookrightarrow}}=[W]_{\Theta_{\hookrightarrow}}=\alpha.$ The implication $\textrm{(i)}\Rightarrow \textrm{(iii)}$ follows.

$\textrm{(iii)}\Rightarrow \textrm{(ii)}.$ Suppose condition $\textrm{(iii)}$ holds. Let
$$
\mathfrak{A}=\{A_{\alpha}: \alpha\in Max_{\sqsubseteq}\}
$$
where $A_{\alpha}\in\alpha$, be a system of distinct representatives for the set $Max_{\sqsubseteq}$.
We claim that $\mathfrak{A}$ is minimal $\mathfrak{M}$-universal. We first establish that $\mathfrak{A}$ is $\mathfrak{M}$-universal. Let $X\in\mathfrak{M}$ and let $\gamma\in\mathfrak{M}/\Theta_{\hookrightarrow}$ such that $X\in\gamma.$ By $\textrm{(iii)}$, there is $\beta\in Max_{\sqsubseteq}$ such that $\gamma\sqsubseteq\beta$. Using \eqref{l3eq2}, we obtain $X\hookrightarrow A_{\beta}$. Using \eqref{l3eq2} it is also easy to prove that $A_{\alpha}$ and $A_{\beta}$ are incomparable if $\alpha \ne \beta$. Hence, conditions $\textrm{(i)}$ and $\textrm{(iii)}$ of Definition~\ref{guniv*} are satisfied. We end the proof by demonstration that condition $\textrm{(ii)}$ of this definition holds. In accordance with Remark~\ref{R3}, we can suppose that all $A_{\alpha}$ are not shifted. Hence, by Proposition~\ref{vspompr*}, $\mathfrak{A}$ is minimal universal for $\mathfrak{A}$ itself. Consequently, by Definition~\ref{guniv*}, we have
$$
(\mathfrak{A}\hookrightarrow\mathfrak{B})\wedge(\mathfrak{B}\preceq\mathfrak{A})\Rightarrow(\mathfrak{A}\preceq\mathfrak{B})
$$
for every class $\mathfrak{B}$ of metric spaces. Since $\mathfrak{A}\preceq\mathfrak{M},$ we have
$$
(\mathfrak{M}\hookrightarrow\mathfrak{B})\wedge(\mathfrak{B}\preceq\mathfrak{A}) \Rightarrow(\mathfrak{A}\hookrightarrow\mathfrak{B}) \wedge(\mathfrak{B}\preceq\mathfrak{A}).
$$
Consequently, we obtain \eqref{impl},
$$
(\mathfrak{M}\hookrightarrow\mathfrak{B}) \wedge(\mathfrak{B}\preceq\mathfrak{A})\Rightarrow(\mathfrak{A}\preceq\mathfrak{B})
$$
for every class $\mathfrak{B}$ of metric spaces. Condition $\textrm{(ii)}$ of Definition~\ref{guniv*} follows.

To complete the proof, it suffices to observe that the implication $\textrm{(ii)}\Rightarrow \textrm{(i)}$ is trivial.
\end{proof}

We finish a section by construction of a minimal universal class for the class of linear normed spaces with given finite dimension.

\begin{lem}\label{l3.25}
let $X$ be a complete non--empty metric space and let $Y$ be a connected metric space. Then every open $f: X\hookrightarrow Y$  is an isometry.
\end{lem}
\begin{proof}
Let $f: X\hookrightarrow Y$ be open. The mapping $f$ is an isometry if $f(X)=Y$. The set $f(X)$ is an open non--empty subset of $Y$. Since $Y$ is connected it is sufficient to show that $f(X)$ is closed.
Suppose that a point $a$ belongs to the set $\overline{f(X)}\setminus f(X),$ where $\overline{f(X)}$  is the closure of $f(X)$ in $Y.$
Then there is a sequence $(y_n)_{n\in\mathbb N}$ such that $\mathop{\lim}\limits_{n\to\infty}d_{Y}(y_n, a)=0$ and $y_n\in f(X)$ for every $n\in\mathbb N.$
Since $f$ is an isometric embedding, the sequence $(f^{-1}(y_n))_{n\in\mathbb N}$ is a Cauchy sequence in $X$.
The completeness of $X$ implies that there is $b\in X$ such that $$\lim_{n\to\infty}d_{X}(f^{-1}(y_n), b)=0.$$ Using the continuity of $f$, we obtain
$$f(b)=\lim_{n\to\infty}f(f^{-1}(y_n))=\lim_{n\to\infty}y_n=a.$$
Hence $f(b)=a,$ so that $a\in f(X)$ contrary to $a\in\overline{f(X)}\setminus f(X).$ Thus, $\overline{f(X)}=f(X),$ i.e. $f(X)$ is closed.
\end{proof}

\begin{prop}\label{p3.26}
Let $n$ be a positive integer number and let $\mathcal{N}_{n}$  be the class of all normed $n$-dimensional linear spaces over the field $\mathbb R.$ Then there is a minimal $\mathcal{N}_{n}$-universal subset of $\mathcal{N}_{n}.$
\end{prop}
\begin{proof}
Every finite--dimensional normed linear space over $\mathbb R$ is separable. Hence, by the Banach-Mazur theorem, for every $Y\in\mathcal{N}_{n},$ there is $Y'\subseteq C[0, 1]$ such that $Y\simeq Y'.$  Consequently, without loss of generality, we may identify $\mathcal{N}_{n}$ with a subset of $2^{C[0, 1]}$ and denote this subset by the same symbol $\mathcal{N}_{n}$. Let $\{[Y]: Y\in\mathcal{N}_{n}\}$  be the quotient set of $\mathcal{N}_{n}$ by the relation $\simeq$, where, as above, $X\simeq Y$ means that $X$ and $Y$ are isometric. Let $\mathfrak{M}$ be a system of representatives for $\{[Y]: Y\in\mathcal{N}_{n}\}$. We claim that $\mathfrak{M}$ is a minimal $\mathcal{N}_{n}$-universal set of metric spaces that $\mathfrak{M}\subseteq\mathcal{N}_{n}$ and $\mathfrak{N}_{n}\hookrightarrow\mathfrak{M}.$ By Proposition~\ref{p3.15a}, $\mathfrak{M}$ is minimal $\mathfrak{N}_{n}$--universal if and only if:

\medskip

\noindent $(\textrm{i}_1)$  Every $Y\in\mathfrak{M}$ is not shifted;

\medskip

\noindent $(\textrm{i}_2)$  Every two distinct $X, Y\in\mathfrak{M}$ are incomparable.

\medskip

\noindent Since  every finite-dimensional normed linear space over $\mathbb R$ is homogeneous and boundedly compact, property $(\textrm{i}_1)$
follows from Lemma~\ref{homsh}. Let us consider $(\textrm{i}_2)$. Brouwer's invariance  of domain theorem says that a continuous injective map $f: U\to\mathbb R^{n}$ is open for every open $U\subseteq\mathbb R^{n}$ (see, for example, \cite[p. 34]{JWV}). Since all $X, Y\in\mathfrak{M}$
are homeomorphic to $\mathbb R^{n}$, we see that conditions of Lemma~\ref{l3.25} are valid for every $f:X\hookrightarrow Y$ with $X, Y\in\mathfrak{M}$. From this lemma and from the definition of $f$ it follows that if $X$ and $Y$ are comparable and $X, Y\in\mathfrak{M}$, then $X=Y$.  Condition $(\textrm{i}_2)$ follows.
\end{proof}

Using the minimal $\mathcal{N}_{n}$-universal set $\mathfrak{M}$ that was constructed in the proof of Proposition~\ref{p3.26}, we can construct a disjoint union
$$
X=\coprod_{Y\in\mathfrak M}Y
$$
such that $X$  is minimal $\mathcal{N}_{n}$-universal. See Corollary~\ref{c4.12} in the next section.

\section{From minimal universal classes to minimal universal spaces. Disjoint unions of metric spaces.}
\label{sect6}

 In this section we shall denote by $I$ a non--empty index set.

\begin{defin}\label{part} A metric space $X$ is a disjoint union of metric spaces $Y_{i},$ $i\in I,$ if there is a partition $\{X_{i}: i\in I\}$ of the set $X$ such that $X_{i}\simeq Y_{i}$ for every $i\in I.$ In this case we write
$$
X=\mathop{\coprod}\limits_{i\in I}Y_{i}.
$$
\end{defin}

Recall that $\{X_{i}: i\in I\}$ is a partition of $X$ if $X=\mathop{\bigcup}\limits_{i\in I}X_{i},$ $X_{i}\ne\varnothing$ and $X_{i}\cap X_{j}=\varnothing$ for all distinct $i, j\in I.$ In particular, if $X=\mathop{\coprod}\limits_{i\in I}Y_{i},$ then all spaces $Y_{i}$ are non--empty.

The next proposition follows from Proposition~3.8 of \cite{DMV}.

\begin{prop}
A disjoint union $\mathop{\coprod}\limits_{i\in I}Y_{i}$ exists for every non--empty family of non--empty metric spaces $Y_i,$ $i\in I.$
\end{prop}

It is clear that $\mathfrak{M}\hookrightarrow\mathop{\coprod}\limits_{i\in I}Y_{i}$ holds for $\mathfrak{M}=\{Y_i: i\in I\}.$ A legitimate question to raise at this point is whether there exists a metric $d_X$ on $X=\mathop{\coprod}\limits_{i\in I}Y_{i}$ such that $X$ is minimal $\mathfrak{M}$-universal.

For simplicity, we assume that the metric spaces $Y_i$, $i\in I$, in the following proposition, are pairwise disjoint, $Y_i \cap Y_j=\varnothing$ if $i\ne j$, and $\mathop{\coprod}\limits_{i\in I}Y_{i}$ is a metric space with the ground set $X=\mathop{\bigcup}\limits_{i\in I}Y_{i}$ such that the restriction $d_{X}|_{Y_i \times Y_i}$ is equal to $d_{Y_i}$ for every $i\in I.$

\begin{prop}\label{pr4}
Let $\mathfrak{M}=\{Y_{i}: i\in I\}$ be a set of metric spaces. If a disjoint union $X=\mathop{\coprod}\limits_{i\in I}Y_{i}$ is a minimal $\mathfrak{M}$-universal metric space, then all $Y_{i}\in\mathfrak{M}$ are not shifted and every two distinct $Y_{i}, Y_{j}\in\mathfrak{M}$ are incomparable.
\end{prop}

\begin{proof}
Let $X$ be minimal $\mathfrak{M}$-universal. Suppose that, contrary our claim, $Y_{i_1}\hookrightarrow Y_{i_2}$ for some distinct $i_1, i_2\in I$. Then the metric space
$$
X^{i_1}=\mathop{\coprod}\limits_{i\in I, i\ne i_1}Y_{i}
$$
with the metric induced from $X$, is $\mathfrak{M}$-universal and $X\setminus X^{i_1}\ne\varnothing,$ in contradiction with the minimal universality of $X.$

Similarly, if $Y_{i_1}$ is shifted and $Y_{i_1}^{0}\subseteq Y_{i_1}$ such that
$$
Y_{i_1}\simeq Y_{i_1}^{0} \text{ and } Y_{i_1}\setminus Y_{i_1}^{0}\ne\varnothing,
$$
then the metric space
$$
X_{0}^{i_1}=\left(\coprod_{i\in I, i\ne i_1}Y_i\right)\sqcup Y_{i_1}^{0}
$$
with the metric induced from $X$, is $\mathfrak{M}$-universal, contrary to the condition.
\end{proof}

\begin{rem}\label{rem2.4.4}
In the next section (see Corollary~\ref{c5.4}), we shall construct a family $\mathfrak{M}=\{Y_i:i\in I\}$ such that:
\begin{itemize}
  \item $\mathfrak{M}$ admits a minimal universal metric space ;
  \item Every $Y_i$ is not shifted and every two distinct $Y_i, Y_j$ are incomparable;
  \item There is no minimal $\mathfrak{M}$-universal $X$ having a form $X=\coprod\limits_{i\in I}Y_i$.
\end{itemize}
\end{rem}

Proposition~\ref{pr4} and Proposition~\ref{vspompr*} imply the following.

\begin{cor}\label{c4.3}
If $X=\mathop{\coprod}\limits_{i\in I}Y_{i}$ is a minimal universal metric space for a class $\mathfrak{M}=\{Y_{i}: i\in I\},$ then there is a class $\mathfrak{A}$ of metric spaces such that $\mathfrak{M}$ is minimal $\mathfrak{A}$-universal.
\end{cor}

The following theorem describes, for given family $\mathfrak{M}$ of metric spaces, the structure of minimal $\mathfrak{M}$-universal metric spaces $X$ which have a form $X=\coprod\limits_{Y\in \mathfrak{M}}Y$.

\begin{thm}\label{t4.4}
Let $\mathfrak{M}=\{Y_{i}: i\in I\}$ be a non--empty set of non--empty metric spaces and let $X$ be a minimal $\mathfrak{M}$-universal metric space. Suppose that all $Y_i \in\mathfrak{M}$ are not shifted and every two distinct $Y_{i_1},\, Y_{i_2}\in\mathfrak{M}$ are incomparable. Then the following statements are equivalent.
\begin{enumerate}
\item[\rm(i)]\textit{The metric space $X$ is a disjoint union of metric spaces $Y_i,$ $i\in I,$
$$
X=\mathop{\coprod}\limits_{i\in I}Y_{i}.
$$
}

\item[\rm(ii)]\textit{The metric space $X$ satisfies the following conditions.}

\smallskip

\begin{enumerate}

\item[($\textrm{ii}_{1}$)]\textit{For every $x_0 \in X$, there is a unique $Y_{i_0}\in\mathfrak{M}$ such that $$Y_{i_0}\not\hookrightarrow X\setminus\{x_0\}.$$}
\item[($\textrm{ii}_{2}$)]\textit{For every non--empty subset $\mathfrak{M}_{0}$ of the set $\mathfrak{M},$ there is a unique minimal $\mathfrak{M}_{0}$-universal metric subspace $X^{0}$ of the space $X.$}
\end{enumerate}
\end{enumerate}
\end{thm}

\begin{proof}
Let (i) hold. Then, by Definition~\ref{part}, there is a partition $\{X_i : i\in I\}$ of $X$ such that $X_i\simeq Y_i$ holds for every $i\in I.$ If $x_0$ is an arbitrary point of $X,$ then there is a unique $i_0 \in I$ with $x_{0}\in X_{i_0}.$ It is clear that
$$
Y_{i}\hookrightarrow \bigcup_{i\in I,\, i\ne i_{0}}X_{i}\hookrightarrow X\setminus\{x_0\}
$$
holds for every $i\in I\setminus\{i_0\}.$
Consequently we have
$$
Y_{i_0}\not\hookrightarrow X\setminus\{x_0\},
$$
because $X$ is minimal $\mathfrak{M}$-universal. The implication $(\textrm{i})\Rightarrow (\textrm{ii}_{1})$ follows.

To deal with $(\textrm{i})\Rightarrow (\textrm{ii}_{2}),$ consider an arbitrary non--empty $I_{0}\subseteq I.$ Let $\mathfrak{M}_{0}=\{Y_i: i\in I_{0}\}.$ It is easy to see that the metric space $\mathop{\cup}\limits_{i\in I_0}X_{i},$ with the metric induced from $X$, is minimal $\mathfrak{M}_{0}$-universal. Indeed, it is clear that $\mathfrak{M}_{0}\hookrightarrow\mathop{\bigcup}\limits_{i\in I_0}X_{i}.$
If there is $X^{0}\subseteq\mathop{\bigcup}\limits_{i\in I_0}X_{i}$ such that $\mathfrak{M}_{0}\hookrightarrow X^{0}$ and
\begin{equation}\label{4(1)}
\left(\bigcup_{i\in I_0}X_i\right)\setminus X^{0}\ne\varnothing,
\end{equation}
then
\begin{equation}\label{4(2)}
\mathfrak{M}\hookrightarrow X^{0}\cup\left(\bigcup_{i\in I\setminus I_{0}} X_i\right),
\end{equation}
because
\begin{equation*}
\mathfrak{M}=\mathfrak{M}_{0}\cup (\mathfrak{M}\setminus\mathfrak{M}_{0}) \ \,\text{ and } \ \, \mathfrak{M}\setminus\mathfrak{M}_{0}\hookrightarrow\bigcup_{i\in I\setminus I_0}X_i.
\end{equation*}

It follows from \eqref{4(1)} and \eqref{4(2)} that $X$ is not minimal $\mathfrak{M}$-universal, contrary to the condition of the theorem. Now suppose that $X^{1}\subseteq X$ is minimal $\mathfrak{M}_{0}$-universal and
\begin{equation}\label{4(5)}
X^{1}\ne\bigcup_{i\in I_{0}}X_i.
\end{equation}
If $X^{1}\supseteq\bigcup_{i\in I_{0}}X_i,$ then \eqref{4(5)} implies that $X^{1}$ is not minimal $\mathfrak{M}_{0}$-universal. Hence, there is $x_{0}\in\bigcup_{i\in I_{0}}X_i$ such that $x_{0}\not\in X^{1}.$ Consequently, we have
$$
\mathfrak{M}\hookrightarrow X^{1}\cup\left(\mathop{\bigcup}\limits_{i\in I\setminus I_0}X_i\right)\quad\mbox {and}\quad X^{1}\cup\left(\mathop{\bigcup}\limits_{i\in I\setminus I_0}X_i\right)\subseteq X\setminus\{x_0\},
$$
contrary to the minimality of $X$. The implication $(\textrm{i})\Rightarrow (\textrm{ii}_{2})$ is proved, so that $(\textrm{i}) \Rightarrow (\textrm{ii})$ follows.

Let us prove the validity of $(\textrm{ii})\Rightarrow (\textrm{i})$. Suppose that $(\textrm{ii})$ holds. Since $X$ is $\mathfrak{M}$-universal, for every $i\in I$, there is $f_{i}: Y_{i}\hookrightarrow X$. Write $X_{i}=f_{i}(Y_{i})$. It is clear that $X_{i}\simeq Y_i$ holds for every $i\in I.$ Hence it suffices to prove that $\{X_i: i\in I\}$ is a partition of $X$. All sets $Y_i$ are non--empty. Hence all $X_i$ are non--empty also. Since $X$ is minimal $\mathfrak{M}$-universal and
\begin{equation*}
\mathfrak{M}\hookrightarrow\bigcup_{i\in I}X_{i} \quad\mbox{and}\quad \bigcup_{i\in I}X_{i}\subseteq X,
\end{equation*}
we have $X=\mathop{\bigcup}\limits_{i\in I}X_{i}.$ It remains to prove that
$$
X_{i}\cap X_{j}=\varnothing
$$
holds for every pair distinct $i, j\in I$. Suppose contrary that there exist $i_1, i_2\in I$ with
\begin{equation*}
X_{i_1}\cap X_{i_2}\ne\varnothing \quad\mbox{and}\quad i_1\ne i_2.
\end{equation*}
Let $i\in I$. Since $Y_i \in\mathfrak{M}$ is not shifted, Proposition~\ref{vspompr} implies, for every $A_{i}\subseteq X$ with $A_{i}\simeq Y_i,$ that $A_i$ is a minimal universal metric space for the class $\{Y_i\}$. Using condition $(\textrm{ii}_{2})$, with $\mathfrak{M}_{0}=\{Y_{i_k}\},$ $k=1, 2,$ we obtain that there is a unique $A_{i_k}\subseteq X$ such that $A_{i_k}\simeq Y_{i_k}.$ By definition, $X_{i_k}\subseteq X$ and, moreover, $X_{i_k}\simeq Y_{i_k}$ holds for $k=1, 2.$ Consequently, if $$x_{0}\in X_{i_1}\cap X_{i_2},$$ then we have
\begin{equation*}
Y_{i_1}\not\hookrightarrow X\setminus\{x_0\} \quad\mbox{and}\quad Y_{i_2}\not\hookrightarrow X\setminus\{x_0\},
\end{equation*}
contrary to condition $(\textrm{ii}_{2}).$
\end{proof}

\begin{rem}\label{r4.6}
Theorem~\ref{t4.4} remains valid if we replace condition $(\textrm{ii}_{2})$ by the following:
\begin{enumerate}
\item[($\textrm{ii}_{2}^{0}$)]\textit{For every $Y_{i}\in\mathfrak{M}$, there is a unique $X^{i}\subseteq X$ such that $X^{i}\simeq Y_{i}.$}
\end{enumerate}
\end{rem}

\begin{proof}
By the hypothesis of Theorem~\ref{t4.4}, all $Y_{i}\in\mathfrak{M}$ are not shifted. From Proposition~\ref{vspompr} it follows that a metric space $X^{i}$ is minimal universal for the class $\{Y_i\}$ if and only if $X^{i}\simeq Y_i.$ Hence $(\textrm{ii}_{2}^{0})$ is an equivalent for the condition:
\begin{enumerate}
\item[($\textrm{ii}_{2}^{1}$)]For every $i\in I$ there is a unique $X^{i}\subseteq X$ which is minimal universal for $\{Y_i\}.$
\end{enumerate}
We evidently have $(\textrm{ii}_{2})\Rightarrow(\textrm{ii}_{2}^{1}).$ To complete the proof it suffices to note that in the proof of the implication $(\textrm{ii})\Rightarrow(\textrm{i})$ we use, in fact, condition $(\textrm{ii}_{2}^{1})$ instead of $(\textrm{ii}_{2}).$
\end{proof}

\begin{thm}\label{t4.7}
Let $\mathfrak{M}=\{Y_{i}: i\in I\}$ be a non--empty family of non--empty metric spaces. A disjoint union
\begin{equation*}
X=\coprod_{i\in I}Y_{i}
\end{equation*}
is a minimal $\mathfrak{M}$-universal metric space if and only if the following statements hold.
\begin{enumerate}
\item[\rm(i)] \textit{Every $Y_i\in\mathfrak{M}$ is not shifted}.
\item[\rm(ii)] \textit{Metric spaces $Y_i$ and $Y_j$ are incomparable for distinct $i,j\in I,$}.
\item[\rm(iii)] \textit{For every $Y_i\in\mathfrak{M}$ there is a unique $X^{i}\subseteq X$ such that $Y_{i}\simeq X^{i}.$}
\end{enumerate}
\end{thm}
\begin{proof}
Suppose that $X$ is minimal $\mathfrak{M}$-universal. Then, by Proposition~\ref{pr4}, conditions $(\textrm{i})$ and $(\textrm{ii})$ are valid. Now Theorem~\ref{t4.4} implies $(\textrm{iii})$. Conversely, if $(\textrm{i}),$ $(\textrm{ii})$ and $(\textrm{iii})$ hold, then, using Theorem~\ref{t4.4} and Remark~\ref{r4.6}, we see that $X$ is minimal $\mathfrak{M}$-universal if and only if condition $(\textrm{ii}_{1})$ of Theorem~\ref{t4.4} holds. To prove $(\textrm{ii}_{1})$ suppose that $x_0$ is an arbitrary point of $X.$ Since we have $X=\mathop{\sqcup}\limits_{i\in I}Y_i,$ there is a partition $\{X_i: i\in I\}$ of $X$ such that $X_i\simeq Y_i$ for every $i\in I.$ Let $X_{i_0}$ be a part such that $x_0\in X_{i_0},$ $i_{0}\in I.$ It suffices to show that $Y_{i_0}\not\hookrightarrow X\setminus\{x_0\}$ and $$(Y_{i}\not\hookrightarrow X\setminus\{x_0\})\Rightarrow(i=i_0)$$ for every $i\in I.$ As in the proof of Theorem~\ref{t4.4}, we obtain $Y_{i}\hookrightarrow X\setminus\{x_0\}$ for $i\in I\setminus\{i_0\}.$ Suppose we also have $Y_{i_0}\hookrightarrow X\setminus\{x_0\}.$ Then there is $X^{i_0}\subseteq X\setminus\{x_0\}$ such that $Y_{i_0}\simeq X^{i_0}$. Thus we obtain $X^{i_0}\ne X_{i_0},$ because $x_0\notin X^{i_0}$ and $x_0\in X_{i_0}.$ It is a contradiction with $(\textrm{iii}).$
\end{proof}

Let us give some applications of Theorem~\ref{t4.7} to construction of minimal universal metric spaces based on the disjoint unions.

\begin{defin}\label{d4.8}
Let $\varepsilon\in (0, \infty).$ A metric space $X$ is said to be $\varepsilon$-connected if for every $x, y\in X$ there is a finite sequence $(x_i),$ $i=1,...,n,$ such that:

$x_{i}\in X$, for all $i\in 1, \ldots, n$ and  $x_1 =x$, $x_n = y$ and
$$
d_{X}(x_i, x_{i+1})\le\varepsilon
$$
for every $i\in\{1,..., n-1\}$.
\end{defin}

\begin{thm}\label{t4.9}
Let $\mathfrak{M}=\{Y_i: i\in I\}$ be a non--empty set of non--empty metric spaces. Suppose all $Y_i$ are not shifted and $Y_i\not\hookrightarrow Y_j$ if $i\ne j$. If there exists $\varepsilon>0$ such that $Y_i$ is $\varepsilon$-connected for every $i\in I,$ then there is a metric on
$$
X=\coprod_{i\in I}Y_{i}
$$
such that the disjoint union $X$ with this metric is a minimal $\mathfrak{M}$-universal metric space.
\end{thm}

The proof of Theorem~\ref{t4.9} can be divided into three following steps.

\medskip

\noindent $\bullet$ We consider a disjoint metric spaces $X_i\simeq Y_i$, $i\in I$ and construct a connected weighted graph $(G, w)$ with the vertex set $V(G),$ the edge set $E(G)$ and the weight $w: E(G)\rightarrow\mathbb R^{+}$, $\mathbb R^{+}=[0, \infty)$, such that:

$V(G)=X$ with $X=\bigcup\limits_{i\in I}X_i$;

If $x, y$ are distinct points of $X_i,$ $i\in I,$ then $\{x, y\}\in E(G)$;

The equality $$w(\{x, y\})=d_{X_i}(x, y)$$ holds for every $i\in I$ and all distinct $x, y\in X_i.$

\medskip

\noindent $\bullet$ Using some results of \cite{DMV}, we show that the weighted shortest path pseudometric $d_{w, X}$ is a metric on $X$ satisfying the equality $$d_{w, X}(x, y)= d_{X_i}(x, y)$$ for every $i\in I$ and all $x, y\in X_i.$ Thus, $X$ with the metric $d_{w, X}$ is a disjoint union of metric spaces $Y_i,$ $i\in I.$

\medskip

\noindent $\bullet$ We apply Theorem~\ref{t4.7} to the metric space $(X, d_{w, X})$ to obtain that it is minimal $\mathfrak{M}$-universal.

In the proof of Theorem~\ref{t4.9} we will use a basic terminology of the Graph Theory (see, for example, \cite{BM}).  All our graphs $G$ are simple, so that we can identify the edge set $E(G)$ with a set of two-elements subsets of the vertex set $V(G)$. We use the standard definition of the path and the cycle. It should be noted that all vertices of any cycle $C$ can be labeled as $x_0,...,x_n$ with $|V(C)|=n$, $x_0=x_n$ and for $0\leqslant i < j \leqslant n,$ $\{x_j,x_i\}\in E(G)$ if and only if $j-i=1$.

If $(G, w)$ is a weighted graph, then for each subgraph $F$ of the graph $G$ define
\begin{equation}\label{weight}
w(F):=\sum_{e\in E(F)}w(e).
\end{equation}

Recall that a symmetric function $\rho: X\times X\rightarrow\mathbb R^{+}$ is a pseudometric on $X$ if $\rho$ satisfies the triangle inequality $\rho(x, y)\le\rho(x, z)+\rho (z, y)$ and the equality $\rho(x, y)=0$ for all $x, y, z\in X,$ but, in general, the implication
$$(\rho(x, y)=0)\Rightarrow(x=y)$$ does not hold.

\begin{defin}
Let $(G, w)$ be a weighted graph with $E(G)\ne\varnothing.$ The weight $w$ is pseudometrizable if there is a pseudometric $d_{V}$ on the vertex set $V=V(G)$ such that
\begin{equation}\label{cond}
d_{V}(x, y)=w(\{x, y\}) \quad\mbox{for every}\quad \{x, y\}\in E(G).
\end{equation}
A pseudometrizable weight $w$ is metrizable if there is a metric $d_{V}$ satisfying \eqref{cond}.
\end{defin}

Let $(G, w)$ be a weighted graph and let $u, v$ be distinct vertices of $G.$ Let us denote by $P_{u, v}$ the set of all paths joining $u$ and $v$ in $G.$ Thus, a finite sequence $(t_0, ..., t_n)$ belongs to $P_{u, v}$ if and only if $u=t_0,$ $v=t_n,$ $t_{i}\in V(G)$ for $0\le i\le n,$ and $t_i \ne t_j$ for $0\le i<j\le n$ and $\{t_{i}, t_{i+1}\}\in E(G)$ for $0\le i\le n-1.$

Let $u, v\in V(G)$. Write
\begin{equation}\label{e4.8}
d_{w, G}(u, v):=\begin{cases}
         \inf\{w(F): F\in P_{u, v}\} & \mbox{if} $ $ u \ne v\\
         0 & \mbox{if}$ $ u = v. \\
         \end{cases}
\end{equation}
Recall that a graph $G$ is connected if $P_{u, v}\ne\varnothing$ for all distinct $u, v\in V(G).$ For the connected weighted graphs it is well known that the function $d_{w, G}$ is a pseudometric on the set $V(G).$ This pseudometric will be termed as the weighted shortest path pseudometric. It coincides with the usual path metric if $w(e)=1$ for every $e\in E(G).$

The next lemma follows at once from Proposition~2.1 and Proposition~3.1 of \cite{DMV}.

\begin{lem}\label{l4.11}
Let $(G, w)$ be a connected weighted graph with $V(G)\ne\varnothing.$ The following conditions are equivalent:
\begin{enumerate}
\item[\rm(i)]\textit{The weight $w$ is pseudometrizable;}

\item[\rm(ii)]\textit{For every cycle  $C=(t_0, t_1, ..., t_n)$ in $G$ we have the inequality}
\begin{equation}\label{e4.9}
2\max_{0\le i\le n-1}w\{t_i, t_{i+1}\}\le w(C),
\end{equation}
\textit{where $w(C)$ is determined in accordance with \eqref{weight}.}
\end{enumerate}
A pseudometrizable weight $w$ is metrizable if and only if the inequality
\begin{equation}
\inf\{w(F): F\in P_{u, v}\}>0
\end{equation}
holds for every pair of distinct $u, v\in V(G).$ If $w$ is a metrizable weight, then the weighted shortest path pseudometric $d_{w,G}$ is a metric and condition \eqref{cond} holds with $d_{V}=d_{w, G}$.
\end{lem}

\begin{proof}[Proof of Theorem~\ref{t4.9}]
Let $\{X_i: i\in I\}$ be a family of metric spaces satisfying the conditions
$$
X_{i}\cap X_{j}=\varnothing \quad\mbox{and}\quad X_{i}\simeq Y_{i}
$$
for all distinct $i, j\in I.$ Write
$$
X=:\bigcup_{i\in I}X_i.
$$
By conditions of the theorem all $Y_i$ are not shifted and $Y_{i}\not\hookrightarrow Y_{j}$ if $i\ne j$ and $i, j\in I$. Suppose that there is $\varepsilon >0$ such that $Y_i$ is $\varepsilon$-connected for every $i\in I$. Since, for every $i\in I,$ $Y_i$ is not empty, there is a system of distinct representatives $\{a_i: i\in I\}$ for the family $\{X_i: i\in I\},$ i.e., $$a_{i}\in X_{i}\quad\mbox{and}\quad a_i\ne a_j$$
for all distinct $i, j\in I.$ Let us consider a graph $G$ with the vertex set $V(G)=X$ and with the edge set $E(G)$ defined by the rule: if $x, y$ are distinct points of $X$, then
\begin{equation}\label{e4.11}
(\{x, y\}\in E(G))\Leftrightarrow ((\exists\, i\in I: x, y\in X_i)\,\vee\, (\exists\, i, j\in I:x=a_i\, \wedge\, y=a_j))
\end{equation}
(see Figure 1). The graph $G$ is connected, because $\{X_i: i\in I\}$ is a partition of $X=V(G)$.
\begin{figure}\label{f1}
\begin{center}
\includegraphics[width=0.8\textwidth]{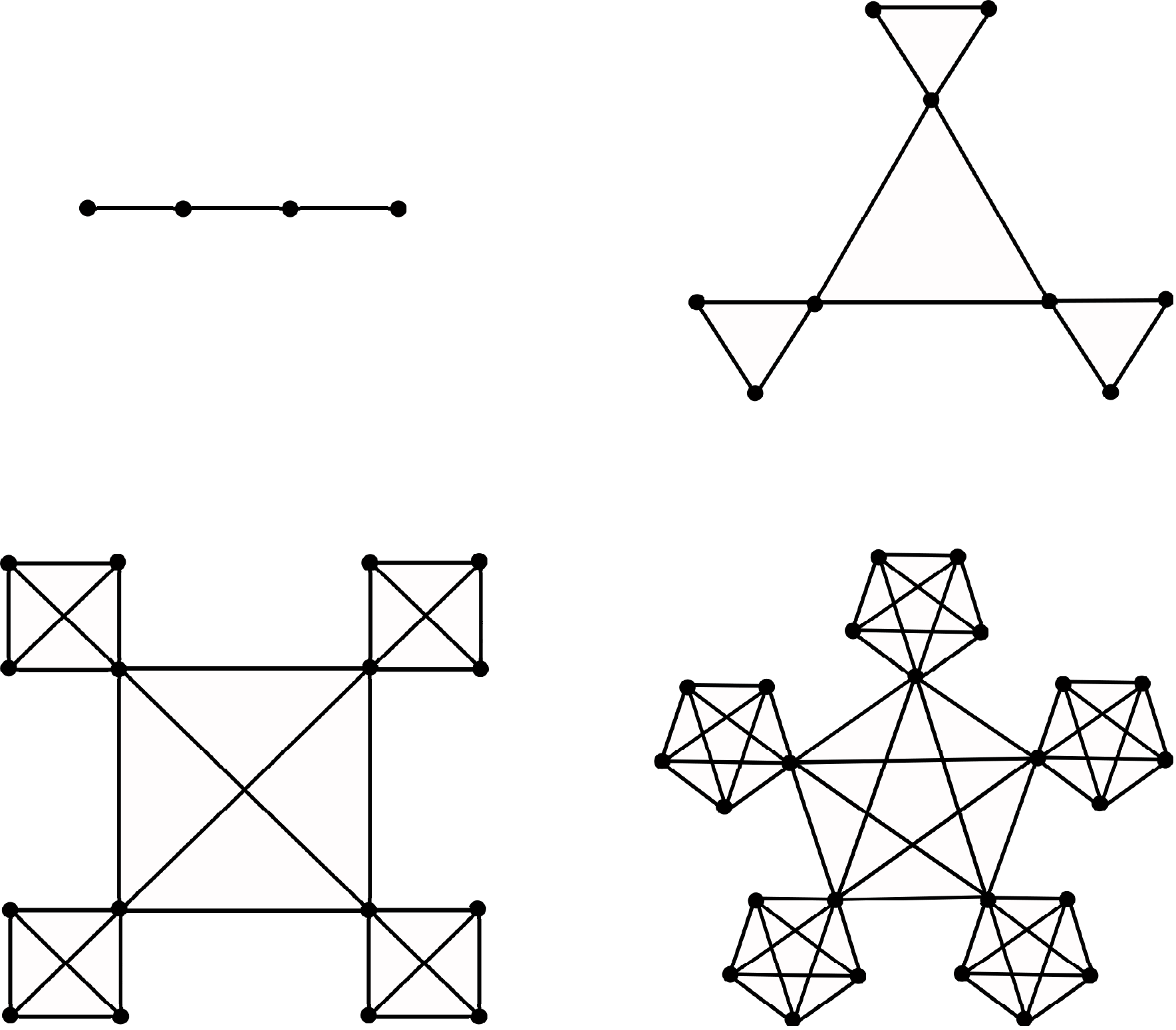}\\
\caption{The graphs $G$ for the families $\{Y_i:i\in I\}$ with $I=\{1, 2, ..., k\}$ and $|Y_1|=|Y_2|=...=|Y_k|$, \, $k=2, 3, 4, 5.$}
\end{center}
\end{figure}
Let $\varepsilon<\varepsilon_{1}<\infty$. Define a weight $w: E(G)\rightarrow\mathbb R^{+}$ as follows
\begin{equation}\label{e4.12}
w(\{x, y\}):=\begin{cases}
         d_{X_i}(x, y) & \mbox{if} $ $ \{x, y\}\subseteq X_i \quad \mbox{for some} \quad i\in I\\
         \varepsilon_{1} & \mbox{if}$ $ x = a_i\, \,\mbox{and} \,\, y=a_j \,\, \mbox{for some distinct} \,\, i, j\in I. \\
         \end{cases}
\end{equation}
By Lemma~\ref{l4.11}, the weight $w$ is pseudometrizable if inequality \eqref{e4.9} holds for every cycle $C$ in $G$. Let $C=(t_0, ..., t_n),$ $t_0=t_n,$ be an arbitrary cycle in $G.$  If there is $i\in I$ such that $V(C)\subseteq X_i,$ then \eqref{e4.9} holds because $C$ is a cycle in the induced subgraph $G[X_i]$ for which the restriction $w\mid_{E(G[X_i])}$ of the weight function $w$ is evidently metraizable. Recall that, by the definition of induced subgraphs, we have $V(G[X_i])=X_i$ and, for every $x, y\in X_i$
$$
(\{x, y\}\in E(G[X_i]))\Leftrightarrow (\{x,y\}\in E(G)).
$$
Suppose now that
\begin{equation}\label{e4.13}
V(C)\not\subseteq X_i \quad\mbox{for any}\quad i\in I.
\end{equation}
Let $f: X\to I$ be a function satisfying the statement
$$
(f(x)=i)\Leftrightarrow(x\in X_i)
$$
for all $x\in X$, $i\in I$.  Condition \eqref{e4.13} implies the existence of $k\in\{0, ..., n-1\}$ such that
\begin{equation}\label{e4.14}
f(t_k)\ne f(t_{k+1}).
\end{equation}
Indeed, if the equality $f(t_k)=f(t_{k+1})$ holds for every $k\in\{0, ..., n-1\},$ then we have
\begin{equation*}
V(C)=\{t_0, ..., t_{n-1}\}\subseteq X_{f(t_0)},
\end{equation*}
contrary to \eqref{e4.13}. After a renumbering of the vertices of $C$, we may suppose
\begin{equation}\label{e4.15}
f(t_0)\ne f(t_{1}).
\end{equation}
It follows from \eqref{e4.15} and the definitions of $w$ and $f$ that
\begin{equation}\label{e4.16}
t_{0}=a_{f(t_0)}\quad\mbox{and}\quad t_1=a_{f(t_1)}.
\end{equation}
In addition, since $t_0 =t_n,$ we also have
\begin{equation}\label{e4.17}
t_{n}=a_{f(t_n)}.
\end{equation}
We claim that \eqref{e4.16} and \eqref{e4.17} together with \eqref{e4.11} imply the equality
\begin{equation}\label{e4.18}
t_{i}=a_{f(t_i)}
\end{equation}
for every $i\in\{0, ..., n\}.$ Suppose contrary that there is $i\in\{0, ..., n\}$ such that $t_{i}\ne a_{f(t_i)}.$ Then we can find the smallest index $i_0$ such that
\begin{equation*}
t_{i_0}\ne a_{f(t_{i_0})}\quad\mbox{and}\quad t_i=a_{f(t_i)} \quad\mbox{for}\quad i\in \{0, ..., i_0 -1\}.
\end{equation*}
From the definition of $i_0$ and \eqref{e4.17}, we have
\begin{equation}\label{e4.19}
t_{i_0-1}= a_{f(t_{i_0-1})}\quad\mbox{with}\quad 2\le i_{0}-1 <i_0\le n-1.
\end{equation}
Since
$$
\{t_{i_0 -1}, t_{i_0}\}\in E(C)\subseteq E(G) \text{ and } t_{i_0}\in X_{f(t_{i_0})}\setminus\{a_{f(t_{i_0})}\},
$$
rule~\eqref{e4.11} implies that
\begin{equation}\label{e4.20}
f(t_{i_0})=f(t_{i_0-1}).
\end{equation}
Let $i_1\in\{i_0, i_{0}+1, ..., n\}$ such that
\begin{equation}\label{e4.21}
t_{i_0}\ne a_{f(t_{i_0})}, \,\, t_{i_0+1}\ne a_{f(t_{i_0+1})},..., \,\,t_{i_1}\ne a_{f(t_{i_1})}
\end{equation}
and
\begin{equation}\label{e4.22}
t_{i_1+1}= a_{f(t_{i_1+1})}.
\end{equation}
Then we have $i_0\le i_1 \le n-1$ and, similarly \eqref{e4.20}, the equality
\begin{equation}\label{e4.23}
f(t_{i_1})=f(t_{i_1+1})
\end{equation}
holds. Using \eqref{e4.21} and \eqref{e4.11} we can also prove the chain of the equalities
\begin{equation}\label{e4.24}
f(t_{i_0})=f(t_{i_0+1})=...=f(t_{i_1+1}).
\end{equation}
Thus we have $t_{i_{0}-1}=a_{f(t_{i_{0}-1})},$ by \eqref{e4.19}, $a_{f(t_{i_{0}-1})}=a_{f(t_{i_0})},$ by \eqref{e4.20}, $a_{f(t_{i_0})}=a_{f(t_{i_0+1})}=...=a_{f(t_{i_1+1})},$ by \eqref{e4.24} and $a_{f(t_{i_1+1})}=t_{i_1+1}$ by \eqref{e4.22}. Consequently, the equality
\begin{equation}\label{e4.25}
t_{i_0-1}=t_{i_1+1}
\end{equation}
holds. Since $2\le i_{0}-1<i_0<i_1+1\le n,$ equality \eqref{e4.25} leads to a contradiction of the definition of cycles.

From \eqref{e4.18} and \eqref{e4.12} we obtain that
\begin{equation*}
w(\{t_0, t_1\})=w(\{t_1, t_2\})=...=w(\{t_{n-1}, t_n\})=\varepsilon_1.
\end{equation*}
This chain of equalities evidently implies inequality \eqref{e4.9}. Hence, by Lemma~\ref{l4.11}, the weight $w$ defined by \eqref{e4.12}, is pseudometraizable. Using this lemma again, we obtain that $w$ is metraizable if the inequality
\begin{equation}\label{e4.26}
\inf\{w(F): F\in P_{x, y}\}>0
\end{equation}
holds for all distinct $x, y\in X.$ Let us prove \eqref{e4.26}. Let $x$ and $y$ be two distinct points of $X$ and let $F=(t_0, ..., t_n)$ be a path joining $x$ and $y$ in $G$, $x=t_0$ and $y=t_n.$ If there exists $i\in I$ such that $\{t_k, t_{k+1}\}\subseteq X_{i}$ for every $k\in\{0, ..., n-1\},$ then $x, y\in X_i$ and using the triangle equality, we see that
\begin{equation*}
w(F)=\sum_{k=0}^{n-1}w(\{t_k, t_{k+1}\})=\sum_{k=0}^{n-1}d_{X_i}(t_k, t_{k+1})\ge d_{X_i}(x, y)>0.
\end{equation*}
(Note that $x$ and $y$ are distinct points of $X$). Otherwise, there is $k_0\in\{0, ..., n-1\}$ such that $f(t_{k_0})\ne f(t_{k_0}+1)$ and similarly \eqref{e4.16} we obtain
\begin{equation}\label{e4.27}
t_{k_0}=a_{f(t_{k_0})}\quad\mbox{and}\quad t_{k_0+1}=a_{f(t_{k_0+1})}.
\end{equation}
Since $w$ is nonnegative, \eqref{e4.27} and \eqref{e4.12} imply that
\begin{equation*}
w(F)=\sum_{k=0}^{n-1}w(\{t_k, t_{k+1}\})\ge w(\{t_{k_0}, t_{k_0+1}\})=w(\{a_{f(t_{k_0})}, a_{f(t_{k_0+1})}\})=\varepsilon_{1}.
\end{equation*}
Thus, for every $F\in P_{x, y},$  we obtain
\begin{equation}\label{e4.28}
w(F) \ge d_{X_i}(x, y)
\end{equation}
if there is $i\in I$ such that $x, y\in X_i$, and
\begin{equation}\label{e4.29}
w(F)\ge\varepsilon_1\
\end{equation}
if $x\in X_i$ and $y\in Y_j$ with distinct $i, j\in I$. Inequality \eqref{e4.26} follows from \eqref{e4.28} and \eqref{e4.29}. Hence the weight $w$ is metraizable.

Let $d_{w, G}$ be the weighted shortest path metric defined by \eqref{e4.8}. (Note that inequality \eqref{e4.26} implies that $d_{w, G}$ really is a metric). Using the last statement of Lemma~\ref{l4.11}, we see that the metric space $X$ with the metric $d_{X}=d_{w, G}$ is $\mathfrak{M}$-universal, i.e.
\begin{equation*}
X=\coprod_{i\in I}Y_i\quad\mbox{if}\quad d_{X}=d_{w, G}.
\end{equation*}
In accordance with Theorem~\ref{t4.7}, to prove that $X$ is minimal $\mathfrak M$-universal it suffices to show that for every $i\in I$ there is a unique $X^{i}\subseteq X$ such that
$$
Y_{i}\simeq X^{i}.
$$
It is clear that $Y_{i}\simeq X_{i}$ for every $i\in I$. Suppose there is $i_0\in I$ such that $X^{i_0}\ne X_{i_0}$ and $Y_{i_0}\simeq X^{i_0}$. If $X^{i_0}\subseteq X_{i_0}$, then we have $X^{i_0}= X_{i_0}$ because $X_{i_0}\simeq Y_{i_0}\simeq X^{i_0}$ and $Y_{i_0}$ is not shifted. Hence there is a point $x_1\in X^{i_0}$ and an index $i_1\in I$ such that $x_1\in X_{i_1}$ and $i_1\neq i_0$. If the inclusion $X^{i_0}\subseteq X_{i_1}$ holds, then $Y_{i_0}\hookrightarrow X_{i_1}$ and consequently $Y_{i_0}\hookrightarrow Y_{i_1}$. This leads to a contradiction with the condition that $Y_{i_0}$ and $Y_{i_1}$ are incomparable if $i_0\ne i_1$. Consequently, there is $i_2\in I$ and $x_2\in X^{i_0}$ such that $x_2\in X_{i_2}$ with $i_2\ne i_1.$ Since $x_1, x_2\in X^{i_0},$ $X^{i_0}\simeq Y_{i_0}$ and $Y_{i_0}$ is $\varepsilon$-connected, there is a finite sequence $t_0, ..., t_n$ of points from $X^{i_0}$ with
\begin{equation}\label{e4.30}
t_{i}\in X, \, \, t_0=x_1, \  t_n=x_2, \,\,\mbox{and}\  d_{X}(t_i, t_{i+1})<\varepsilon
\end{equation}
for every $i\in\{0, ..., n-1\}$. Since $t_0=x_1$ and $t_n=x_2$, we have $t_0\in X_{i_1}$ and $t_n\in X_{i_2}$  with $i_1\ne i_2$. Consequently, we can find $i\in\{0, ..., n-1\}$  such that $f(t_i)\ne f(t_{i+1})$.  Now using \eqref{e4.29}, we obtain
\begin{equation*}
d_{X}(t_i, t_{i+1})=d_{w, G}(t_i, t_{i+1})=\inf\{w(F): F\in P_{t_i, t_{i+1}}\}\ge\varepsilon_1,
\end{equation*}
contrary to \eqref{e4.30}. Hence if $X^{i_0}\ne X_{i_0},$ then $X^{i_0}\not\simeq Y_{i_0}$ holds. It follows that $X$ is minimal $\mathfrak{M}$-universal.
\end{proof}

It is well known that every connected metric space is $\varepsilon$-connected for all $\varepsilon>0$. Hence, Theorem~\ref{t4.9} entails the following corollary.
\begin{cor} \label{c4.12}
Let $\mathfrak{M}$  be a non--empty set of connected non--empty metric spaces $Y$. Suppose that all $Y$ are not shifted and pairwise incomparable, then there is a metric $d_X$ on
$$
X=\mathop{\coprod}\limits_{Y\in\mathfrak{M}}
$$
such that $(X, d_X)$ is minimal $\mathfrak{M}$--universal.
\end{cor}

The following corollary can be considered as a special case of Corollary~\ref{c4.12}.

\begin{cor}\label{c4.13}
Let $n$ be a positive integer number and let $\mathcal{N}_{n}$ be the class of all normed $n$-dimensional linear spaces over the field $\mathbb R.$ Then there is a set $\mathfrak{M}\subseteq\mathcal{N}_{n}$
and a disjoint union
$$
X=\mathop{\coprod}\limits_{Y\in \mathfrak{M}}Y
$$
such that $\mathfrak{M}$ is a minimal $\mathcal{N}_{n}$-universal set of metric spaces and $X$ is a minimal $\mathcal{N}_{n}$-universal metric space.
\end{cor}
\begin{proof}
The existence of minimal $\mathcal{N}_{n}$-universal $\mathfrak{M}\subseteq \mathcal{N}_{n}$ was proved in Proposition~\ref{p3.26}. Since every normed vector space over $\mathbb R$ is connected, the existence of a minimal $\mathcal{N}_{n}$-universal disjoint union $\mathop{\coprod}\limits_{Y\in \mathfrak{M}}Y$ follows from Corollary~\ref{c4.12}.
\end{proof}

\section{Metric betweenness and minimal universality}
\label{sect7}

The ternary ``betweenness'' relation was introduced in explicit form by D.~Hil\-bert in \cite{Hil}. In the theory of metric spaces, the notion of ``metric betweenness'', which is essential in the present section first appeared at K.~Menger's paper \cite{Men} in the following form.

\begin{defin}
Let $X$ be a metric space and let $x, y$ and $z$ be points of $X.$ One says that $y$ lies between $x$ and $z$ if
\begin{equation*}
d_{X}(x,z)=d_{X}(x,y)+d_{X}(y,z).
\end{equation*}
\end{defin}

The ``betweenness'' relation thus defined is fundamental for the theory of geodesics on metric spaces (see, e. g., \cite{Papad}), and it is naturally arises in the studies of the best approximations in metric spaces \cite{DS}.

Characteristic properties of the ternary relations that are ''metric betweenness'' relations for (real-valued) metrics were determined by A.~Wald in \cite{Wald}. Later, the problem of ''metrization'' of betweenness relations (not necessarily by real-valued metrics) was considered in  \cite{Mendris}, \cite{Mos} and \cite{Sim}. An infinitesimal version of the metric betweenness was obtained in \cite{BD} and \cite{DD}.

Let $X$ be a metric space. It is easy to verify that, for every three points $x,y,z\in X,$ the equality
\begin{equation}\label{e5.1}
2\max\{d_{X}(x,y), d_{X}(x,z), d_{X}(y,z)\}=d_{X}(x,y)+d_{X}(x,z)+d_{X}(y,z)
\end{equation}
holds if and only if one of these points lies between the other two points. The necessary and sufficient condition for \eqref{e5.1} is the equality of the Cayley-Menger determinant to zero:
$$
\textrm{det}\begin{vmatrix}
    0&d_{X}^{2}(x,y)&d_{X}^{2}(x,z)&1\\
    d_{X}^{2}(y,x)&0&d_{X}^{2}(y,z)&1\\
    d_{X}^{2}(z,x)&d_{X}^{2}(z,y)&0&1\\
    1&1&1&0\\
\end{vmatrix}=0,
$$
(see, e.g.,  \cite[p. 290]{Berger}).

Let $\mathfrak{M}\mathfrak{B}$ be the class of metric spaces $X$, with points $x,y,z$ satisfying~\eqref{e5.1}. Thus $X\in\mathfrak{M}\mathfrak{B}$ if and only if, among any three points of $X$, there exists a point that lies between the others.

Let $\mathbb R$ denote the real line with the standard metric $d_{\mathbb R}(x,y)=|x-y|.$ It is clear that $\mathbb R\in\mathfrak{M}\mathfrak{B}.$

\begin{defin}\label{d5.2*}
A metric space $X$ is called a pseudolinear quadruple if $X\not\hookrightarrow \mathbb R$ and $A\hookrightarrow \mathbb R$ for every $A \subseteq X$ with $|A|\leqslant 3$.
\end{defin}

\begin{lem}\label{5.2**}
A metric space $X$ is a pseudo-linear quadruple if and only if $|X|=4$ and the points of $X$ may be labeled $p_0, p_1, p_2, p_3$ so that
$$
d_X(p_0,p_1)=d_X(p_2,p_3), \quad d_X(p_1,p_2)=d_X(p_3,p_0)
$$
and
$$
d_X(p_0,p_2)=d_X(p_0,p_1)+ d_X(p_1,p_2)=d_X(p_1,p_3).
$$
\end{lem}
Definition~\ref{d5.2*} and Lemma~\ref{5.2**} are the one-dimensional cases of much more general Definition 44.1 and Theorem 45.4 of Blumenthal's book~\cite{Bl} dealing with isometric embeddings of semimetric spaces in the finite dimensional Euclidean spaces $E^n$ with an arbitrary $n\geqslant 1$. An elementary proof of Lemma~\ref{5.2**} can be found in~\cite{DD}.

Write
$$
\mathfrak{M}\mathfrak{B}^{5}:=\{X\in\mathfrak{M}\mathfrak{B}: |X|\ge 5\}.
$$

\begin{prop}\label{pr1}
The space $\mathbb R$ is minimal $\mathfrak{M}\mathfrak{B}^{5}$-universal. If $X$ is universal for $\mathfrak{M}\mathfrak{B}^{5}$ and $X\in\mathfrak{M}\mathfrak{B}^{5}$ or if $X$ is minimal $\mathfrak{M}\mathfrak{B}^{5}$-universal, then $X\simeq \mathbb R.$
\end{prop}
\begin{proof}
The universality of $\mathbb R$ for $\mathfrak{M}\mathfrak{B}^{5}$ follows from Lemma~\ref{5.2**}.
Since $\mathbb R$ is boundedly compact, homogeneous and belongs to $\mathfrak{M}\mathfrak{B}^{5},$ Theorem~\ref{mainth3*} implies that $\mathbb R$ is minimal $\mathfrak{M}\mathfrak{B}^{5}$-universal and $\mathbb R\simeq X$ for every $\mathfrak{M}\mathfrak{B}^{5}$-universal $X\in\mathfrak{M}\mathfrak{B}^{5}$ and every minimal $\mathfrak{M}\mathfrak{B}^{5}$-universal $X.$
\end{proof}

Let $\mathfrak{F}_2$ be the class of all metric spaces $X$ with $|X|\leqslant 2$. The class $\mathfrak{F}_2$ admits a minimal universal set $\mathfrak{F}_{2}^{1}\subseteq \mathfrak{F}_2$. We can construct such set by the rule: a metric space $Y$ belongs to $\mathfrak{F}_{2}^{1}$ if and only if  there is $t\in (0, \infty)$ such that $Y=\{0,t\}$. Thus, by definition, we have
$$
\mathfrak{F}_{2}^{1}=\{Y_t: t\in (0,\infty)\}
$$
where $Y_t=\{0,t\}$.
It is evident that every $Y_t\in \mathfrak{F}_{2}^{1}$ is not shifted and $Y_{t_1}\not\hookrightarrow Y_{t_2}$ if $t_1\neq t_2$.

\begin{prop}\label{p5.3}
Let $T$ be a non--empty subset of $(0,\infty)$ and let
$$
\mathfrak{A}=\{Y_t : t\in T\} \subseteq \mathfrak{F}_{2}^{1}.
$$
The following conditions are equivalent.
\begin{itemize}
\item [(i)] For every $Z\in \mathfrak{A}$ there is $Y\in \mathfrak{F}_{2}^{1} \setminus \mathfrak{A}$ such that
$$
\operatorname{diam} Z< \operatorname{diam} Y.
$$
\item [(ii)]There is a disjoint union
$$
X=\coprod\limits_{t\in T}Y_t
$$
with a metric $d_X$ such that $(X,d_X)$ is minimal $\mathfrak{A}$-universal.
\end{itemize}
\end{prop}
\begin{proof}
Let (i) hold. It is easy to prove that (i) holds if and only if the set $(0,\infty) \setminus T$ is not bounded. If $T$ is bounded, then all $Y_t$, $t\in T$ are $\varepsilon$-connected metric space with
$$
\varepsilon =\sup\limits_{t\in T}(\operatorname{diam} Y_t).
$$
In this case condition (ii) follows from Theorem~\ref{t4.9}. If $T$ and $(0,\infty) \setminus T$ are both not bounded, then there is a sequence $(p_n)_{n\in \mathbb{N}}$ such that $$ p_1=0, \, p_n\in (0,\infty)\setminus T\quad\mbox{and}\quad (p_{n-1},p_n)\cap T\neq \varnothing$$ for every $n\geqslant 2$ and
$$
\bigcup\limits_{n=2}^{\infty}T\cap (p_{n-1},p_n)=T.
$$
Let $\{X_t: t\in T\}$ be a disjoint family of metric spaces with $X_t\simeq Y_t$ for every $t\in T.$ Write
$$
X=\bigcup\limits_{t\in T} X_t.
$$
Let us define a function $d_X: X\times X\rightarrow\mathbb R^{+}$ as
\begin{equation}\label{eq5.2}
d_X(x,y)=
\begin{cases}
d_{X_t}(x,y) &\text{if there is }\, t\in T \, \text{ such that } \, x,y \in X_t, \\
p_n &\text{if } \, x\in X_{t_1}, \, \ y\in X_{t_2}, \, \ t_1 \neq t_2 \, \\ &\text{and } \, \max\{t_1,t_2\}\in (p_{n-1}, p_n).
\end{cases}
\end{equation}
The function $d_X$ is symmetric and non-negative and, moreover,
$$
(d_X(x,y)=0)\Longleftrightarrow(x=y)
$$
holds for all $x,y\in X$.  We claim also that the strong triangle inequality
$$
d_X(x,y)\leqslant \max\{d_X(x,z),d_X(z,y)\}
$$
is valid for all $x,y,z \in X$. Indeed, if there is $t\in T$ such that $x,y,z\in X_t$, then there is nothing to prove. In the case when $x\in X_{t_1}$, $y\in X_{t_2}$, $z\in X_{t_3}$ and
$$
\min\{t_1, t_2, t_3\}\neq \max \{t_1, t_2, t_3\}\quad\mbox{and}\quad\max\{t_1, t_2, t_3\}\in (p_{n-1}, p_n)
$$
we have
$$
d_X(x,y)\leqslant p_n \,\,\mbox{ and}\,\, \max\{d_X(x,z), d_X(z,y)\}=p_n.
$$
Thus, ~\eqref{eq5.2} holds for all $x,y,z \in X$. Consequently $d_X$ is an ultrametric on $X$.

Since $p_n\in (0,\infty) \setminus T$ for every $n\ge 2$,  it follows from ~\eqref{eq5.2} that, for every $t\in T,$ there is a unique two-point set $\{x,y\}\subseteq X$ such that $d_X(x,y)=t$. In the correspondence with Theorem~\ref{t4.7} the disjoint union $X=\coprod\limits_{t\in T}Y_t$ with the metric $d_X$ is a minimal $\mathfrak A$-universal metric space. The implication (i)$\Rightarrow$(ii) follows.

Suppose now that (ii) holds but (i) does not hold. As was noted above, (i) holds if and only if the set $(0,\infty) \setminus T$ is not bounded.
Consequently there is $t_0\in (0,\infty)$ such that
\begin{equation}\label{eq5.4}
(0,\infty) \setminus T \subseteq (0,t_0).
\end{equation}
Let $t_1 \in (2t_0,\infty)$. Inclusion~\eqref{eq5.4} implies that $t_1\in T$. Hence  $Y_{t_1}\in \mathfrak A$ holds. In the correspondence with the supposition, a disjoint union $X=\coprod\limits_{t\in T}Y_t$ with a metric $d_X$ is a minimal $\mathfrak{A}$-universal metric space. Using~\eqref{eq5.4} and Theorem~\ref{t4.7} we see that for every $t'>t_0$ there is a unique two-point set $\{x',y'\}\subseteq X$ such that $d_X(x',y')=t'$ and, moreover, if $t''>t_0$ and
$$
\{x'',y''\}\subseteq X, \quad d_X(x'',y'')=t'',
$$
then
\begin{equation}\label{eq5.5}
\{x'y'\}\cap\{x'',y''\}=\varnothing.
\end{equation}
Note that~\eqref{eq5.5} follows from the above mentioned uniqueness and Definition~\ref{part}. Let $\{x_1,y_1\}\subseteq X$ and $d_X(x_1,y_1)=t_1$.  Let $z_1$ be an arbitrary point of $X\setminus \{x_1,y_1\}$. From the triangle inequality it follows that we have at least one from the inequalities
\begin{equation}\label{eq5.6}
d_X(x_1, z_1)\geqslant \frac12t_1\,\,\mbox{or}\,\,d_X(y_1, z_1)\geqslant \frac12t_1.
\end{equation}
Without loss of generality, we can suppose that the first inequality of \eqref{eq5.6} holds. Then we have
$$
\{x_1, y_1\}\subseteq X, \quad \{x_1,z_1\}\subseteq X
$$
and
$$
d_X(x_1,y_1)\in T, \quad d_X(x_1, z_1)\in T,
$$
and
$$
\{x_1,y_1\}\cap\{x_1,z_1\}=\{x_1\}
$$
contrary to~\eqref{eq5.5}. The implication (ii)$\Rightarrow$(i) follows.
\end{proof}

\begin{cor}\label{c5.4}
The set $\mathfrak{F}_{2}^{1}$ has the following properties.
\begin{itemize}
\item [(i)] Every $Y\in \mathfrak{F}_{2}^{1}$ is a not shifted metric space and if \, $Y_{t_1}$ and $Y_{t_2}$ are distinct, then $Y_{t_1}\not\hookrightarrow Y_{t_2}$.
\item [(ii)] Any disjoint union $\coprod\limits_{Y\in\mathfrak{F}_{2}^{1} }Y$ is a universal but not minimal universal metric space for $\mathfrak{F}_{2}^{1}$.
\item [(iii)] There are subsets $\mathfrak{A}_1$ and $\mathfrak{A}_2$ of $\mathfrak{F}_{2}^{1}$ such that
$$
\mathfrak{A}_1\cap\mathfrak{A}_2 =\varnothing,\quad  \mathfrak{A}_1 \cup \mathfrak{A}_2 = \mathfrak{F}_{2}^{1}
$$
and the disjoint unions
$$
\coprod\limits_{Y\in \mathfrak{A}_1}Y\quad\mbox{and}\quad\coprod\limits_{Y\in \mathfrak{A}_2}Y
$$
are minimal universal metric spaces for $\mathfrak{A}_1$ and $\mathfrak{A}_2$ respectively.
\end{itemize}
\end{cor}

Thus, there is no minimal $\mathfrak{F}_{2}$-universal metric spaces $X$ of the form
\begin{equation*}
X=\coprod_{Y\in\mathfrak{F}_{2}^{1}}Y.
\end{equation*}
Nevertheless, the next simple example of a minimal $\mathfrak{F}_{2}$-universal metric space was constructed by Macej Wozniak in 2008.

\begin{exa}(\cite{Ho2})
The set
$$
X=(-1, 0)\cup \left(\bigcup_{n=1}^{\infty}\{n\}\right)
$$
with the metric induced from $\mathbb R$ is a minimal $\mathfrak{F}_{2}$-universal metric space (see Figure~2). The minimality of $X$ is, in fact, a consequence of the uniqueness of the representation of $x\in (0, \infty)$ in the form
\begin{equation}\label{eq5.7}
x=n+t
\end{equation}
where $n\in\mathbb N\cup\{0\}$ and $t\in [0, 1)$. For the proof note that if \eqref{eq5.7} holds, then the equalities
$$
n=\lfloor x\rfloor \text{ and } t=x-\lceil x\rceil
$$
hold.
\end{exa}
\begin{rem}\label{r6.7'}
An example of some non--separable minimal $\mathfrak{F}_{2}$-universal metric space was constructed by Wlodzimer Holst\'{y}nski in~\cite{Ho2}.
\end{rem}

The next our goal is the building of a minimal $\mathfrak{M}\mathfrak{B}$-universal metric space. Let $\mathfrak{PL}$  be the class of all pseudo-linear quadruples. It is clear that $\mathfrak{PL} \subseteq \mathfrak{MB}$ and before building of a minimal $\mathfrak{MB}$- universal metric space we shall construct a minimal $\mathfrak{PL}$-universal metric space.

\begin{figure}\label{f5}
  \begin{center}
  \includegraphics[width=0.8\textwidth]{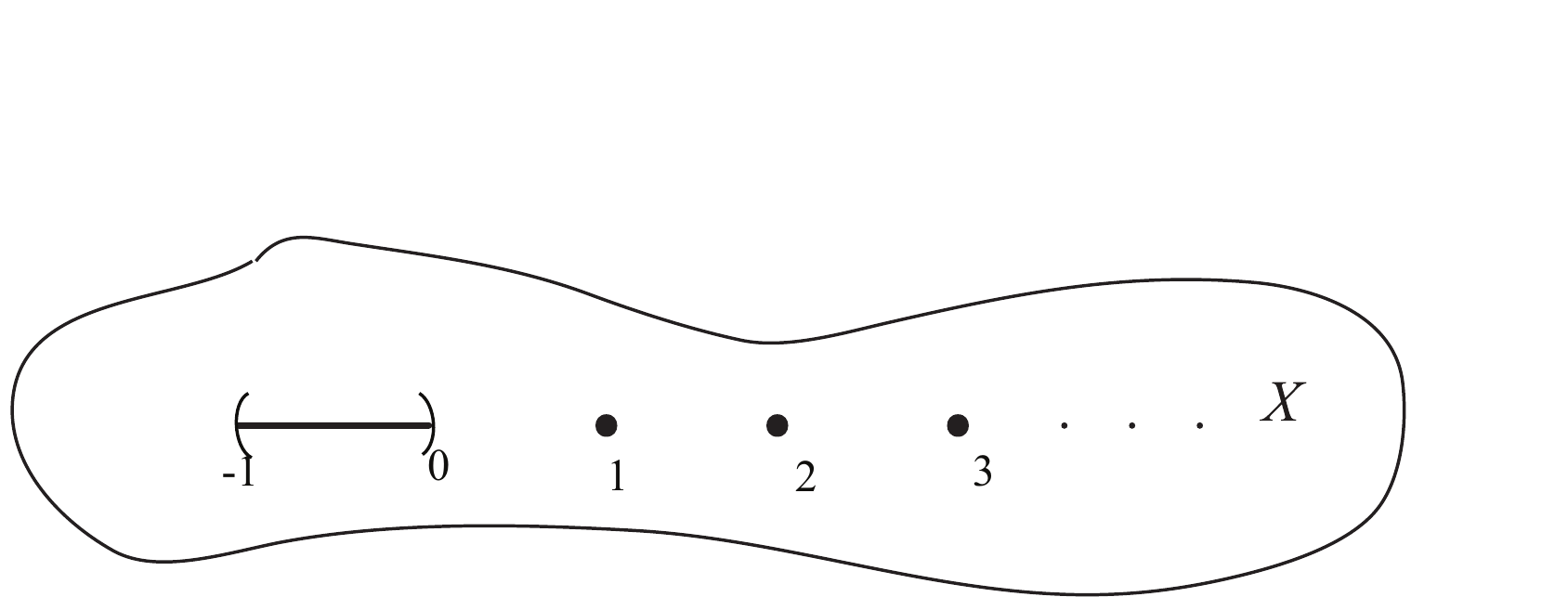}\\
 \caption{The $\mathfrak{F}_{2}$-universal space $\mathbb R$ contains a minimal $\mathfrak{F}_{2}$-universal subspace $X$.}
 \end{center}
\end{figure}

Let $X$ be a metric space and let $A\subseteq X$. We shall say that a two-element subset $\{a,b\}$ of $A$ is a diametrical pair for $A$ if
$$
d_X(a,b) = \operatorname{diam}A
$$
where as usual
$$
\operatorname{diam}A =\sup\{d_X(x,y) : x,y \in A\}.
$$
A point $a\in X$ is a diametrical point for $A$ if $a\in A$ and there is $b\in A$ such that $\{a,b\}$ is a diametrical pair for $A$.

\begin{lem}\label{l5.6}
If $X$ is a pseudo-linear quadruple, then every $x\in X$ is a diametrical point for $X$ and the number of diametrical pairs for $X$ is 2. Furthermore, if $Y$ is also a pseudo-linear quadruple and there are three point sets $\{x_1, x_2, x_3\}\subseteq X$ and $\{y_1, y_2, y_3\}\subseteq Y$ such that
$$\{x_1, x_2, x_3\}\simeq\{y_1, y_2, y_3\},$$
then $X\simeq Y.$
\end{lem}

A proof is immediate from Lemma~\ref{5.2**}.

\medskip

As in the case of the class $\mathfrak{F}_{2},$ there is a minimal $\mathfrak{PL}$-universal disjoint set $\mathfrak{PL}^{1}\subseteq\mathfrak{PL}.$ Let

\begin{equation}\label{eq6.6}
X:=\bigcup_{Y\in\mathfrak{PL}^{1}}Y
\end{equation}
and let $d_{X}: X\times X\to\mathbb R^{+}$ be a function such that
\begin{equation}\label{eq6.7}
d_X(z,y)=
\begin{cases}
d_{Y}(z,y) &\text{if there is }\, Y\in \mathfrak{PL}^{1} \, \text{ with } \, z,y \in Y \\
\textrm{diam} Y \vee \textrm{diam} Z &\text{if there are distinct} \, Y, Z\in \mathfrak{PL}^{1} \, \text{ with } y \in Y, z\in Z.  \\
\end{cases}
\end{equation}

Note that \eqref{eq6.7} is correct because $\mathfrak{PL}^{1}$ is disjoint.

\begin{lem}\label{lem6.9}
If $X$ and $d_{X}: X\times X\to\mathbb R^{+}$ are defined as in \eqref{eq6.6} and \eqref{eq6.7} respectively, then $d_{X}$ is a metric on $X.$
\end{lem}
\begin{proof}
It suffices to show that the strong triangle inequality
\begin{equation}\label{eq6.8}
d_{X}(y, z)\le d_{X}(y, w)\vee d_{X}(w, z)
\end{equation}
holds if $y\in Y,\, z\in Z, \,w\in W$ and
$Y, Z, W\in\mathfrak{PL}^{1}$ and at least one from the conditions
\begin{equation*}
Y\ne Z, \quad Y\ne W \quad \text{and}\quad Z\ne W
\end{equation*}
is valid. Under these conditions there exist the next four possibilities:

\medskip

(i) $Y\ne Z$ and $Y\ne W$ and $Z\ne W;$

(ii) $Y\ne Z$ and $Z= W;$

(iii) $Y\ne Z$ and $Y= W;$

(iv) $Y= Z$ and $Z\ne W.$

In the case when (i) holds, inequality \eqref{eq6.8} is evident. Suppose we have~(ii). Then, the equalities
$$
\textrm{diam} Y \vee \textrm{diam} Z=d_{X}(y, z)
$$
and
$$
\textrm{diam} Y \vee \textrm{diam} Z=\textrm{diam} W \vee \textrm{diam} Z=d_{X}(y, w)
$$
hold that implies \eqref{eq6.8}. Case (iii) is similar to (ii). Let (iv) be valid. Then using \eqref{eq6.7}, we obtain
\begin{equation*}
d_{X}(y, z)\le d_{Z}(y, z)\le \textrm{diam} Z.
\end{equation*}
Moreover, from $Z\ne W$ and \eqref{eq6.7} it follows that
\begin{equation*}
d_{X}(z, w)= \textrm{diam} Z \vee \textrm{diam} W.
\end{equation*}
Inequality \eqref{eq6.8} follows.
\end{proof}

\begin{rem}\label{r6.9*}
The metric $d_X$ in the above lemma is not an ultrametric.
\end{rem}

\begin{prop}\label{pr6.10}
If $X$ and $d_{X}: X\times X\to\mathbb R^{+}$ are defined as in \eqref{eq6.6} and \eqref{eq6.7} respectively, then $(X, d_X)$ is a minimal $\mathfrak{PL}$-universal metric space.
\end{prop}

\begin{proof}
It follows directly from \eqref{eq6.7} and Lemma~\ref{lem6.9}, that $(X, d_X)$ is a disjoint union of $Y\in\mathfrak{PL}^{1}$. Consequently, by Theorem~\ref{t4.7}, $(X, d_X)$ is minimal $\mathfrak{PL}^{1}$-universal if and only if, for every $Y\in\mathfrak{PL}^{1},$ there is a unique
$X_{Y}\subseteq X $ such that $X_Y\simeq Y.$ The last condition holds if for all distinct $y_i\in Y_i, \, i=1, ..., 4$ with $\{y_1, y_2, y_3, y_4\}\in\mathfrak{PL}$ we have
\begin{equation}\label{yk}
Y_1=Y_2=Y_3=Y_4.
\end{equation}
Suppose $y_i \in Y_i$, $i=1, \ldots, 4$ but  \eqref{yk} does not hold. Then, as was shown in the proof of Lemma~\ref{lem6.9} there are three distinct points $x_i, x_j, x_k\in\{x_1, ..., x_n\},$ which form an isosceles triangle with the base that is not longer than its legs. Hence the number of diametrical pairs of $\{x_{i_1}, x_{i_2}, x_{i_3}\}$ is more or equal two. Consequently, $\{x_{i_1}, x_{i_2}, x_{i_3}\}\not\in\mathfrak{MB}.$ By Definition~\ref{d5.2*}, it follows that $\{x_1, ..., x_4\}\not\in\mathfrak{PL}.$ Thus, $(X, d_X)$ is minimal $\mathfrak{PL}^{1}$-universal and, consequently, minimal $\mathfrak{PL}$-universal.
\end{proof}

Now we are ready to construct a minimal $\mathfrak{MB}$-universal metric space.

Let $p\in\mathbb R$ and $b\in X,$ where
\begin{equation*}
X=\bigcup_{Y\in\mathfrak{PL}^{1}}Y,
\end{equation*}
and let $r\in (0, \infty)$. For simplicity, we may suppose also that $X\cap\mathbb R=\varnothing$. Write $M:=X\cup\mathbb R$ and define a symmetric function $d_{M}: M\times M \rightarrow \mathbb R^{+}$ such that
\begin{equation}\label{eqv6.10}
d_M(x,y)=
\begin{cases}
d_{X}(x,y) &\text{if}\, x,y \in X \\
|x-y| &\text{if}\, x,y \in\mathbb R \\
|x-p|+r+ d_{X}(b,y)& \text{if}\, x\in\mathbb R \, \text{and} \, y\in X.\\
\end{cases}
\end{equation}

\begin{lem}\label{lem6.11}
The function $d_{M}: M\times M\to\mathbb R^{+}$ is a metric on $M$.
\end{lem}

A proof of Lemma~\ref{lem6.11} can be obtained directly from \eqref{eqv6.10}. Note only, that $d_{M}$ is the weighted shortest-path metric for the weighted graph $(G, w)$ with the vertex set
$$
V(G)=M=X\cup\mathbb R,
$$
and the edge set
\begin{equation*}
E(G)=\{\{x, y\}: x, y\in X, x\ne y\}\cup\{\{p, b\}\}\cup\{\{x, y\}: x, y\in\mathbb R, \, x\ne y\},
\end{equation*}
and the weight
\begin{equation*}\label{eq6.10}
w(\{x, y\})=
\begin{cases}
d_{X}(x,y) &\text{if}\, \{x,y\} \subseteq X \\
|x-y| &\text{if}\, \{x,y\} \subseteq\mathbb R \\
r& \text{if}\, \{x, y\}=\{p,b\}.\\
\end{cases}
\end{equation*}

\begin{lem}\label{lem6.12}
Let $x, y\in\mathbb R,$ $x< y$ and let $z\in X$. If $\{x, y, z\}\in\mathfrak{M}\mathfrak{B},$ then $p$ does not belong to the interval $(x,y)=\{t\in\mathbb R: x<t<y\}$.
\end{lem}

\begin{proof}
Suppose $\{x, y, z\}\in\mathfrak{M}\mathfrak{B}.$ Then there exist three possibilities:

\medskip

(i) $z$ lies between $x$ and $y,$

\begin{equation}\label{eq6.11}
d_M(x,y)=d_M(x,z)+d_M(z,y);
\end{equation}

(ii) $y$ lies between $x$ and $z,$

\begin{equation}\label{eq6.12}
d_M(x,z)=d_M(x,y)+d_M(y,z);
\end{equation}

(iii) $x$ lies between $y$ and $z,$

\begin{equation}\label{eq6.13}
d_M(y,z)=d_M(y,x)+d_M(x,z).
\end{equation}

For $p\in (x, y),$ equality \eqref{eq6.11} can be written in the form
\begin{equation}\label{eq6.14}
|x-y|=(|x-p|+r+d_{X}(b,z))+(|y-p|+r+d_{X}(b,z))
\end{equation}
(see Figure~3 below).
The condition $p\in(x,y)$ implies
\begin{equation}\label{eqvr}
|x-p|+|y-p|=|x-y|.
\end{equation}
Hence from \eqref{eq6.14} we obtain
\begin{equation*}
0=2(r+d_{X}(b, z))\ge 2r,
\end{equation*}
contrary to the inequality $r>0.$ If we have \eqref{eq6.12}, then
\begin{equation*}
|x-p|+r+d_{X}(b,z)=|x-y|+|y-p|+z+d_{X}(b,z),
\end{equation*}
so that
\begin{equation*}
|x-p|=|x-y|+|y-p|.
\end{equation*}
The last equality and \eqref{eqvr} imply
\begin{equation*}
|x-p|=|x-p|+2|y-p|.
\end{equation*}
Hence we obtain $|y-p|=0.$ It is a contradiction with $p\in (x,y)$.

Case (iii) can be considered analogously to case (ii). Thus, $p\in (x,y)$ contradicts the supposition $\{x, y, z\}\in\mathfrak{M}\mathfrak{B}.$
\begin{figure}\label{f2}
\begin{center}
  \includegraphics[width=0.8\textwidth]{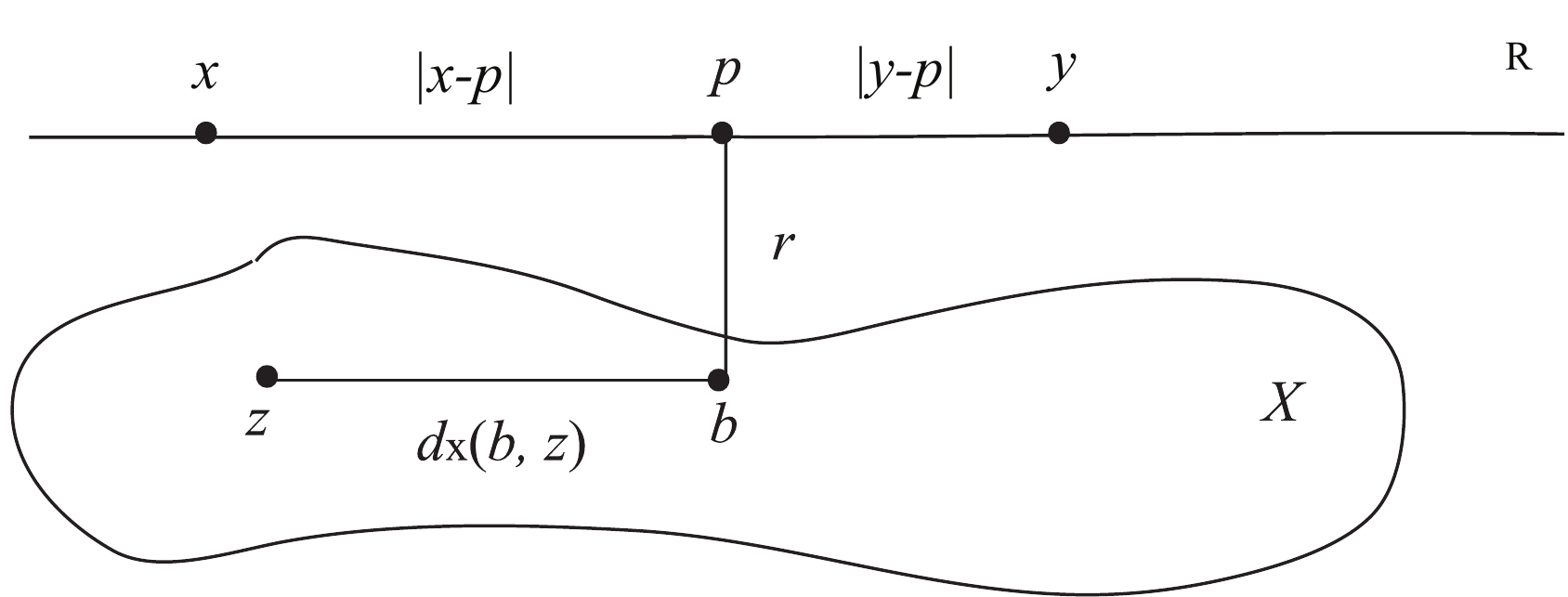}\\
\caption{\hspace{1mm}}
\end{center}
\end{figure}
\end{proof}

\begin{thm}
The metric space $(M, d_M)$ is minimal $\mathfrak{M}\mathfrak{B}$-universal.
\end{thm}
\begin{proof}
From Lemma~\ref{lem6.11} and the definitions of $d_M$ and $X$, it follows that
\begin{equation*}
M=X\sqcup\mathbb R=\left(\coprod_{Y\in\mathfrak{PL}^{1}}Y\right)\sqcup\mathbb R.
\end{equation*}
Hence, $\{\mathbb R\}\cup\mathfrak{PL}\hookrightarrow M$. Moreover, using Definition~\ref{d5.2*} and Lemma~\ref{5.2**}, we see that, for every $Z\in\mathfrak{MB}$, either $Z\hookrightarrow\mathbb R$ or there is $Q\in\mathfrak{PL}$ such that $Q\simeq Z$. Consequently, $\mathfrak{M}\mathfrak{B}\hookrightarrow M$ and the set $\mathfrak{PL}^{1}\cup\{\mathbb R\}$ is a minimal $\mathfrak{M}\mathfrak{B}$-universal subset of $\mathfrak{M}\mathfrak{B}.$ Hence, it suffices to show that $M$ is a minimal universal metric space for $\mathfrak{PL}^{1}\cup\{\mathbb R\}.$
Condition (iii) of Theorem~\ref{t4.7} implies that $M$ is minimal universal for $\mathfrak{PL}^{1}\cup\{\mathbb R\}$ if and only if for every $Z\in\mathfrak{PL}^{1}\cup\{\mathbb R\}$ there is a unique $S_{Z}\subseteq M,$ such that $S_Z \simeq Z.$ Indeed, all $Z\in\mathfrak{PL}^{1}$ are not shifted as finite metric spaces and $\mathbb R$ is not shifted by Lemma~\ref{homsh}. Moreover, it follows from Definition~\ref{d5.2*}, Lemma~\ref{5.2**} and the definition of $\mathfrak{PL}^{1}$ that any distinct $Z, W\in\mathfrak{PL}^{1}\cup\{\mathbb R\}$ are incomparable. Thus, conditions (i) and (ii) of Theorem~\ref{t4.7} hold and it remains to prove condition (iii). Let $S\subseteq M$ and $S\simeq\mathbb R.$ We must show that $S=\mathbb R.$ Since $\mathbb R$ is not shifted, it suffices to prove the inclusion $S\subseteq\mathbb R.$ Suppose contrary that $S\cap X\ne\varnothing,$ where $X=\mathop{\coprod}\limits_{Y\in\mathfrak{PL}^{1}}Y.$ Note that every point $x_0\in X$ is an isolated point of $M.$ Indeed, the inequality $$d_{M}(x_0, y)\ge r$$ holds if $y\in\mathbb R$ (see \eqref{eqv6.10}) or we have the inequality
$$
d_{M}(x_0, y)\ge \textrm{diam} Z_{0},
$$
if $x_0\in Z_0\in\mathfrak{PL}^{1}$ and $y\in Z_1\in\mathfrak{PL}^{1}$, $Z_1\ne Z_0$ (see \eqref{eq6.7}). Moreover,
\begin{equation}\label{eq6.15}
d_{M}(x_0, y) \geq \min\{d_{M}(x, y): x, y\in Z_0, x\ne y\}>0,
\end{equation}
holds if $x_0, y\in Z_0\in\mathfrak{PL}^{1},$ $x_0\ne y$. (Note that the second inequality in \eqref{eq6.15} holds, because $Z_0$ is finite). Hence, if $S\cap X=\varnothing,$ then $S$ contains an isolated point, that is impossible in the case $S\simeq\mathbb R.$

Suppose now that $S\subseteq M$ and $S\simeq Z$, with $Z\in\mathfrak{PL}^{1}.$ If $S\subseteq X,$ then the equality
\begin{equation}\label{6.16}
S=Z
\end{equation}
was, in fact, proved in the proof of Proposition~\ref{pr6.10}. The inclusion $S\subseteq\mathbb R$ does not take place since $S\simeq Z$ and $Z\not\hookrightarrow\mathbb R$. Consequently, we have
$$
S\cap X\ne \varnothing\,\text{ and } S\cap\mathbb R\ne\varnothing.
$$
Since $|S|=4$ and $S\cap X\ne \varnothing$, we can consider the following possibilities
\begin{equation*}
|S\cap\mathbb R|=3, \, \, |S\cap\mathbb R|=2\,\,\mbox{or}\,\, |S\cap\mathbb R|=1.
\end{equation*}

\textbf{Case $|S\cap\mathbb R|=3$}. Let $S\cap\mathbb R=\{x, y, s\},$ $x< y,$ $S\cap X=\{z\}$ and let
$$
d_{M}(x, y)=\textrm{diam} (S\cap\mathbb R).
$$
Since $S\simeq Z$ and $Z\in\mathfrak{PL}^{1},$ the statement $\{x, y, z\}\in\mathfrak{M}\mathfrak{B}$ holds. Consequently, by Lemma~\ref{lem6.12}, we have $p\not\in(x, y)$ (see Figure~3).
For example, if $x<y\le y,$ we obtain
$$
\textrm{diam}\{x, y, s\}=d_{M}(x, y)=|x-y|< |x-y|+d_{M}(y, p)+r+d_{X}(b, z)\le \textrm{diam} S.
$$
Thus,
\begin{equation}\label{eq6.17}
\textrm{diam}\{x, y, s\}<\textrm{diam} S.
\end{equation}
Now note that Lemma~\ref{5.2**} implies the equality
\begin{equation*}
\textrm{diam}\{x, y, s\}=\textrm{diam} S,
\end{equation*}
contrary to \eqref{eq6.17}.

\textbf{Case $|S\cap\mathbb R|=2.$} Let $S\cap\mathbb R=\{x, y\}$ and $S\cap X=\{s, z\}.$ As in the case $|S\cap\mathbb R|=3,$ we can show that $\{x, y\}$ is not a diametrical pair for $S.$ Moreover, using Figure~4, it is easy to see that $x$ or $y$ is not a diametrical point for $X,$ contrary to Lemma~\ref{l5.6}.

\textbf{Case $|S\cap\mathbb R|=1.$} Let $S\cap\mathbb R=\{x\}$ and $S\cap X=\{y, s, z\}.$ Since
$S\in\mathfrak{PL}^{1},$ we have
\begin{equation}\label{eq6.18}
\{y,s,z\}\in\mathfrak{M}\mathfrak{B}.
\end{equation}

As was shown in the proof of Proposition~\ref{pr6.10}, statement \eqref{eq6.18} can be valid if and only if there is $W\in\mathfrak{PL}^{1}$ such that
\begin{equation*}
\{y, s, z\}\subseteq W,
\end{equation*}
(see \eqref{eqv6.10}). By Lemma~\ref{l5.6}, we obtain that $W\simeq S.$ Without loss of generality, we may suppose also that
\begin{equation}\label{eq6.22}
\textrm{diam} S=d_{X}(y, z).
\end{equation}
If $b\not\in W$, then it follows that
\begin{equation*}
d_{M}(y, z)=d_{X}(y, z)\le d_{X}(z, b)<d_{X}(z, b)+r+|p-x|=d_{M}(z, x).
\end{equation*}
Hence, $d_{M}(y, z)<d_{M}(z, x)$ holds, contrary to \eqref{eqv6.10}. Consequently, $b\in W$ holds. If $z=b,$ then
\begin{equation*}
d_{M}(y, z)=d_{M}(y, b)< d_{X}(y, b)<d_{X}(z, b)+r+|p-x|=d_{M}(y, x),
\end{equation*}
that contradicts \eqref{eq6.22}. Hence, we obtain $z\ne b.$ Similarly we have $s\ne b$. Suppose $y=b$. Since $\{y, z\}$ is a diametrical pair for $S$, Lemma~\ref{l5.6} implies that $\{x, s\}$ is also a diametrical pair for $S$. Using this lemma again, we obtain the inequality
\begin{equation}\label{eq6.23}
d_{M}(x, s)> d_{M}(x, y).
\end{equation}
From the definition of $d_M$, it follows that
\begin{equation*}
d_{M}(x, s)=d_{M}(x, b)=|x-p|+r \text{ and }
d_{M}(x, y)=|x-p|+r+d_{M}(b, y),
\end{equation*}
\begin{figure}
  \begin{center}
  \includegraphics[width=0.8\textwidth]{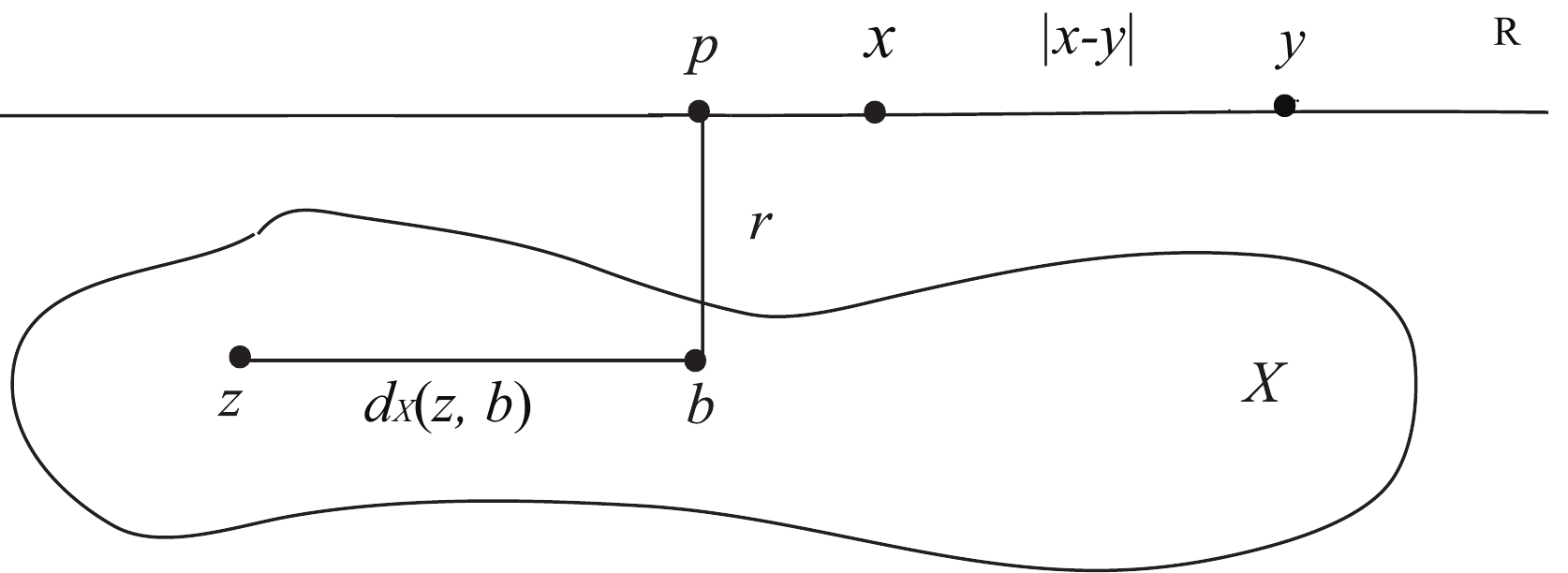}\\
 \end{center}
\end{figure}
\begin{figure}\label{f3}
  \begin{center}
  \includegraphics[width=0.8\textwidth]{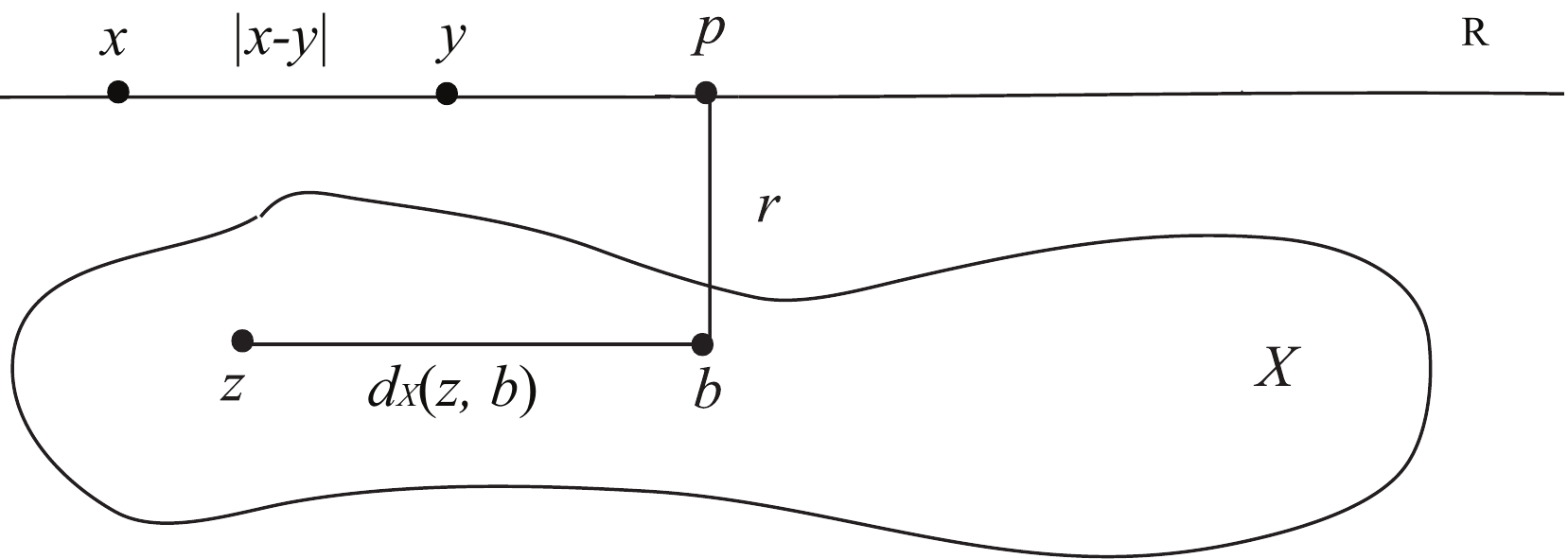}\\
 \caption{\ }
 \end{center}
\end{figure}
(see Figure~5). Thus, $d_{M}(x, s)<d_{M}(x, y),$ contrary to \eqref{eq6.23}. Let us denote by $t$ the fourth point of the pseudo-linear quadruple $W,$ $W=\{y, s, z, t\}.$ Since $b\not\in\{y, z, s\}$ and $b\in W,$ we have $t=b.$ Since $S$ and $W$ are pseudo-linear quadruples and
$$
\{y, s, z\}\subseteq W\cap S.
$$
Lemma~\ref{l5.6} implies that $W\simeq S.$ In particular, the equalities
\begin{equation*}
d_{M}(t, s)=\textrm{diam}W=\textrm{diam}S=d_{M}(s, x)
\end{equation*}
hold.
\begin{figure}\label{f6}
  \begin{center}
  \includegraphics[width=0.8\textwidth]{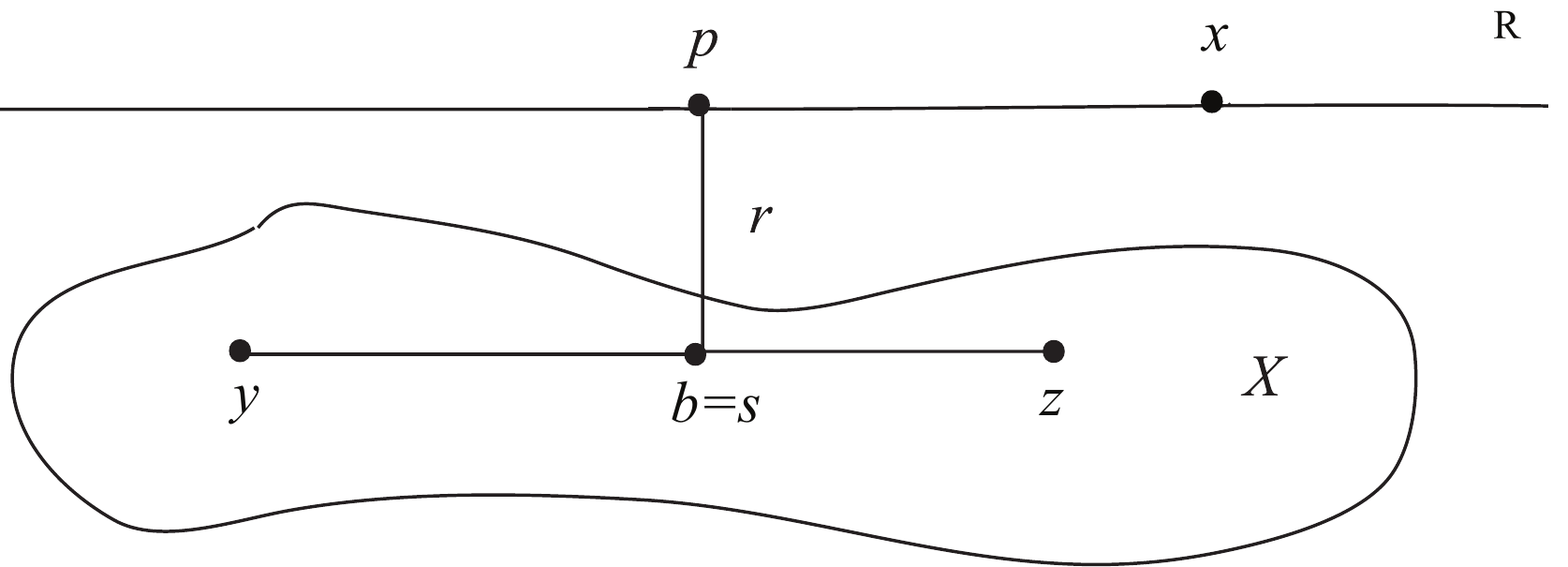}\\
 \caption{\ }
 \end{center}
\end{figure}
It follows that
\begin{equation*}
d_{M}(t, s)=d_{M}(s, x).
\end{equation*}
Now note that the last equality cannot be hold, because
The last equality does not hold
\begin{equation*}
d(x, s)=|x-p|+r+d_{M}(b, s),
\end{equation*}
(see Figure~6).
\begin{figure}\label{f7}
  \begin{center}
  \includegraphics[width=0.8\textwidth]{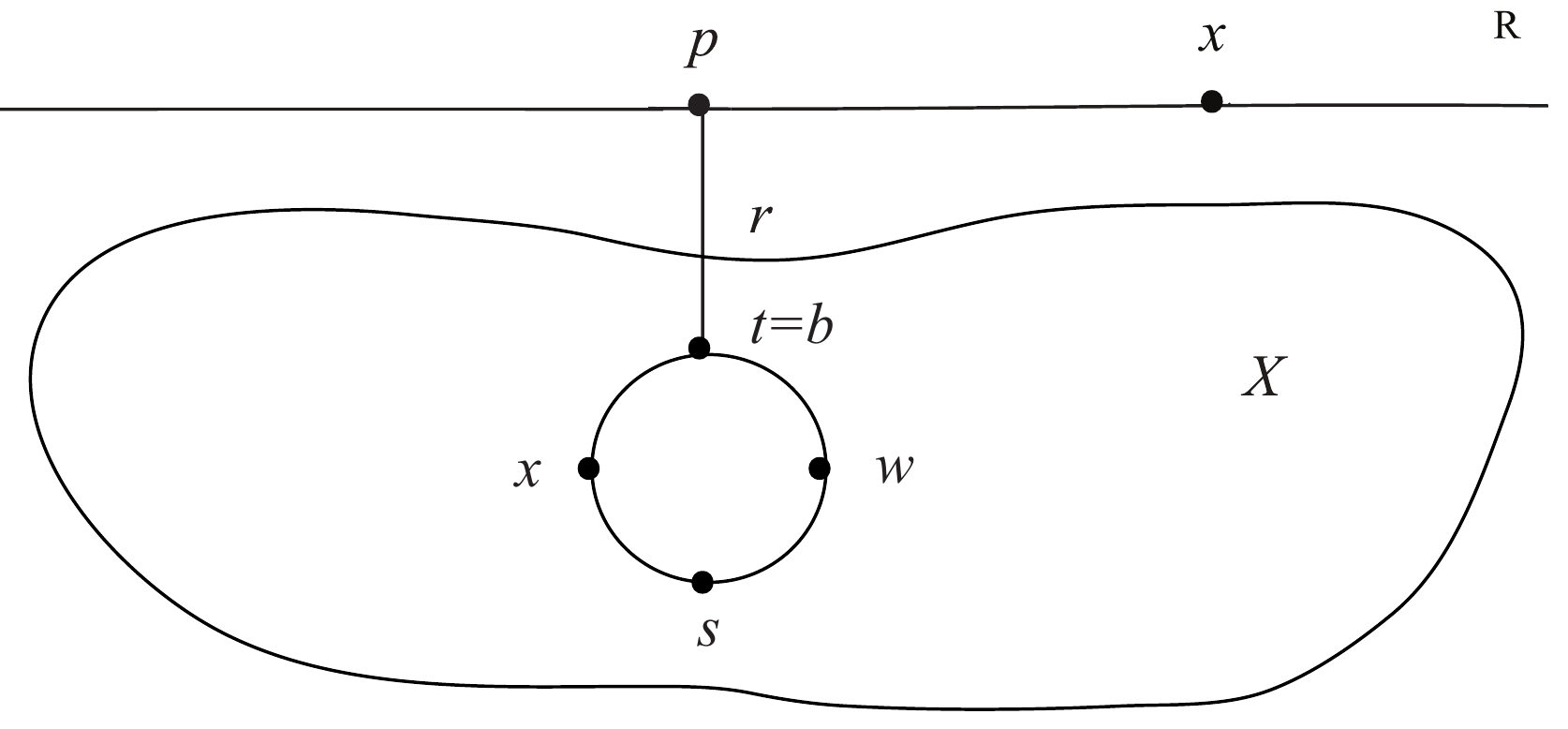}\\
 \caption{\ }
 \end{center}
\end{figure}
\end{proof}

The following proposition gives us an example of a class $\mathfrak{M}$ of metric spaces for which all minimal $\mathfrak{M}$--universal metric spaces are isometric.

\begin{prop}\label{Prop6.16}
Let $\mathfrak{M}$ be the set of all closed intervals $[a,b] \subseteq \mathbb R$ with $|a-b|<1$ and let $I$ be the open interval $(0,1)\subseteq \mathbb R$. Then $I$ is a minimal $\mathfrak{M}$-universal metric space. Moreover, if $X$ is an arbitrary minimal $\mathfrak{M}$-universal metric space, then $X \simeq I$.
\end{prop}
\begin{proof}
Let $X$ be a minimal $\mathfrak{M}$-universal metric space and let $p\in X$. Then for every $t \in (0,1)$ there is $f_t: [0,t] \hookrightarrow X$. Write $X_t = f_t([0,t])$. We claim that there exists $t_0 \in (0,1)$ such that the statement
$$
p \in X_t
$$
holds for all $t\in [t_0,1)$. Suppose the contrary and choose a sequence $(t_n)_{n \in \mathbb N}$ such that $t_n \in (0,1)$ and $\lim_{n\to\infty} t_n =1$ and $p\notin X_{t_n}$ for every $n \in \mathbb N$. Then
$$
Y = \bigcup_{n\in \mathbb N} X_{t_n}
$$
is a subspace of $X$, $p \notin Y$ and $\mathfrak{M}\hookrightarrow Y$. (The last follows from the equality $\lim_{n\to\infty} t_n =1$.) Hence $X$ is not minimal $\mathfrak{M}$-universal, contrary to the condition.

Now it is easy to show that $X \in \mathfrak{MB}$. Indeed, $X \in \mathfrak{MB}$ if and only if equality~\eqref{e5.1},
$$
2\max\{d_X(x,y), d_X(y,z), d_X(z,x)\} = d_X(x,y)+d_X(y,z)+d_X(z,x)
$$
holds for all $x,y,z \in X$. Using the above proved claim we can find $X_{t_*}$ such that $x,y,z \in X_{t_*}$ for given $x,y,z \in X$. It is clear that $X_{t_*} \in \mathfrak{MB}$. Consequently equality \eqref{e5.1} holds, so that $X \in \mathfrak{MB}$. Proposition \ref{pr1} implies that $X$ is isometric to a subspace $\mathbb R$. The metric space $X$ is connected because $X_t$ is connected for every $t \in (0,1)$ and, for given $x,y \in X$, there is $t_* \in (0,1)$  such that
$$
x,y \in X_{t_*} \subset X.
$$
A subset of $\mathbb R$ is connected if and only if this subset is an interval. Hence $X$ is an interval. Let us denote by $m(X)$ the length of $X$. It is clear that $\mathfrak{M}\not\hookrightarrow X$ if $m(X)<1$ and that $X$ is not minimal $\mathfrak{M}$-universal if $m(X)>1$. Consequently we have $m(X)=1$. Let $a$ and $b$  be the endpoints of $X$. Since $\mathfrak{M} \hookrightarrow X\setminus\{a,b\}$ and $X$ is minimal $\mathfrak{M}$-universal, the interval $X$ is open. Every open interval $(a,b)$ with $|a-b|=1$ is isometric to $I$. Thus $X \simeq I$ holds.

It still remains to note that $I$ is minimal $\mathfrak{M}$-universal. Indeed $\mathfrak{M}\hookrightarrow I$ is immediate. Now if $0<p<1$ and $p<t<1$, then we evidently have $[0,t] \in \mathfrak{M}$ and
$$
[0,t] \not\hookrightarrow I\setminus\{p\}.
$$
Thus $I$ is a minimal $\mathfrak{M}$-universal metric space as required.
\end{proof}

\section{Minimal universal subsets of $\mathbb R^2$ for the class of three--point metric spaces}
\label{sect8}

Let $\mathfrak{F}_3$ be the class of metric spaces $X$ with $|X|\leq 3$. It is a basic fact of the theory of metric spaces that the Euclidean plane $\mathbb R^2$ is a universal metric space for the class $\mathfrak{F}_3$. The main goal of the present section is to construct some minimal $\mathfrak{F}_3$-universal subsets of $\mathbb R^2$. Our first example of such subset is closely related to the so-called Fermat-Torricelli point of a triangle. Recall that $t\in\mathbb R^2$ is the Fermat-Torricelli point of a triangle $\{a,b,c\}$ if the inequality
\begin{equation}\label{e7.12}
    d_{\mathbb R^2}(a,t)+d_{\mathbb R^2}(b,t)+d_{\mathbb R^2}(c,t)\leq d_{\mathbb R^2}(a,x)+d_{\mathbb R^2}(b,x)+d_{\mathbb R^2}(c,x)).
\end{equation}
holds for every $x\in \mathbb R^2$.
The geometric construction of the Fermat-Torricelli points can be found, for example, in Coxeter's book~\cite[p. 21--22]{Co}.

If all angles of a given triangle $\{a,b,c\}$ are smaller than $\frac{2\pi}{3}$, then the Fermat-Torricelli point $t$ satisfies the condition
\begin{equation}\label{e7.2}
    \angle atb =\angle btc=\angle cta=\frac{2\pi}{3}.
\end{equation}
For the convenience of the reader we repeat the following lemma here.
\begin{lem}\label{l7.1}
Let $\{a,b,c\}$ be a triangle in $\mathbb R^2$. The following statements hold.
\begin{itemize}
  \item [(i)] All angles  of $\{a,b,c\}$ are smaller than $\frac{2\pi}{3}$ if and only if there is $t\in \mathbb R^2$ satisfying~(\ref{e7.2}).
  \item [(ii)] If $\mathbb R^2$ contains $t$ satisfying~(\ref{e7.2}), then such $t$ is unique and belongs to the interior of the triangle $\{a,b,c\}$.
\end{itemize}
\end{lem}
\begin{proof}
It suffices to note that for $x\in\mathbb R^2$ the equality $\angle axb =\frac{2\pi}{3}$ holds if and only if $x$ belongs to the symmetric lens $afbf'a$ (without points $a$ and $b$) where the lens is the union of the arcs $afb$ and $af'b$ of two symmetric circles passing through $a$ and $b$ such that the angles made by the chord $ab$ and the tangents $al$ and $al'$ are equal to $\frac{\pi}{3}$ (see Figure~\ref{f7.1}).
\begin{figure}[h]
\begin{center}
\includegraphics[width=0.5\textwidth]{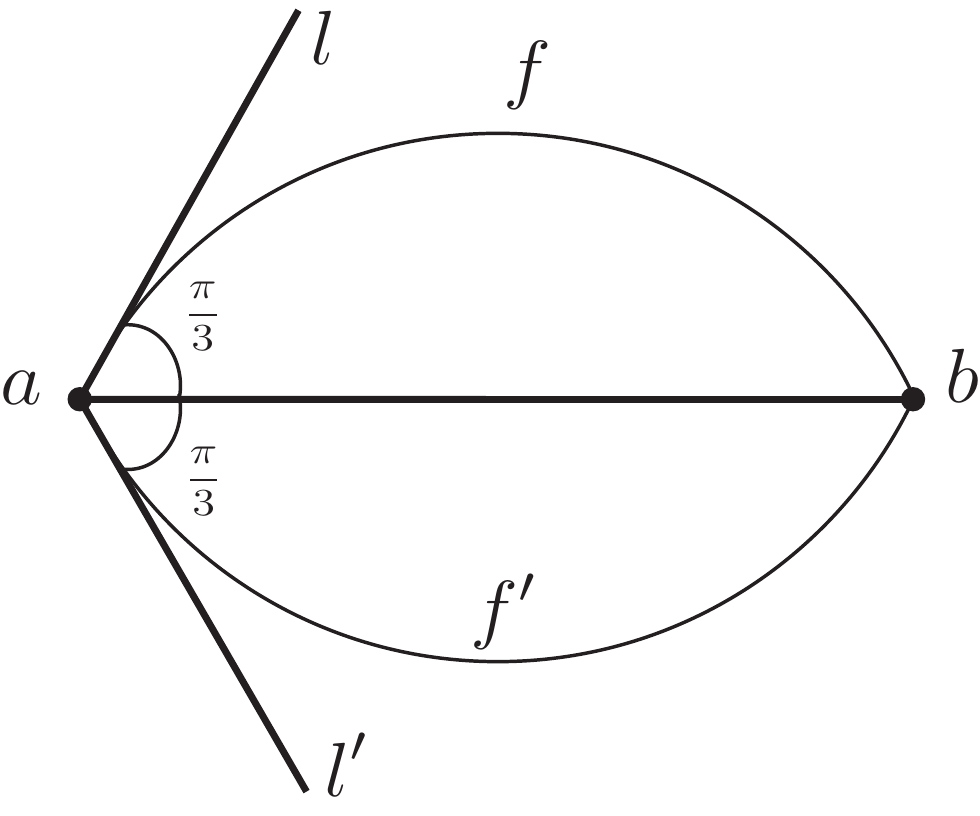}
\end{center}
\caption{The locus of points $x$ with $\angle axb =\frac{2\pi}{3}$.} \label{f7.1}
\end{figure}
Consequently if $\angle abc \geqslant \frac{2\pi}{3}$, then the lenses corresponding to the sides $[a,b]$ and $[b,c]$ meet at the point $b$ only. Otherwise there is a unique point $t$ satisfying~(\ref{e7.2}) and lying in the interior of the triangle $\{a,b,c\}$.
\end{proof}

\begin{thm}\label{t7.2}
Let $X\subseteq \mathbb R^2$ be the union of some rays $\overrightarrow{o\vphantom{b}a}$, $\overrightarrow{ob}$, and $\overrightarrow{o\vphantom{b}c}$ with
$$
\angle aob=\angle boc=\angle coa=\frac{2\pi}{3}.
$$
Then $X$ is minimal $\mathfrak{F}_3$-universal.
\end{thm}
\begin{proof}
It is clear that $Y\hookrightarrow X$ holds for every $Y \in \mathfrak{MB}$ with $|Y|\leqslant 3$. Let a triangle $\{e,f,g\}\notin \mathfrak{MB}$ and let $\angle efg$ be the maximal angle of this triangle. Suppose $\angle efg\geqslant \frac{2\pi}{3}$. We can locate $\{e,f,g\}$ such that $f,g\in \overrightarrow{ob}$ and $e$ lies between
the rays $\overrightarrow{o\vphantom{b}a}$ and $\overrightarrow{ob}$ (see Figure~\ref{r1}). Doing a parallel shift of $\{e,f,g\}$ along the ray $\overrightarrow{ob}$ we can find a position of its image $\{e_1,f_1,g_1\}$ such that $e_1\in \overrightarrow{o\vphantom{b}a}$ (see Figure~\ref{r1}). Thus $\{e,f,g\}\hookrightarrow X$ holds.

\begin{figure}[h]
\begin{center}
\includegraphics[width=0.5\linewidth]{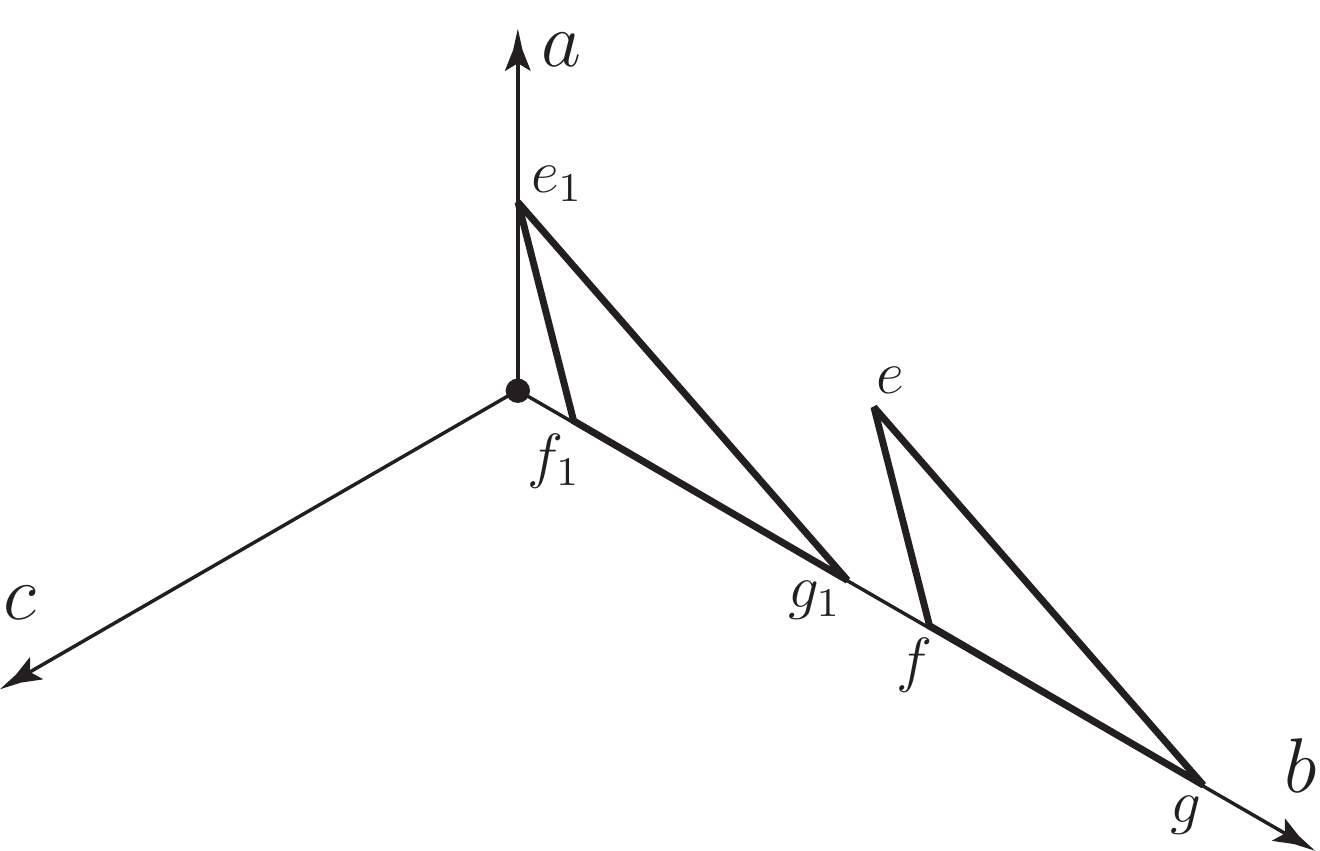}
\caption{}
\label{r1}
\end{center}
\end{figure}

Suppose now that the maximal angle of the triangle $\{e,f,g\}$ is strictly less than $\frac{2\pi}{3}$. Then $\{e,f,g\}\hookrightarrow X$ holds by Lemma~\ref{l7.1} (see Figure~\ref{r2}). Thus $X$ is $\mathfrak{F}_3$-universal.

\begin{figure}[h]
\begin{center}
\includegraphics[width=0.4\linewidth]{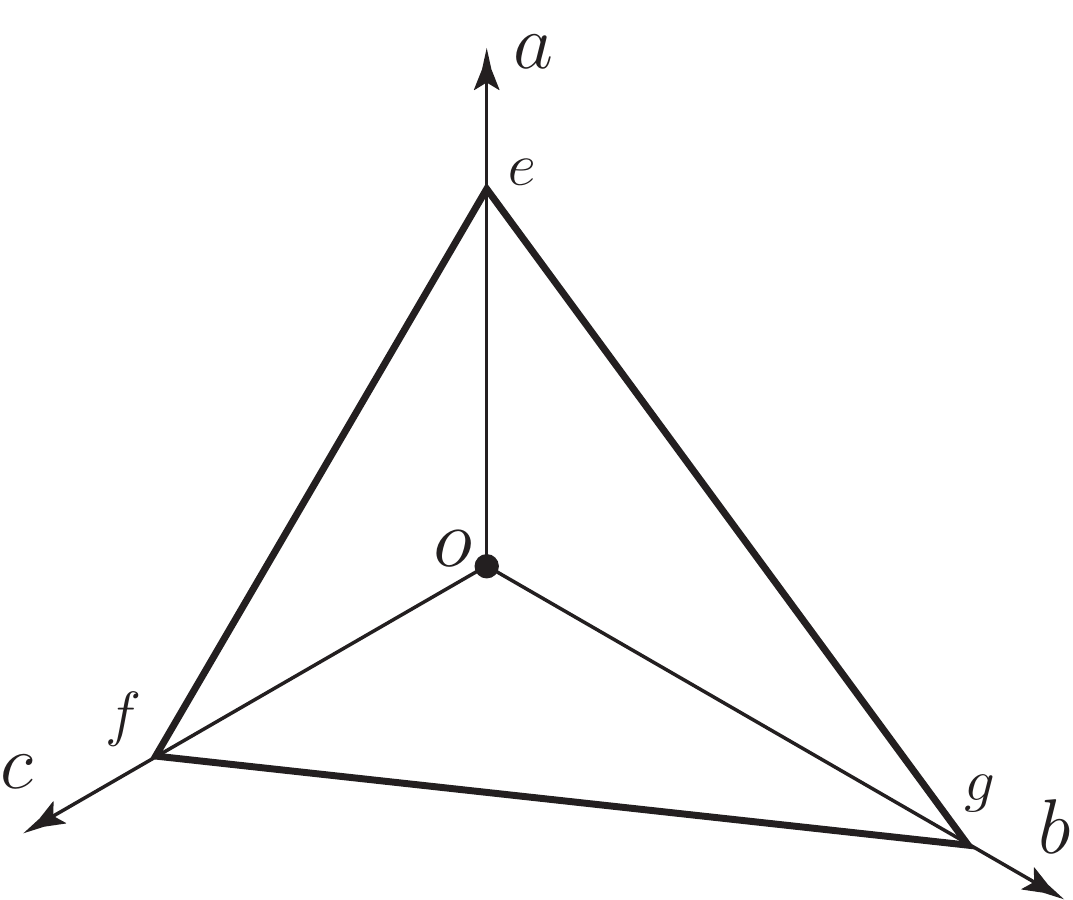}
\caption{}
\label{r2}
\end{center}
\end{figure}

Let us prove the minimality. Let $e$ be a point of $X\setminus \{o\}$. Without loss of generality we may take $e\in \overrightarrow{ob}$ (see Figure~\ref{r3}). Let us consider the points $f\in \overrightarrow{o\vphantom{b}c}$ and $g\in \overrightarrow{o\vphantom{b}a}$ satisfying the equalities
\begin{equation}\label{e7.1}
    d_{\mathbb R^2}(o,f)=d_{\mathbb R^2}(o,g)=d_{\mathbb R^2}(o,e).
\end{equation}
Let $\Psi$ be an arbitrary  isometric embedding of $\{e,f,g\}$ in $X$. We claim that
\begin{equation}\label{e7.23}
    \Psi(\{e,f,g\})=\{e,f,g\}.
\end{equation}
It is easy to show that the sets
$$
\Psi(\{e,f,g\})\cap(\overrightarrow{o\vphantom{b}a} \setminus \{o\}), \ \Psi(\{e,f,g\})\cap(\overrightarrow{o\vphantom{b}b} \setminus \{o\}) \text{ and } \Psi(\{e,f,g\})\cap(\overrightarrow{o\vphantom{b}c} \setminus \{o\})
$$
are non--empty.
\begin{figure}[h]
\begin{center}
\includegraphics[width=0.4\linewidth]{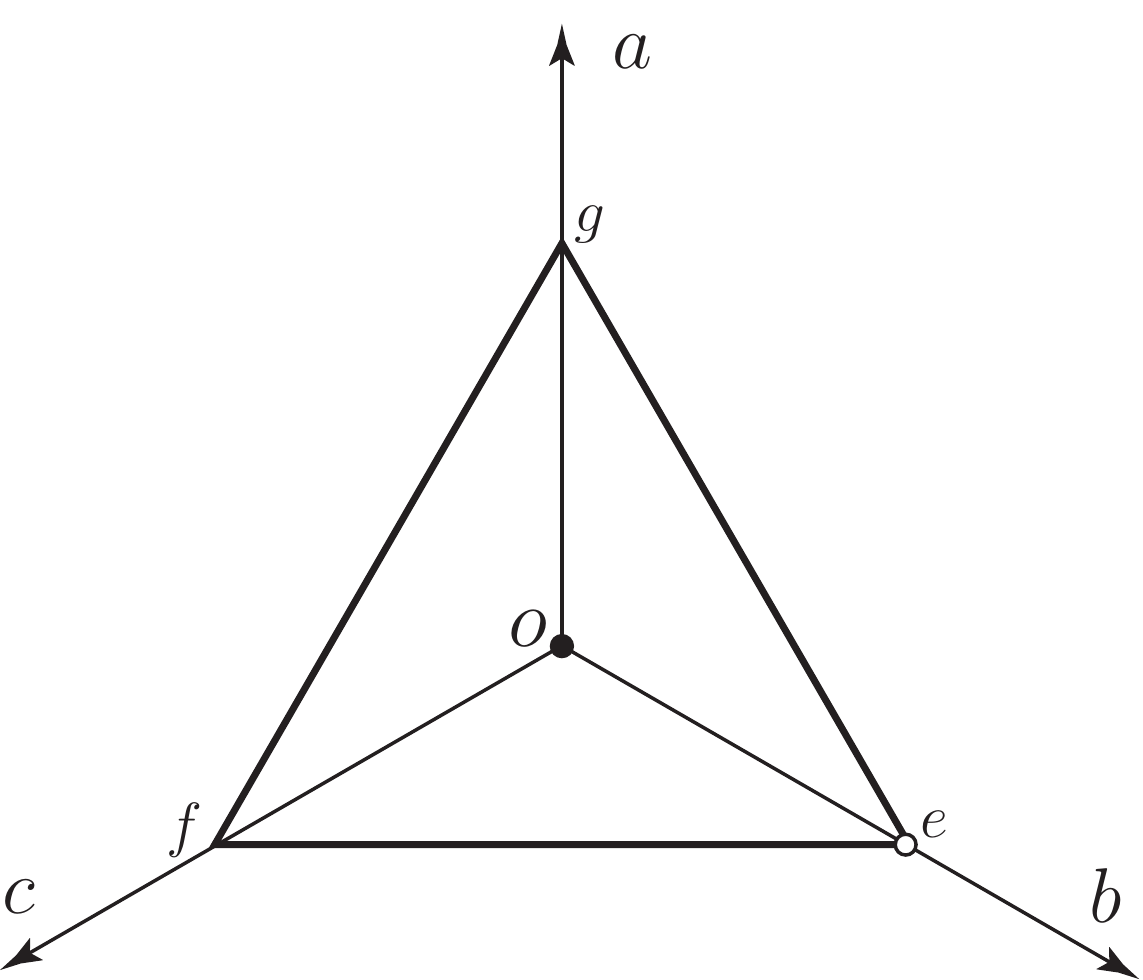}
\caption{} \label{r3}
\end{center}
\end{figure}
Moreover, we have the equalities
$$
\angle \Psi(e)o\Psi(f)=\angle \Psi(f)o\Psi(g)=\angle \Psi(g)o\Psi(e)=\frac{2\pi}{3}.
$$
Since $\{\Psi(e),\Psi(f),\Psi(g)\}$ is an equilateral triangle, Lemma~\ref{l7.1} implies that $o$ is the incenter of $\{\Psi(e),\Psi(f),\Psi(g)\}$, i.e. the point of the intersection of the interior angle bisectors of this triangle. Consequently we have
\begin{equation}\label{e7.3}
d_{\mathbb R^2}(o,\Psi(f))=d_{\mathbb R^2}(o,\Psi(g))=d_{\mathbb R^2}(o,\Psi(e)).
\end{equation}
In view that $\{e,f,g\}$ is equilateral and that
$$
\{e,f,g\}\simeq \{\Psi(e), \Psi(f), \Psi(g)\},
$$
(\ref{e7.23}) follows from~(\ref{e7.1}) and~(\ref{e7.3}).

Using Lemma~\ref{l7.1} we see also that $\{e,f,g\}\not\hookrightarrow X\setminus \{o\}$ if $\{e,f,g\}$ is an triangle with an angle equal to $\frac{2\pi}{3}$. This finishes the proof.
\end{proof}

\begin{thm}\label{t7.3}
Let $X_{\alpha}\subseteq \mathbb R^2$ be a set consisting of two rays $\overrightarrow{oa\vphantom{b}}$ and $\overrightarrow{ob}$ without the point $o$ and let $\alpha$ be the smaller angle between these rays. Then the metric space $X_{\alpha}$ is $\mathfrak{F}_3$-universal if and only if $0 < \alpha <\frac{\pi}{3}$. The space $X_{\alpha}$ is minimal $\mathfrak{F}_3$-universal if and only if $\frac{\pi}{5}\leqslant\alpha<\frac{\pi}{3}$.
\end{thm}
\begin{proof}
It is easy to see that in  the case when $\frac{\pi}{3}\leqslant \alpha \leqslant \pi$ the inscribing of any equilateral triangle in $X_{\alpha}$ is impossible.

Let us prove the $\mathfrak{F}_3$-universality of $X_\alpha$ in the case when $0 < \alpha <\frac{\pi}{3}$. For the triangles $\{e,f,g\}\in \mathfrak{MB}$ it is trivial. Now let $\{e,f,g\}\notin \mathfrak{MB}$. Without loss of generality suppose $\angle feg\geqslant \frac{\pi}{3}$. Write $\beta=\angle feg$ for short. Let us locate $\{e,f,g\}$ in the way depicted on Figure~\ref{f11}.
\begin{figure}[h]
\begin{center}
\begin{minipage}[h]{0.45\linewidth}
\includegraphics[width=1\linewidth]{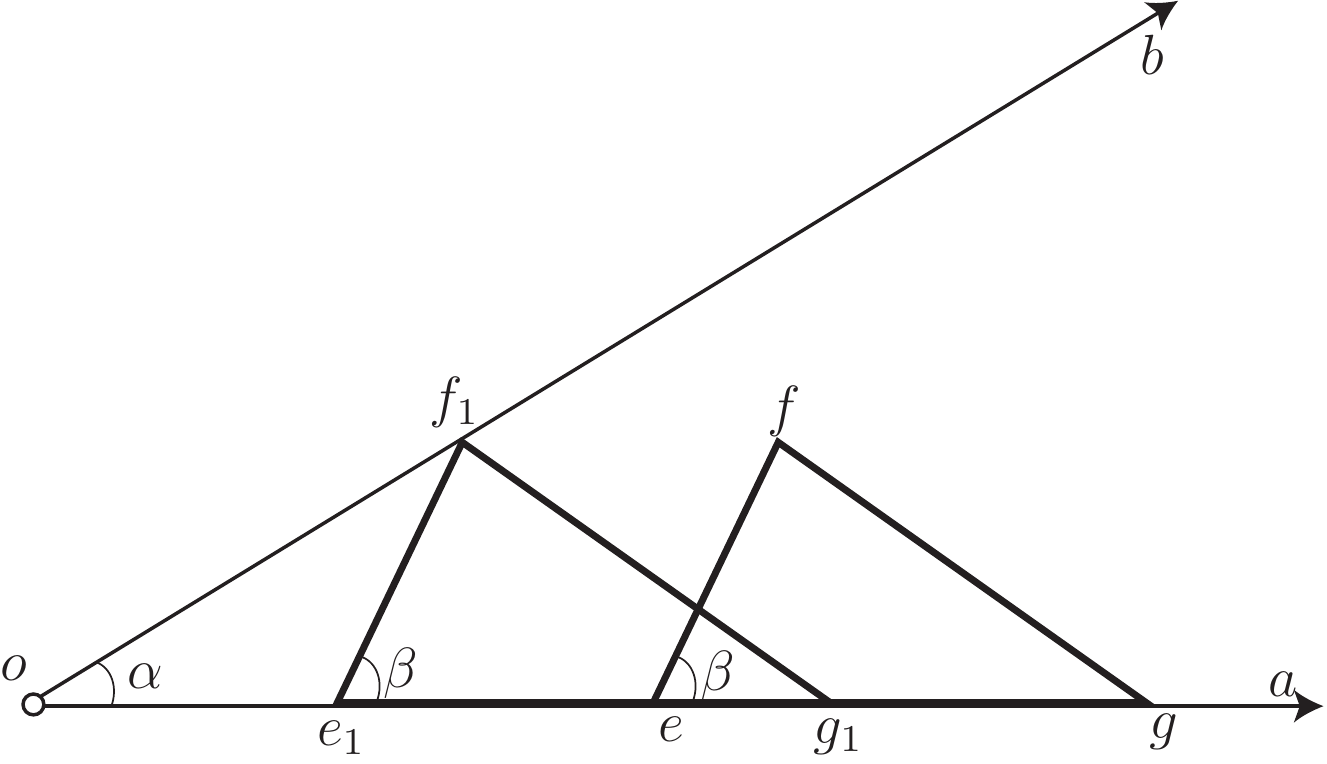}
\caption{If $0<\alpha<\frac{\pi}{3}$, then $X_{\alpha}$ is $\mathfrak{F}_3$-universal.}
\label{f11}
\end{minipage}
\hfill
\begin{minipage}[h]{0.45\linewidth}
\includegraphics[width=1\linewidth]{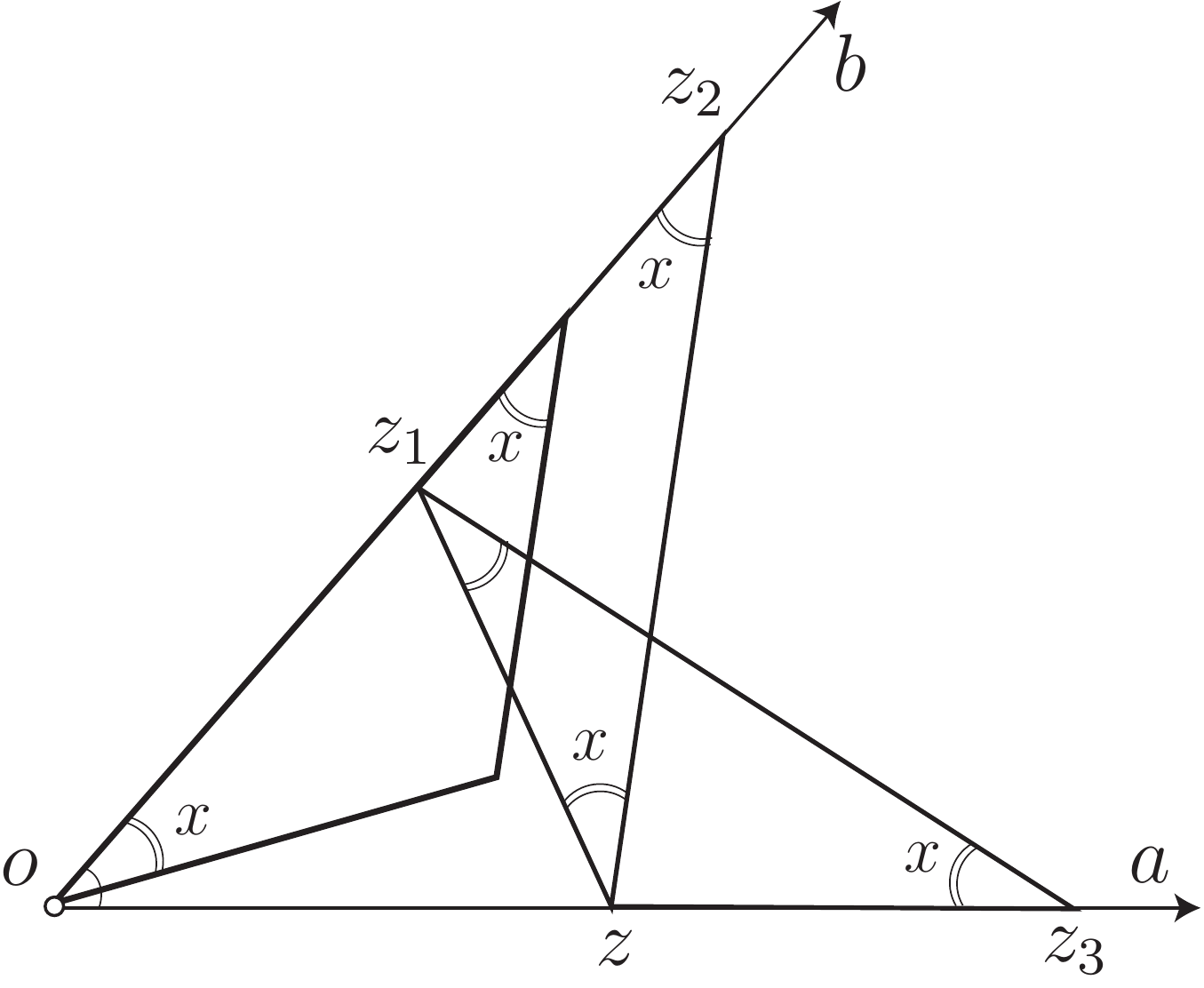}
\caption{If $\frac{\pi}{5}\leqslant\alpha<\frac{\pi}{3}$, then $X_{\alpha}$ is minimal.}
\label{f22}
\end{minipage}
\end{center}
\end{figure}
Doing a parallel shift of the $\{e,f,g\}$ along the ray $\overrightarrow{oa\vphantom{b}}$ we can find a position of its image $\{e_1, f_1, g_1\}$ with $e_1\in \overrightarrow{ob}$. Moreover, we have $e_1\neq o$ because $\beta>\alpha$. The desirable embedding is obtained.

Let us prove that  the space $X_{\alpha}$ is minimal $\mathfrak{F}_3$-universal in the case $\frac{\pi}{5}\leqslant \alpha <\frac{\pi}{3}$. It suffices to show that after deleting any point $z$ from the set $X_{\alpha}$ there exists a triangle which is not embeddable into the set
$$
X_{\alpha}\setminus \{z\}.
$$
We can suppose, without loss of generality, that $z\in \overrightarrow{oa\vphantom{b}}$ (see Figure~\ref{f22}). Let us choose $z_1\in \overrightarrow{ob}$ and $z_2\in \overrightarrow{ob}$ such that
$$
d_{\mathbb{R}^2}(o,z)=d_{\mathbb{R}^2}(o,z_1) \text{ and } d_{\mathbb{R}^2}(z,z_1)=d_{\mathbb{R}^2}(z,z_2).
$$
We claim that $\{z,z_1,z_2\}\not\hookrightarrow X_{\alpha}\setminus \{z\}$. Let $z_3\in \overrightarrow{oa\vphantom{b}}$ with $d_{\mathbb{R}^2}(z,z_1) = d_{\mathbb{R}^2}(z,z_3)$. It is easy to see that $\{z,z_1,z_2\}\simeq\{z,z_1,z_3\}$. Let us show that there is no any another triangles in $X_{\alpha}$ which are isometric to $\{z,z_1,z_2\}$. For the angle $x$, we have the equality $x=\frac{\pi}{4}-\frac{\alpha}{4}$ (see Figure~\ref{f22}). The double inequality  $\frac{\pi}{5}\leqslant \alpha <\frac{\pi}{3}$ implies $\frac{\pi}{6}<x\leqslant \frac{\pi}{5}$, i.e. $x\leqslant\alpha$. A simple geometric reasoning shows that in virtue of the inequality $x\leqslant\alpha$ the  inscribing of $\{z,z_1,z_2\}$ into the set $X_{\alpha}$ provided that the angles $x$ and $\alpha$ have the common rays $\overrightarrow{oa\vphantom{b}}$ or $\overrightarrow{ob}$ which is impossible (see Figure~\ref{f22}).

\begin{figure}[h]
\begin{center}
\includegraphics[width=0.6\linewidth]{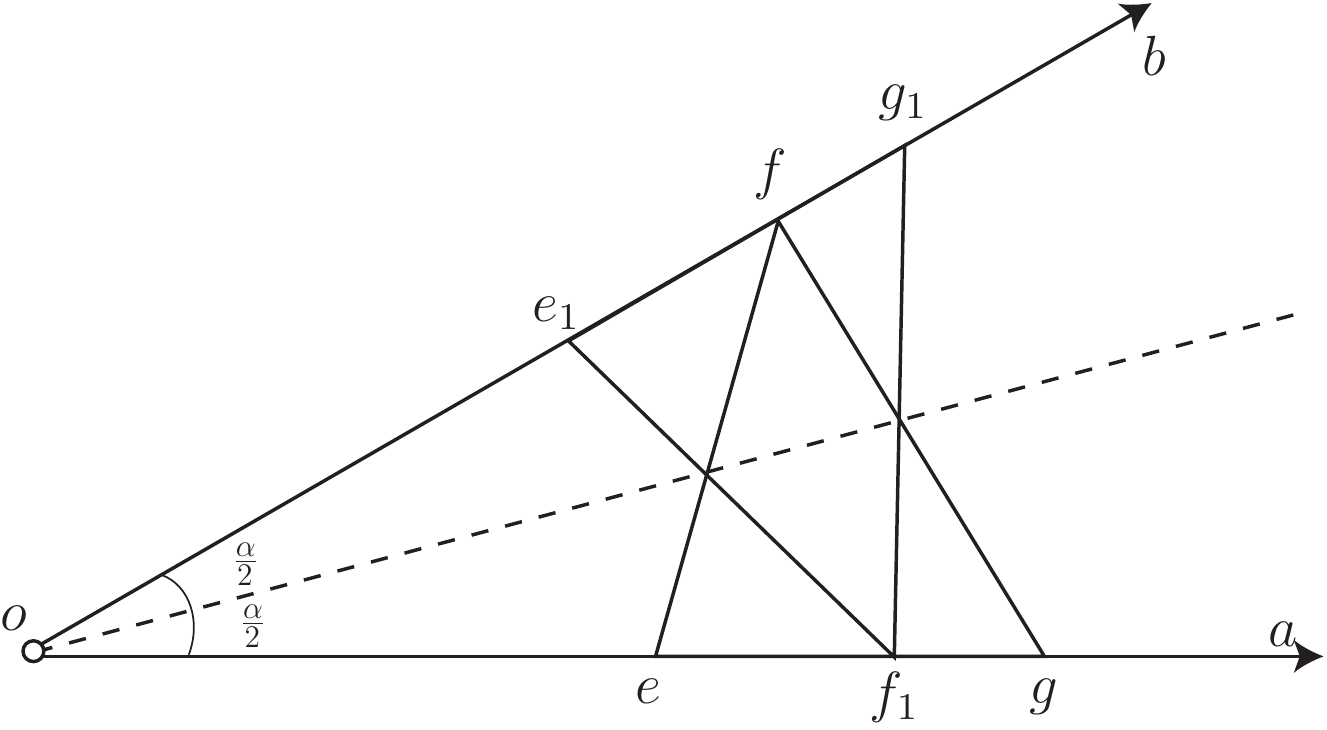}
\caption{} 
\label{f33} 
\end{center}
\end{figure}

Let us prove that $X_{\alpha}$ is not minimal if $0<\alpha<\frac{\pi}{5}$.  It suffices to show that for every point $z\in X_{\alpha}$ and every  triangle $\{e,f,g\}$ there are some embeddings $\{e,f,g\}\hookrightarrow X_{\alpha}\setminus \{z\}$.

Consider first the case when $\{e,f,g\}$ is not isosceles. Without loss of generality assume $z\in \overrightarrow{oa\vphantom{b}}$. Let $\Psi\colon \{e,f,g\}\hookrightarrow X_{\alpha}$ be an isometric embedding with $\Psi (e)=z$ (see Figure~\ref{f33}).
\begin{figure}[h]
\begin{center}
\includegraphics[width=0.6\linewidth]{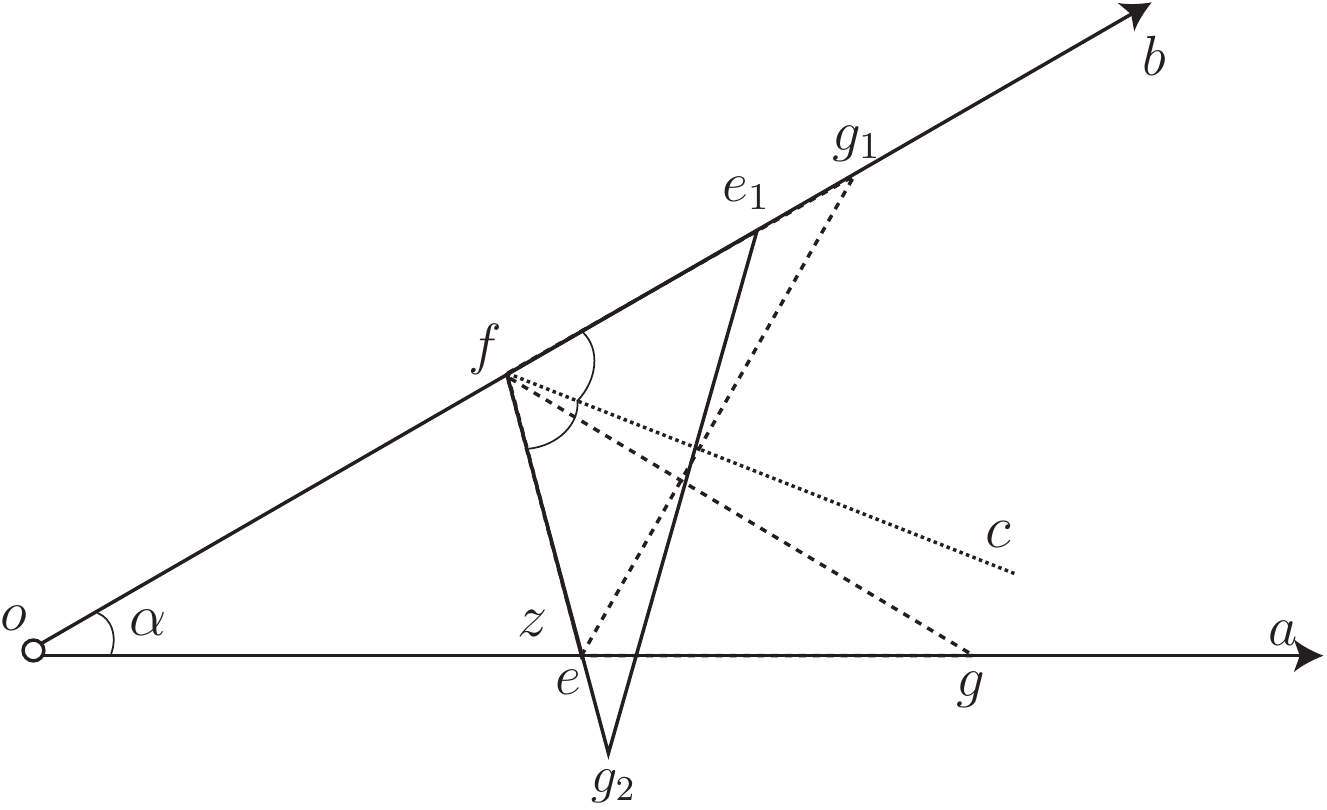}
\caption{}
\label{f4}
\end{center}
\end{figure}
Let us construct a new triangle $\{e_1,f_1,g_1\}$ which is symmetric to $\{\Psi(e),\Psi(f),\Psi(g)\}$ with respect to the bisector of the angle $\alpha$ (see Figure~\ref{f33}). If $d_{\mathbb R^2}(o,\Psi(e))\neq d_{\mathbb R^2}(o,\Psi(f))$, then $f_1$ does not coincide with $z$ so that we have found the desirable embedding. Otherwise, let us consider the triangle $\{\Psi(f),g_1,g_2\}$ which is the reflection of $\{\Psi(e), \Psi(f), g_1\}$ w.r.t. the bisector of the angle $\angle \Psi(e)\Psi(f)g_1$ (see Figure~\ref{f4}). Here we have $d_{\mathbb R^2}(\Psi(e),\Psi(f))\neq d_{\mathbb R^2}(\Psi(f),g_2)$ because $\{e,f,g\}$ is not isosceles. Doing a parallel shift of the $\{\Psi(f),g_1,g_2\}$ along the ray $\overrightarrow{ob}$ one can find a position of its image such that the point $g_2$ belongs to the ray $\overrightarrow{oa\vphantom{b}}$. The  desirable embedding is obtained.

Consider the case when $\{e,f,g\}$ is isosceles. Let $z\in X_{\alpha}$ and let
$$
\Psi \colon \{e,f,g\}\hookrightarrow X_{\alpha}  \ \text{ with }  \ z\in \{\Psi(e),\Psi(f),\Psi(g)\}.
$$
We can clearly assume $z\in \overrightarrow{oa\vphantom{b}}$. We are interested only in the cases when two of the vertices of $\{\Psi(e),\Psi(f),\Psi(g)\}$ are symmetric with respect to the bisector of the angle $\alpha$ and one of them coincides with $z$. Otherwise one can find a desirable embedding as above.
\begin{figure}[h]
\begin{center}
\begin{minipage}[h]{0.45\linewidth}
\includegraphics[width=1\linewidth]{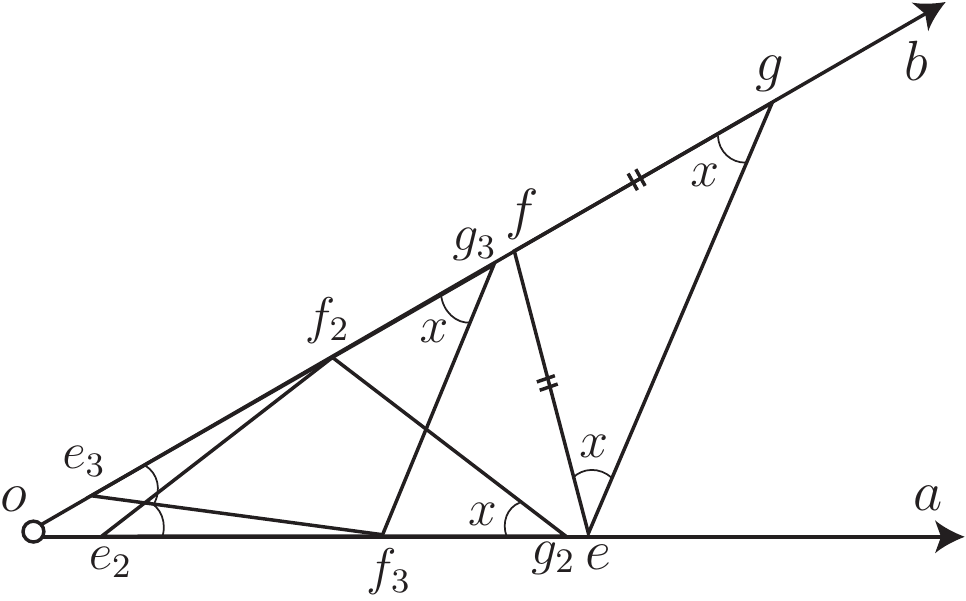}
\caption{Case (i): $x=\frac{\pi}{4} -\frac{\alpha}{4}$.} 
\label{f77} 
\end{minipage}
\hfill
\begin{minipage}[h]{0.45\linewidth}
\includegraphics[width=1\linewidth]{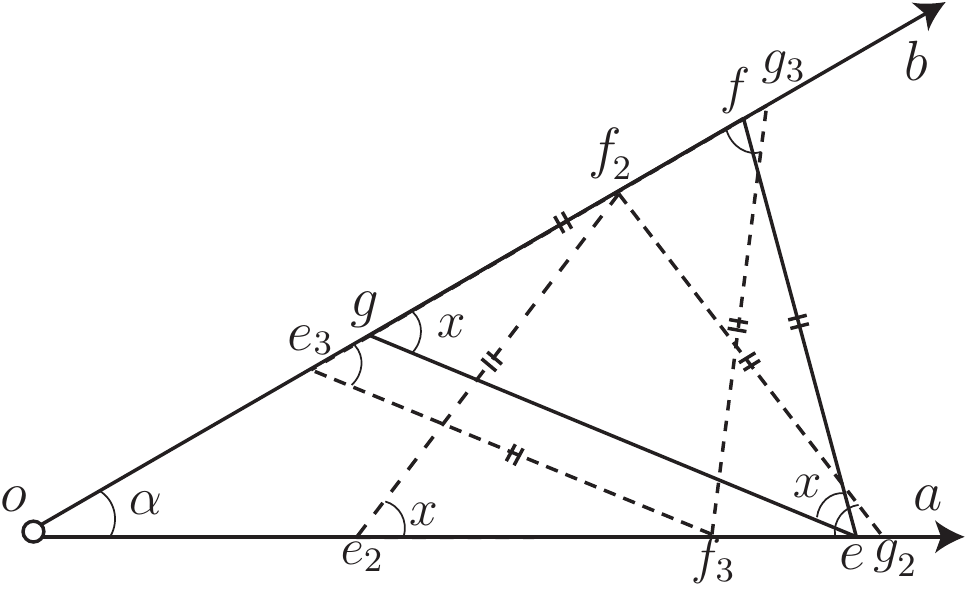}
\caption{Case (ii): $x=\frac{\pi}{4} +\frac{\alpha}{4}$.}
\label{f8}
\end{minipage}
\end{center}
\end{figure}
Let us show that in the case under consideration there is $\Psi\colon \{e,f,g\}\hookrightarrow X_{\alpha}$ with $z\notin \{\Psi(e),\Psi(f),\Psi(g)\}$. Further, without loss of generality, we assume that $\Psi(e)$ and $\Psi(f)$ are symmetric w.r.t. the bisector of $\alpha$. There exists three cases:
\begin{itemize}
\item [(i)] $\Psi(g) \in \overrightarrow{ob}$ and $d_{\mathbb R^2}(o,\Psi(g))> d_{\mathbb R^2}(o,\Psi(f))$ (see Figure~\ref{f77});
\item [(ii)] $\Psi(g) \in \overrightarrow{ob}$ and $d_{\mathbb R^2}(o,\Psi(g))< d_{\mathbb R^2}(o,\Psi(f))$ (see Figure~\ref{f8});
\item [(iii)] $\Psi(g) \in \overrightarrow{oa\vphantom{b}}$ and $d_{\mathbb R^2}(o,\Psi(g))< d_{\mathbb R^2}(o,\Psi(e))$ (see Figure~\ref{f9}).
\end{itemize}

Consider case (i). Write $x=\angle \Psi(f)\Psi(e)\Psi(g)=\angle \Psi(f)\Psi(g)\Psi(e)$. A simple geometric reasoning  gives us $x=\frac{\pi}{4}-\frac{\alpha}{4}$. If $0<\alpha<\frac{\pi}{6}$, then $\frac{\pi}{6}<x<\frac{\pi}{4}$. Hence $x>\alpha$. In virtue of the last inequality there exist $\{e_2,f_2,g_2\}\subseteq X_{\alpha}$ and $\{e_3,f_3,g_3\}\subseteq X_{\alpha}$ such that
$$
\{e_2,f_2,g_2\}\cap \{e_3,f_3,g_3\}=\varnothing \ \text{ and } \ \{e_2,f_2,g_2\}\simeq \{e,f,g\}\simeq \{e_3,f_3,g_3\}
$$
(see Figure~\ref{f77}). Consequently we have $z\notin \{e_2,f_2,g_2\}$ or $z\notin \{e_2,f_3,g_3\}$. The existence of $\Psi \colon \{e,f,g\}\hookrightarrow X_{\alpha}\setminus \{z\}$ follows.

\begin{figure}[h]
\begin{center}
\includegraphics[width=0.5\linewidth]{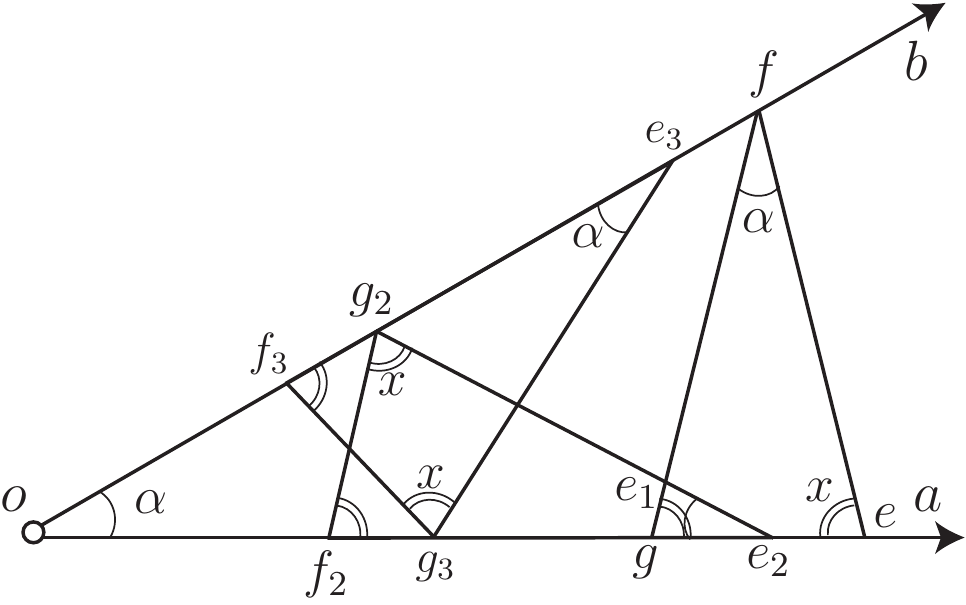}
\end{center}
\caption{Case (iii): $x=\frac{\pi}{2} -\frac{\alpha}{2}$.} 
\label{f9} 
\end{figure}

Cases (ii) and (iii) are similar.

\end{proof}


\end{document}